\documentclass[letterpaper,11pt]{article}
\usepackage[utf8x]{inputenc}
\usepackage[title]{appendix}
\usepackage[noTeX]{mmap}
\usepackage[left=1.3in, right=1.3in, bottom = 1.4in, top = 1.4in]{geometry}
\usepackage{hyperref} 
\hypersetup{colorlinks}
\usepackage{amsmath,amsfonts, amsthm, amssymb}
\usepackage{color}
\usepackage{enumerate}
\usepackage{hypbmsec}
\usepackage{bbm}
\usepackage{dsfont}
\usepackage{mathtools}
\usepackage{graphicx}
\usepackage{caption}
\usepackage{subcaption}
\usepackage[section]{placeins}
\usepackage{tikz}
\usetikzlibrary{calc}
\usetikzlibrary{decorations.pathreplacing}
\usetikzlibrary{arrows}
\usepackage{pgffor}
\usepackage{longtable}
\usepackage{comment}
\usepackage{multirow}
\usepackage{rotating}
\usepackage{bm}

\usepackage{fancyhdr}
\usepackage{tocloft}

\theoremstyle{plain}
\newtheorem{theorem}{Theorem}[section]

\newtheorem{corollary}[theorem]{Corollary}
\newtheorem{lemma}[theorem]{Lemma}
\newtheorem{definition}[theorem]{Definition}

\theoremstyle{definition}

\newtheorem{notation}[theorem]{Notation}

\newcommand{\R}{\mathcal{R}}
\newcommand{\Z}{\mathbb{Z}}
\newcommand{\N}{\mathbb{N}}

\newcommand{\gw}{\mathbf{\mathsf{GW}}}
\newcommand{\rgw}{\mathbf{\mathsf{RGW}}}
\newcommand{\rc}{\mathbf{\mathsf{R}}}

\renewcommand{\d}{\mathrm{d}}

\newcommand{\E}{\mathbb{E}}

\newcommand{\Pb}{\mathbb{P}}

\newcommand{\T}{\mathcal{T}}
\newcommand{\TT}{\mathbb{T}}

\newcommand{\pref}[1]{\hyperref[#1]{[p. \pageref{#1}]}}

\newcommand{\old}[1]{}

\title{Range and speed of rotor walks on trees}

\author{Wilfried Huss and Ecaterina Sava-Huss}

\begin{document}
\maketitle

\begin{abstract}
We prove a law of large numbers for the range of rotor walks with random initial configuration on regular trees and on Galton-Watson trees. We also show the existence of the \emph{speed} for such rotor walks. More precisely, we show that on the classes of trees under consideration, even in the case when the rotor walk is recurrent, the range grows at linear speed.
\end{abstract}



\begin{small}
\tableofcontents
\end{small}

\section{Introduction}\label{sec:intro}

For $d\geq 2$ let $\TT_d$ be the rooted regular tree of degree $d+1$, and denote by $r$ the root. We attach an additional \emph{ sink vertex} $o$ to the root $r$.  We use the notation $\widetilde{\TT}_d = \TT_d\setminus\{o\}$ to denote the tree without the sink vertex. For each vertex
$v\in \widetilde{\TT}_d$ we denote its neighbors by $v^{(0)}, v^{(1)},\ldots,v^{(d)}$,
where $v^{(0)}$ is the parent of $v$ and the other $d$ neighbors, the children of $v$, are ordered counterclockwise.

\begin{figure}[h]
\centering
\includegraphics[width = 0.5\linewidth]{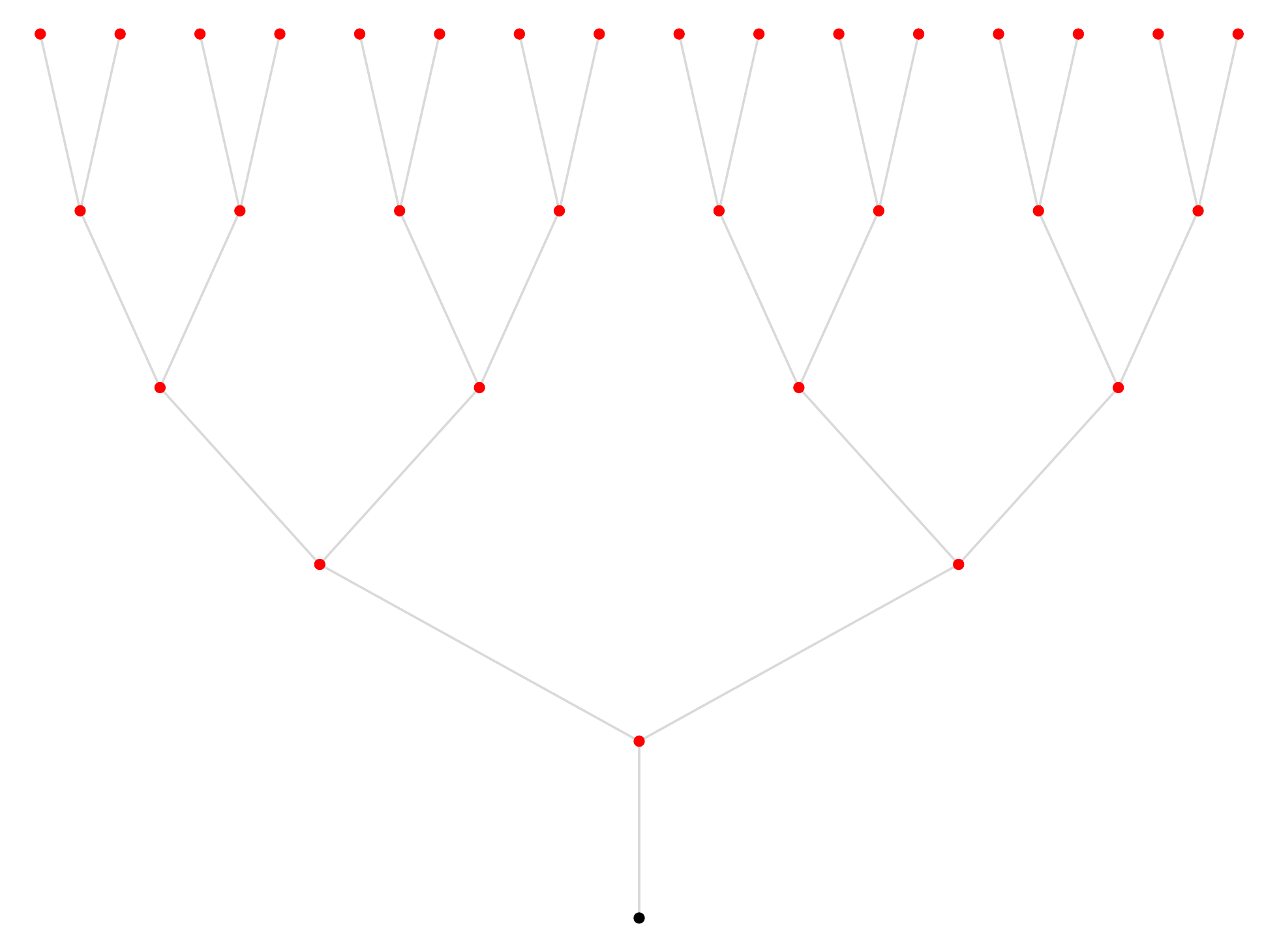}
\caption{The binary tree $\TT_2$.}
\end{figure}

Each vertex $v\in \widetilde{\TT}_d$ is endowed with a rotor $\rho(v) \in \{0,\ldots,d\}$, where $\rho(v)=j$, for $j\in\{0,\ldots,d\}$ means that the rotor in $v$ points to neighbor $v^{(j)}$. 
Let $(X_n)_{n\in\mathbb{N}}$ be a rotor walk on $\TT_d$ starting in $r$ with initial rotor configuration $\rho=(\rho(v))_{v\in \widetilde{\TT}_d}$: for all  $v \in \widetilde{\TT}_d$, let $\rho(v)\in \{0,\dots,d\}$ be independent and identically distributed random variables, with distribution given by $\Pb[\rho(v) = j] = r_j$ with $\sum_{j=0}^dr_j=1$. The rotor walk moves in this way:
at time $n$, if the walker is at vertex $v$, then it first 
rotates the rotor to point to the next neighbor in the counterclockwise order and then it moves to that vertex, that is $X_{n+1}=v^{(\rho(v)+1)\mod (d+1)}$. If the initial rotor configuration is random, then once a vertex has been visited for the first time, the configuration there is fixed.
A child $v^{(j)}$ of a vertex $v\in \TT_d$ is called \emph{good} if $\rho(v)<j$, which means that the rotor walk will first visit the good children before visiting the parent $v^{(0)}$ of $v$. Remark that $v$ has $d-\rho(v)$ good children. The 
\emph{tree of good children} for the rotor walk $(X_n)$, which we denote $\T_d^{\mathsf{good}}$, is a subtree of $\TT_d$, where all the vertices are good children. Let us
denote by $R_n = \{X_0, \ldots, X_n\}$ the range on $\widetilde{\TT}_d = \TT_d\setminus\{o\}$ of the rotor walk $(X_n)$ up to time $n$, that is, the set of distinct visited points by the rotor walk $(X_n)$ up to time $n$, excluding the sink vertex $o$. 
Its cardinality, denoted by $|R_n|$ represents then the number of distinct visited points by the walker up to time $n$. 

We denote by $d(r,X_n):=|X_n|$ the distance from the position $X_n$ at time $n$ of the rotor walker to the root $r$. 
The \emph{speed} or the \emph{rate of escape} of the rotor walk $(X_n)$ is the almost sure limit (if it exists) of $\frac{|X_n|}{n}$.  We say that $|R_n|$ satisfies a law of large numbers if 
$\frac{|R_n|}{n}$ converges almost surely to a constant. The aim of this work is to prove a law of large numbers for $|X_n|$ and $|R_n|$, that is, to find constants $l$ and $\alpha$ such that

\begin{align*}
\lim_{n\to \infty}\frac{|X_n|}{n} & =l,\quad \text{almost surely},\\
\lim_{n\to \infty}\frac{|R_n|}{n}& =\alpha,\quad \text{almost surely},
\end{align*}
when $(X_n)$ is a rotor walk with random initial rotor configuration on a regular tree and on a Galton-Watson tree, respectively. On regular trees, these constants depend on whether the rotor walk $(X_n)$ is recurrent or transient on $\TT_d$, a property which depends only on the expected value $\E[\rho(v)]$ of the rotor configuration at vertex $v$, as shown in \cite[Theorem 6]{angel_holroyd}: if $\E[\rho(v)]\geq d-1$, then the rotor walk $(X_n)$ is recurrent, and if 
$\E[\rho(v)]< d-1$ then it is transient. Since in the case $\E[\rho(v)]=d-1$, the expected return time to the root is infinite, we shall call this case a \emph{critical case}, and we say that the rotor walk is \emph{null recurrent}. Otherwise, if $\E[\rho(v)]>d-1$, we say that $(X_n)$ is \emph{positive recurrent}. The tree $\T^{\mathsf{good}}_d$ of good children for the rotor walk is a Galton-Watson tree with mean offspring number $d-\E[\rho(v)]$ and generating function $f(s)=\sum_{j=0}^{d}r_{d-j}s^j$. 

The main results of this paper can be summarized into the two following theorems.

\begin{theorem}[Range of the rotor walk]
\label{main-thm-range}
If $(X_n)_{n\in\N}$ is a rotor walk with random initial configuration of rotors on $\TT_d$, $d\geq 2$, then there exists a constant $\alpha>0$, such that
\begin{equation*}
\lim_{n\to \infty}\frac{|R_n|}{n} =\alpha,\quad \text{almost surely}. 
\end{equation*}
The constant $\alpha$ depends only on $d$ and on the distribution of $\rho$ and is given by:
\begin{enumerate}[(i)]
 \setlength\itemsep{-1em}
\item If $(X_n)_{n\in\N}$ is positive recurrent, then
\begin{equation*}
\alpha=\frac{d-1}{2\mathbb{E}[\rho(v)]}.
\end{equation*}
\item If $(X_n)_{n\in\N}$ is null recurrent, then 
\begin{equation*}
\alpha=\frac{1}{2}.
\end{equation*}
\item If $(X_n)_{n\in\N}$ is transient, then conditioned on the non-extinction of $\T^{\mathsf{good}}_d$,
\begin{equation*}
\alpha=\frac{q-f'(q)(q^2-q+1)}{q^2+q-f'(q)(2q^2-q+1)} 
\end{equation*}
where $q>0$ is the extinction probability of $\T^{\mathsf{good}}_d$.
\end{enumerate}
\end{theorem}
Remark that, even in the recurrent case, the range of the rotor walk grows at linear speed, which is not the case for simple random walks on regular trees.
The methods of proving the above result are completely different for the transient and for the recurrent case, and the proofs will be done in separate sections.

\begin{theorem}[Speed of the rotor walk]
\label{main-thm-escape}
If $(X_n)_{n\in\N}$ is a rotor walk with random initial configuration of rotors on $\TT_d$, $d\geq 2$, then there exists a constant $l\geq 0$, such that
\begin{equation*}
\lim_{n\to \infty}\frac{|X_n|}{n} =l,\quad \text{almost surely}. 
\end{equation*}
\begin{enumerate}[(i)]
\item If $(X_n)_{n\in\N}$ is  recurrent, then $l=0$.
\item If $(X_n)_{n\in\N}$ is transient, then conditioned on the non-extinction of $\T^{\mathsf{good}}_d$,
\begin{equation*}
l=\frac{(q-f'(q))(1-q)}{q+q^2-f'(q)(2q^2-q+1)},
\end{equation*}
where $q>0$ is the extinction probability of $\T^{\mathsf{good}}_d$.
\end{enumerate}
\end{theorem}
The constant $l$ is in the following relation with the constant $\alpha$ from Theorem \ref{main-thm-range} in the transient and null recurrent case:
\begin{equation}\label{eq:einstein}
2\alpha-l=1.
\end{equation}
We call this equation the \emph{Einstein relation for rotor walks}.
We state similar results for rotor walks on Galton-Watson trees $\T$ with random initial configuration of rotors, and we show that in this case, the constants $\alpha$ and $l$ depend only on the distribution of the configuration $\rho$ and on the offspring distribution of $\T$.

Range of rotor walks and its shape was considered also in \cite{florescu_levine_peres} on comb lattices and on Eulerian graphs. On combs, it is proven that the size of the range $|R_n|$ is of order $n^{2/3}$, and its asymptotic shape is a diamond. It is conjectured in \cite{KD}, that on $\Z^2$, the range of uniform rotor walks is asymptotically a disk, and its size is of order $n^{2/3}$. In the recent paper \cite{swee-hong-wsf}, for special cases of initial configuration of rotors on transient and vertex-transitive graphs, it is shown that the occupation rate of the rotor walk is close to the Green function of the random walk.
 
\paragraph*{Organization of the paper.} We start by recalling some basic facts and definitions about rotor walks and Galton-Watson trees in Section \ref{sec:prelim}. Then in Section \ref{sec:range} we prove Theorem \ref{main-thm-range}, while in Section \ref{sec:rate-escape} we prove Theorem \ref{main-thm-escape}. Then we prove in Theorem \ref{main-thm-range-GW} and in Theorem \ref{main-thm-escape-GW} a law of large numbers for the range and the existence of the speed, respectively for rotor walks on Galton-Watson trees.
Finally in Appendix \ref{sec:contour} we look at the contour function of the range of recurrent rotor walks and its recursive decomposition. 

\section{Preliminaries}\label{sec:prelim}

\subsection{Rotor walks}\label{sec:rotor-walks}

Let $\TT_d$ be the regular infinite rooted tree with degree $d+1$, with root $r$, and an additional vertex $o$ which is connected to the root and is called the \emph{sink}, and let  $\widetilde{\TT}_d = \TT_d\setminus\{o\}$. Every vertex $v\in\widetilde{\TT}_d$ has $d$ children and one parent. For any connected subset $V\subset \widetilde{\TT}_d$, define the set of leaves in $\widetilde{\TT}_d$ as
$\partial_o V = \{v \in \widetilde{\TT}_d\setminus V: \exists u\in V \text{ s. t. } u\sim v\}$ as the set of vertices outside of $V$ that are children of vertices of $V$, that is, $\partial_o V$ is the outer boundary of $V$.  On $\widetilde{\TT}$, the size of $\partial_o V$ depends only on the size of $V$:
\begin{equation}
\label{eq:number_of_leaves}
|\partial_o V| = 1+(d-1)| V|.
\end{equation}
A \emph{rotor configuration} $\rho$ on $\widetilde{\TT}_d$ is a function $\rho: \widetilde{\TT}_d\to \mathbb{N}_0$, with $\rho(x)\in \{0,\ldots,d\}$, which can be interpreted as following: each vertex $v\in\widetilde{\TT}_d$ is endowed with a rotor $\rho(v)$ (or an arrow) which points to one of the $d+1$ neighbors. We fix from the beginning a \emph{counterclockwise ordering of the neighbors} $v^{(0)},v^{(1)},\ldots,v^{(d)}$, which represents the order in which the neighbors of a vertex are visited, where $v^{(0)}$ is the parent of $v$, and $v^{(1)},\ldots,v^{(d)}$ are the children. A rotor walk $(X_n)$ on $\TT_d$ is a process where at each time step $n$, a walker located at some vertex $v\in \widetilde{\TT}_d$ first increments the rotor at $v$, i.e. it changes its direction to the next neighbor in the counterclockwise order, and then the walker moves there. We start all our rotor walks at the root $r$, $X_0=r$, with initial rotor configuration $\rho_0=\rho$. Then $X_n$ represents the position of the rotor walk at time $n$, and $\rho_n$ the rotor configuration at time $n$.  The rotor walk is also used as \emph{rotor-router walk} in the literature. At each time step, we record not only the position of the walker, but also the configuration of rotors, which changes only at the current position. More precisely, if at time $n$ the pair of position and configuration is $(X_n,\rho_n)$, then at time $n+1$ we have 
\begin{equation*}
\rho_{n+1}(x)=
\begin{cases}
\rho_n(x)+1 \mod (d+1)&, \text{if } x= X_n\\
\rho_n(x) &, \text{otherwise}.
\end{cases}
\end{equation*}
and $X_{n+1}=X_n^{(\rho_{n+1}(X_n))}$.
As defined above, $(X_n)$ is a deterministic process once $\rho_0$ is determined. Throughout this paper we are interested in rotor walks $(X_n)$ which start with a \emph{random initial configuration} $\rho_0$ of rotors, which makes $(X_n)$ a random process that is not a Markov chain.

\paragraph{Random initial configuration.} For the rest of the paper we consider $\rho$ a random initial configuration on $\widetilde{\TT}_d$, in which $(\rho(v))_{v\in \widetilde{\TT}_d}$ are independent random variables  with distribution on $\{0,1,\ldots,d\}$ given by 
\begin{equation}\label{eq:ran-in-cfg}
\mathbb{P}[\rho(v)=j]=r_j,
\end{equation}
with $\sum_{j=0}^{d}r_j=1$.
If $\rho(v)$ is uniformly distributed on the neighbors, then we call the corresponding rotor walk \emph{uniform rotor walk}.
Depending on the distribution of the initial rotor configuration $\rho$, the rotor walk can exhibit one of the following two behaviors: either the walk visits each vertex infinitely often, and it is \emph{recurrent}, or each vertex is visited at most finitely many times, and it escapes to infinity, and this is the \emph{transient case}. For rotor walks on regular trees, the  recurrence-transience behavior was proven in \cite[Theorem 6]{angel_holroyd}, and the proof is based on the extinction/survival of a certain branching process, which will also be used in our results.
Similar results on recurrence and transience of rotor walks on Galton-Watson trees have been proven  in \cite{huss_muller_sava_gw}.

For the rotor configuration $\rho$ on $\widetilde{\TT}_d$, a \emph{live path} is an infinite sequence of vertices $ v_1,v_2,\ldots$ each being the parent of the next, such that for all $i$, the indices $k$ for which $v_{i+1}=v_i^{(k)}$ satisfy $\rho(v)<k$. In other words, $v_1,v_2,\ldots$ is a live path if and only if all $v_1,v_2,\ldots$ are good, and a particle located at $v_i$ will be sent by the rotor walker forward to $v_{i+1}$ before sending it back to the root. An \emph{end} in $\TT_d$ is an infinite sequence of vertices $o=v_0,v_1,\ldots$, each being the parent of the next. 
An end is called \emph{live} if the subsequence $(v_i)_{i\geq j}$ starting at one of the vertices  $o=v_0,v_1,\ldots$ is a live path.
The rotor walk $(X_n)$ can  escape to infinity only via a live path.

\subsection{Galton-Watson trees}

Consider a Galton-Watson process $(Z_n)_{n\in\mathbb{N}_o}$ with offspring distribution $\xi$
given by $p_k=\mathbb{P}[\xi=k]$. We start with one particle $Z_0 = 1$, which has $k$ children with probability $p_k$;
then each of these children independently has children with the same offspring distribution $\xi$, and so on. Then $Z_n$ represents the number of particles in the $n$-th generation.
If $(\xi_{i}^n)_{i,n\in\mathbb{N}}$ are i.i.d. random variables distributed as $\xi$, then 
\begin{equation*}
Z_{n+1}=\sum_{i=1}^{Z_n}\xi_i^n,
\end{equation*}
and $Z_0=1$. Starting with a single progenitor, this process yields a random family tree $\mathcal{T}$, which is called a Galton-Watson tree.
The \emph{mean offspring number} $m$ is defined as the expected number of children of one particle $m=\mathbb{E}[\xi]$. In order to avoid trivialities, we will assume $p_0+p_1<1$. The generating function of the process is the function $f(s)=\sum_{k=0}^{\infty}p_ks^k$ and $m=f'(1)$. If  is well known that the \emph{extinction probability} of the process, defined as $q=\lim_{n\to\infty}\mathbb{P}[Z_n=0]$, which is the probability the process ever dies out, has the following important property.
\begin{theorem}[Theorem 1, page 7, in \cite{athreya-ney}] 
\label{thm:GW-survival}
The extinction probability of $(Z_n)$ is the smallest nonnegative root of $s=f(s)$. It is $1$ if $m\leq 1$ and $<1$ if $m>1$.
\end{theorem}
With probability $1$, we have $Z_n\to 0$ or $Z_n\to \infty$ and  $\lim_n\mathbb{P}[Z_n=0]=1-\lim_n \mathbb{P}[Z_n=\infty]=q$. For more information on Galton-Watson processes, we refer to  \cite{athreya-ney}. When $m<1$, $=1$, or  $>1$, we shall refer to the  Galton-Watson tree as \emph{subcritical, critical, or supercritical}, respectively. 


\section{Range on regular trees}\label{sec:range}

For a simple random walk on a regular tree $\TT_d$, $d\geq 2$, which is transient, if we denote by $S_n$ its range, then it is known \cite[Theorem 1.2]{chen_shijian_zhou} that $|S_n|$ satisfies a law of large numbers:
\begin{align}
\lim_{n\to\infty}\frac{|S_n|}{n} = \frac{d-1}{d}\qquad \text{almost surely}.
\end{align}
We prove a similar result for the range of any rotor walk with random initial configuration on a regular tree 
$\TT_d$.  From
 \cite[Theorem 6]{angel_holroyd}, $(X_n)$ is recurrent if
 $\E[\rho(v)] = \sum_{j=1}^d j r_j \geq d - 1$ and transient if $\mathbb{E}[\rho(v)]<d-1$. 
The tree of good children for the rotor walk, denoted $\T^{\mathsf{good}}_d$ and defined in Section \ref{sec:rotor-walks}, is then a Galton-Watson process with offspring distribution
$\xi=$ number of good children of a vertex, given by
\begin{equation*}
 \mathbb{P}[\xi=j]=r_{d-j}, \text{ for } j=0,1,\ldots d.
\end{equation*}
Each vertex has, independently of all the others, a number of good children with the same distribution $\xi$.
The mean offspring number of $\T^{\mathsf{good}}_d$ is $m=d-\mathbb{E}[\rho(v)]$. Let  
$f(s)=\sum_jr_{d-j}s^j$ be the generating function for $\T^{\mathsf{good}}_d$.

The lemma below is a key observation that is crucial for the main results of this paper. For a proof, we refer to 
\cite{angel_holroyd}.
\begin{lemma}\label{lem:key-lemma}
Let $\T^{\mathsf{good}}_d$ be the tree of good children of the sink vertex of the current rotor configuration. Then for every excursion (i.e. a rotor walk that is started at the sink vertex and is stopped the first time it returns to the sink vertex), we have the following:
\begin{enumerate}[(a)]
\item If the component of the sink vertex in $\T^{\mathsf{good}}_d$ is finite, then every vertex $v$ in this component will be visited in the excursion exactly $d+1-\mathbb{E}[\rho(v)]$ times.
\item If the component of the sink vertex in $\T^{\mathsf{good}}_d$ is infinite, then the walker will escape through the rightmost live path. Furthermore, every vertex to the right of this live path will be visited exactly 
$d+1-\mathbb{E}[\rho(v)]$ times.

\end{enumerate}
\end{lemma}
For proving Theorem \ref{main-thm-range}, we shall treat the three cases separately: the positive recurrent, null recurrent and transient case.

\subsection{Recurrent rotor walks}
\label{subsec:rec-rot-regtrees}

In this section, we consider recurrent rotor walks $(X_n)$ on $\TT_d$, that is, once again from  \cite[Theorem 6]{angel_holroyd}   $\E[\rho(v)] \geq d - 1$. Then  $\T^{\mathsf{good}}_d$ has mean offspring number $m\leq 1$, which by Theorem \ref{thm:GW-survival} dies out with probability one. While in the case $m<1$, where the rotor walk is positive recurrent, the expected size of $\T^{\mathsf{good}}_d$ is finite, this is not the case when $m>1$. For this reason we  handle these two cases separately.
In order to prove a law of large numbers for the range $R_n= \{X_0,X_1,\ldots, X_n\}$ of the rotor walk up to time $n$, we first look 
at the behavior of the rotor walk at the times when it returns to the sink $o$. 
Define the times $(\tau_k)$ of the $k$-th return to the sink $o$, by: $\tau_0 = 0$ and for $k\geq 1$ let
\begin{equation}\label{eq:tauk-rec}
\tau_k = \inf\{n > \tau_{k-1}: X_n = o\}.
\end{equation}
At time $\tau_{k}$, the walker is at sink, all rotors in the visited set $R_{\tau_k}$
point towards the root, while all other
rotors still are in their initial configuration. Between the two consecutive stopping times $\tau_{k-1}$ and $\tau_k$, the rotor walk performed a depth first search in the finite subtree induced by  $R_{\tau_k}$, by visiting every child
of a vertex in right to left order. For every vertex in $v\in R_{\tau_k}$ we can uniquely associate the edge $(v,v^{(0)})$, with $v^{(0)}$ being the unique ancestor of $v$, which implies that $|R_{\tau_k}|$ equals the number of edges in the tree induced by $R_{\tau_k}$. In a depth first search of $R_{\tau_k}$, each edge is visited exactly two times, and in view of the bijection above, it requires exactly $2|R_{\tau_k}|$  steps to return to the origin. We can then deduce that
\begin{equation}
\label{eq:tau_k_increments}
\tau_{k} - \tau_{k-1} = 2 |R_{\tau_k}|.
\end{equation}
 Since we are in the recurrent case, where the rotor walk returns to the sink infinitely many times, these stopping times are almost surely finite. 

\subsubsection{Positive recurrent rotor walks}

If $\E[\rho(v)] > d - 1$, we prove the following.

\begin{theorem}
\label{thm:range_lln_homtree}
For a positive recurrent rotor walk $(X_n)$ on $\TT_d$, with $d\geq 2$ we have
\begin{align*}
\lim_{k\to\infty}\frac{| R_{\tau_k}|}{\tau_k} = 
\frac{d-1}{2\E[\rho]}, \quad \text{almost surely}.
\end{align*}
\end{theorem}

\begin{proof}
For simplicity of notation, we write $\mathcal{R}_k = R_{\tau_k}$. The tree of good children, $\T^{\mathsf{good}}_d=\mathcal{R}_1$ is a subcritical
Galton-Watson tree with mean offspring number $m = d - \E[\rho]<1$, that is, it dies out almost surely. The expected size of the range up to time $\tau_1$ is given by $\E[|\mathcal{R}_1|] = \frac{1}{1-m}>1$.
At the time $\tau_k$ of the $k$-th return to the sink $o$, all rotors in the previously visited set
$\mathcal{R}_k$ point towards the root, and the remaining rotors are
still in their initial configuration. Thus during the
time interval $(\tau_k, \tau_{k+1}]$, the rotor walk visits
all the leaves of $\R_k$ from right to left, and
at each leaf it attaches independently a (random) subtree that has the same distribution as $\R_1$.
If we denote by $L_k = |\partial_o \R_k|$, then
$L_{k+1} = \sum_{i=1}^{L_k} L_{1,i}$, where $L_{1,i}$
are independent copies of $L_1$, and $(L_k)$ is a
supercritical Galton-Watson process with mean offspring number $\nu$ (the mean number of leaves of $|\R_1|$), which in view of \eqref{eq:number_of_leaves}, is given by
$$\nu = 1+(d-1)\mathbb{E}[|\R_1|]=1 + \frac{d-1}{1-m}>1.$$
Moreover $\Pb[L_k = 0] = 0$.
Since $\nu>1$ it follows from the Seneta-Heyde Theorem (see \cite{lyons-peres-book})  applied to the supercritical Galton-Watson process $(L_k)$ that there
exists a sequence of numbers $(c_k)_{k\geq 1}$ and a nonnegative random variable $W$ such that
\begin{enumerate}[(i)]
\setlength{\itemsep}{-1pt}
\item $\displaystyle\frac{L_k}{c_k} \to W,\quad\text{almost surely.}$
\item $\displaystyle
\Pb[W = 0]$ equals the probability of extinction of $(L_k)$.
\item $\displaystyle \frac{c_{k+1}}{c_k} \to \nu$.
\end{enumerate}
By \eqref{eq:number_of_leaves} we have $|\R_k| = \frac{L_k - 1}{d-1}$ almost surely, which together with equation \eqref{eq:tau_k_increments} yields
\begin{align*}
(d-1)\big(\tau_k - \tau_{k-1}\big) + 2 = 2(d-1) |\R_k| + 2 = 2 L_k,\quad \text{almost surely},
\end{align*}
and dividing by $c_k$ gives
\begin{align*}
(d-1)\left(\frac{\tau_k}{c_k} - \frac{\tau_{k-1}}{c_{k-1}}\cdot\frac{c_{k-1}}{c_k}\right) +\frac{2}{c_k} = 2\frac{L_k}{c_k},\quad \text{almost surely}.
\end{align*}
It then follows that
\begin{align*}
\left(\frac{\tau_k}{c_k} - \frac{\tau_{k-1}}{c_{k-1}}\cdot\frac{c_{k-1}}{c_k}\right) = \frac{2}{d-1} \frac{L_k}{c_k} - \frac{2}{(d-1) c_k}\xrightarrow[k\to \infty]{} \frac{2}{d-1} W,
\end{align*}
since $c_k\to\infty$.
Since $\frac{\tau_k}{c_k}$ and $\frac{\tau_{k-1}}{c_{k-1}}$ either both diverge or have
the same limit $\tau^\star$, it follows that
\begin{align*}
\left(\tau^\star - \frac{\tau^\star}{\nu}\right) = \frac{2}{d-1} W.
\end{align*}
Thus $\frac{\tau_k}{c_k}$ converges almost surely to an almost surely positive random variable $\tau^\star$
\begin{equation*}
\tau^\star = \lim_{k\to\infty}\frac{\tau_k}{c_k} = \left(1-\frac{1}{\nu}\right)^{-1} \frac{2}{d-1} W > 0.
\end{equation*}

Hence 
\begin{equation*}
\frac{\tau_{k-1}}{\tau_k} = \frac{\tau_{k-1}}{c_{k-1}}\cdot
\frac{c_{k-1}}{c_k} \cdot \frac{c_k}{\tau_k}
  \xrightarrow[k\to\infty]{} \tau^\star \cdot \frac{1}{\nu} \cdot \frac{1}{\tau^\star} = \frac{1}{\nu},\quad\text{ almost surely.}
\end{equation*}
Now from \eqref{eq:tau_k_increments} we get
\begin{align*}
\lim_{k\to\infty} \frac{|\R_k|}{\tau_k} &=
\lim_{k\to\infty} \frac{1}{2}\left(1- \frac{\tau_{k-1}}{\tau_k}\right) =\frac{1}{2}\left(1-\frac{1}{\nu}\right) = \frac{d-1}{2(d-m)}
= \frac{d-1}{2\E[\rho(v)]},
\end{align*}
and this proves the claim.
\end{proof}
Passing from the range along a subsequence $(\tau_k)$ to the range $R_n$ at all times requires additional work, because of the exponential growth of the increments $(\tau_{k+1}-\tau_k)$. We next prove that the almost sure limit $\frac{|R_n|}{n}$ exists.

\begin{proof}[Proof of Theorem \ref{main-thm-range}(i)]
Let $x_1,\ldots,x_d$ be the $d$ children of the root vertex of $\mathbb{T}_d$. Let $R_k^{(1)},\ldots,R_k^{(d)}$
be the range of $d$ independent recurrent rotor walks on the tree $\mathbb{T}_d$ with i.i.d. initial rotor configurations, at the $k$-th return to the sink vertex.
Moreover, let $\xi_k^{(1)},\ldots,\xi_k^{(d)}$ be the times of the $k$-th visit to the sink vertex by these rotor walks.
One can couple the original rotor walk with these $d$ independent
rotor walks in such a way that the dynamics of the original rotor walk in the
component of the tree rooted at $x_i$ is given by the dynamics in the $i$-th rotor
walk, for $i=1,\ldots,d$.

Let $n$ be an arbitrary positive number in $(\tau_k,\tau_{k+1})$. Then the original rotor walk $(X_n)$ at time $n$ is located in the component of the tree rooted at $x_i$ for some $i$. The coupling above then gives us these two inequalities:
\begin{align*}
n & \geq \xi_{k+1}^{(1)}+\cdots+\xi_{k+1}^{(i-1)}+\xi_k^{(i)}+\cdots+\xi_k^{(d)}+k\\
R_n & \leq R_{k+1}^{(1)}+\cdots+R_{k+1}^{(i)}+R_k^{(i+1)}+\cdots+R_k^{(d)}+k.
\end{align*}
Note that the addition of $k$ in the inequalities above is to account for the
time spent at the sink vertex in the original rotor walk. The contribution of
$k$ is negligible as $n\to\infty$ as far as we are concerned. It then follows that
\begin{equation*}
\limsup_{n\to\infty}\frac{R_n}{n}\leq \limsup_{n\to \infty}\dfrac{R_{k+1}^{(1)}+\cdots+R_{k+1}^{(i)}+R_k^{(i+1)}+\cdots+R_k^{(d)}}{\xi_{k+1}^{(1)}+\cdots+\xi_{k+1}^{(i-1)}+\xi_k^{(i)}+\cdots+\xi_k^{(d)}},
\end{equation*}
where $k:=k(n)$ and $i:=i(n)$ depend on $n$. By Theorem \ref{thm:range_lln_homtree}, it then follows that
$$
\limsup_{n\to\infty}\frac{R_n}{n}\leq \limsup_{n\to\infty}\alpha\dfrac{\xi_{k+1}^{(1)}+\cdots+\xi_{k+1}^{(i)}+\xi_k^{(i+1)}+\cdots+\xi_k^{(d)}}{\xi_{k+1}^{(1)}+\cdots+\xi_{k+1}^{(i-1)}+\xi_k^{(i)}+\cdots+\xi_k^{(d)}},
$$
for $ \alpha=\frac{d-1}{2\mathbb{E}[\rho]}$. A direct calculation then gives us 
\begin{equation}\label{eq:limsup-r_n}
\limsup_{n\to\infty}\frac{R_n}{n}\leq \alpha+\alpha \limsup_{n\to\infty}\dfrac{\xi_{k+1}^{(i)}-\xi_k^{(i)}}{\xi_{k+1}^{(1)}+\cdots+\xi_{k+1}^{(i-1)}+\xi_k^{(i)}+\cdots+\xi_k^{(d)}}.
\end{equation}
Recall from the proof of Theorem \ref{thm:range_lln_homtree} that there exists $c_k>0$ and an integrable random variable $W$ such that, for any $i$
$$
\lim_{k\to\infty}\frac{\xi^{(i)}_k}{c_k}=W \quad \text{and }\quad \lim_{k\to\infty}\frac{c_{k+1}}{c_k}=\nu.
$$
Let now $W_1,\ldots,W_d$ be i.i.d random variables with the same distribution as $W$. It then follows that
\begin{align*}
\limsup_{n\to\infty}\frac{R_n}{n} & \leq \alpha+\alpha\limsup_{n\to\infty}\dfrac{c_k^{-1}\left(\xi_{k+1}^{(i)}-\xi_k^{(i)}\right)}{c_k^{-1}\left(\xi_{k+1}^{(1)}+\cdots+\xi_{k+1}^{(i-1)}+\xi_k^{(i)}+\cdots+\xi_k^{(d)}\right)}\\
& =\alpha +\alpha\dfrac{(\nu-1)\max_{i\leq d}W_i}{\nu(W_1+\cdots+W_{i-1})+W_i+\cdots+W_d}\\
& \leq \alpha +\alpha (\nu-1)\frac{\max_{i\leq d}W_i}{W_1+\cdots+W_d}.
\end{align*}
The same coupling can be applied not only to the $d$ children of the root, but also to the vertices of $\mathbb{T}_d$ and any given level $j$. This means that, for any $j\geq 0$, we have
$$
\limsup_{n\to\infty}\frac{R_n}{n}\leq \alpha+\alpha (\nu-1)\frac{\max_{i\leq d^j}W_i}{W_1+\cdots+W_{d^j}}.
$$
Since $W$ is an integrable random variable, it follows from the law of large numbers that $\frac{\max_{i\leq d^j}W_i}{W_1+\cdots+W_{d^j}}\to 0$ as $j\to\infty$. Hence we conclude that 
$$
\limsup_{n\to\infty}\frac{R_n}{n}\leq \alpha.
$$
By symmetry, we can also conclude that $\liminf_{n\to\infty}\frac{R_n}{n}\geq \alpha$. Theorem \ref{main-thm-range}(i) now follows.
\end{proof}

If $\nu$ is the mean offspring number of the supercritical Galton-Watson process $(L_k)$,
with $L_k=|\partial_o R_{\tau_k}|$, then we can also write  
\begin{align*}
\lim_{n\to\infty}\frac{|R_n|}{n} = 
\frac{1}{2}\left(1-\frac{1}{\nu}\right), \quad \text{almost surely}.
\end{align*}

\subsubsection{Null recurrent rotor walks}

In this section we consider null recurrent rotor walks, that is $\E[\rho(v)] = d - 1$.
Recall the stopping times $\tau_k$ as defined in \eqref{eq:tauk-rec}. The proofs for the law of large numbers for the range will be slightly different, arising from the fact that the expected return time to the sink for the rotor walker is infinite. We first prove the following.

\begin{theorem}\label{thm:range-null-rec}
For a null recurrent rotor walk $(X_n)$ on $\TT_d$, $d\geq 2$ we have
\begin{equation*}
\lim_{k\to\infty}\dfrac{|R_{\tau_k}|}{\tau_k}=\frac{1}{2}, \text{ almost surely}.
\end{equation*}
\end{theorem}
\begin{proof}
Rewriting equation \eqref{eq:tau_k_increments}, we get
\begin{equation*}
\frac{1}{2}\left(1-\frac{\tau_{k-1}}{\tau_k}\right)=\frac{|R_{\tau_k}|}{\tau_k},\quad \text{almost surely},
\end{equation*}
and we prove that the quotient $\frac{\tau_{k-1}}{\tau_k}$ goes to zero almost surely. 
We write again $\R_k=R_{\tau_k}$ and we first show that $\frac{\tau_k}{\tau_{k-1}}\to\infty$ almost surely, by finding a lower bound which converges to $\infty$ almost surely. From \eqref{eq:tau_k_increments}, we have
\begin{equation}\label{eq:frac-tau-k}
\frac{\tau_k}{\tau_{k-1}}>2\frac{|\R_k|}{\tau_{k-1}},\quad \text{almost surely}.
\end{equation}
If $\partial_0 \R_{k-1}$ is the set of leaves of $\R_{k-1}$, then in the time interval $\tau_k-\tau_{k-1}$,
the i.i.d critical Galton-Watson trees rooted at the leaves $\partial_0 \R_{k-1}$ will be added 
 to the current range $\R_{k-1}$.

Recall that from the proof of Theorem \ref{thm:range_lln_homtree} that $L_k=|\partial_0 \mathcal{R}_k|$. For each $k=1,2,\ldots$ we partition the time interval $(\tau_{k},\tau_{k+1}]$ into finer intervals, on which the behavior of the range can be easily controlled. The vertices in $\partial_oR_{\tau_k}=\{x_1,x_2,\ldots,x_{L_k}\}$  are ordered from right to left. We introduce the following two (finite) sequences of stopping times $(\eta_k^i)$ and $(\theta_k^i)$ of random length $L_k+1$, as follows: let $\theta_k^0=\tau_k$ and $\eta_k^{L_k+1}=\tau_{k+1}$ and for $i=1,2,\ldots,L_k$
\begin{equation}
\label{eq:eta-theta}
\begin{aligned}
\eta_k^i=&\min\{j> \theta_k^{i-1}:\ X_j=x_i\}\\
\theta_k^i=&\min\{j>\eta_k^i:\ X_j=x_i\text{ and } \rho(x_j)=x_i^{(d)}\}.
\end{aligned}
\end{equation}
That is, for each leaf $x_i$, the time $\eta_k^i$ represents the first time the rotor walk reaches $x_i$, and $\theta_k^i$ represents the last time the rotor walk returns to $x_i$ after making a full excursion in the critical Galton-Watson tree rooted at $x_i$. Then
\begin{equation*}
(\tau_k,\tau_{k+1}]=\left\{\cup_{i=1}^{L_k+1}\left(\theta_k^{i-1},\eta_k^i\right]\right\}\cup \left\{\cup_{i=1}^{L_k}\left(\eta_k^i,\theta_k^i\right]\right\},
\end{equation*}
almost surely. It is easy to see that the increments $(\theta_k^i-\eta_k^i)_{i}$ are i.i.d, and distributed according to the distribution of $\tau_1$, which is the time a rotor walk needs to return to the sink for the first time. Once the rotor walk reaches the leaf $x_i$ for the first time at time $\eta_k^i$, the subtree rooted at $x_i$ was never visited before by a rotor walk. Even more, the tree of good children with root $x_i$ is a critical Galton-Watson tree, which becomes extinct almost surely. Thus, the rotor walk on this subtree is (null) recurrent, and it returns to $x_i$ at time $\theta_k^i$. Then  $(\theta_k^i-\eta_k^i)$ represents the length of this excursion which has expectation $\E[\tau_1]=\infty$. 
In the time interval $(\theta_k^{i-1},\eta_k^i]$, the rotor walk leaves the leaf $x_{i-1}$
and returns to the confluent between $x_{i-1}$ and $x_i$, from where it continues its journey until it reaches $x_i$.
Then $\eta_k^i-\theta_k^{i-1}$ is the time the rotor walk needs to reach the new leaf $x_i$ after leaving $x_{i-1}$. In this time interval, the range does not change, since $(X_n)$ makes steps only in $R_{\tau_k}$.
We have, as a consequence of \eqref{eq:tau_k_increments}
\begin{equation}\label{eq:rk-rec}
|\R_k|=|\R_{k-1}|+\frac{1}{2}\sum_{i=1}^{L_{k-1}}\left(\theta_{k-1}^i-\eta_{k-1}^i\right),\quad \text{ almost surely}.
\end{equation}
By the strong law of large numbers we have on one side 
\begin{equation}\label{eq:lln-increments}
\frac{\sum_{i=1}^{L_{k-1}}\left(\theta_{k-1}^i-\eta_{k-1}^i\right)}{L_{k-1}}\to\E[\tau_1]=\infty,\quad \text{almost surely}.
\end{equation}
From \eqref{eq:number_of_leaves} we obtain $L_{k-1}=1+(d-1)|\R_{k-1}|$ which together with \eqref{eq:frac-tau-k} and \eqref{eq:rk-rec} yields
\begin{align}
\frac{\tau_k}{\tau_{k-1}} & >\frac{\sum_{i=1}^{L_{k-1}}\left(\theta_{k-1}^i-\eta_{k-1}^i\right)}{L_{k-1}}\cdot \frac{1+(d-1)|\R_{k-1}|}{\tau_{k-1}} \label{ineq:tauk}\\
& \geq \frac{\sum_{i=1}^{L_{k-1}}\left(\theta_{k-1}^i-\eta_{k-1}^i\right)}{L_{k-1}}\cdot \frac{|\R_{k-1}|}{\tau_{k-1}}\\
& = \frac{\sum_{i=1}^{L_{k-1}}\left(\theta_{k-1}^i-\eta_{k-1}^i\right)}{L_{k-1}}\cdot 
\frac{1}{2}\left(1-\frac{\tau_{k-2}}{\tau_{k-1}}\right),
\end{align}
almost surely, where the last inequality follows from \eqref{eq:tau_k_increments}.
By letting $l=\liminf \frac{\tau_k}{\tau_{k-1}}$, and taking limits the previous equation yields \begin{equation*}
l\geq \frac{1}{2}\E[\tau_1]\left(1-\frac{1}{l}\right).
\end{equation*}
Unless $l=1$, the right hand side above goes to infinity almost surely, which implies $l=\infty=\liminf \frac{\tau_k}{\tau_{k-1}}\leq\limsup \frac{\tau_k}{\tau_{k-1}}$, therefore $\lim_k \frac{\tau_k}{\tau_{k-1}}=\infty$ and $\lim_k \frac{\tau_{k-1}}{\tau_{k}}=0$, almost surely.
Suppose now $l=1$, almost surely. Once at the root at time $\tau_{k-1}$, until the next return at time $\tau_k$, the rotor walk visits everything that was visited before plus new trees where the configuration is in the initial status. As a consequence of Lemma \ref{lem:key-lemma}, for visiting the previously visited set, it needs time $\tau_{k-1}$, therefore $\tau_k-\tau_{k-1}>\tau_{k-1}$, which gives that $\liminf_k\frac{\tau_k}{\tau_{k-1}}>2$, which contradicts the fact that $l=1$. Therefore $\lim_k \frac{\tau_k}{\tau_{k-1}}=\infty$, almost surely.
Finally, we show that indeed $\lim_{k\to\infty}\frac{|\R_k|}{\tau_k}$ exists. On one hand, from \eqref{eq:tau_k_increments}, it is easy to see that $\frac{|\R_k|}{\tau_k}\leq \frac{1}{2}$.  We have
\begin{equation*}
\frac{1}{2}\geq \frac{|\R_k|}{\tau_k} \geq \liminf \frac{|\R_k|}{\tau_k}=\frac{1}{2}\liminf \left(1-\frac{\tau_{k-1}}{\tau_k}\right)=\frac{1}{2},
\end{equation*}
almost surely, and the claim follows.
\end{proof}

\begin{proof}[Proof of Theorem \ref{main-thm-range}(ii)]
Set again $\mathcal{R}_k:=R_{\tau_k}$ and recall from the proof of Theorem \ref{thm:range-null-rec}, the definition of the stopping times $\eta_k^i$ and $\theta_k^i$, for $i=1,\ldots,L_k=|\partial_0\mathcal{R}_k|$. Let $n$ be an arbitrary positive number in $(\tau_k,\tau_{k+1}]$. Then there exists $i\in\{1,\ldots,L_k\}$ such that  $n\in (\theta_k^{i-1},\theta_k^i].$
Then we have
\begin{align}
|\mathcal{R}_k|+\sum_{j=1}^{i-1}|\mathcal{R}_1^j| & \leq |R_n|\leq |\mathcal{R}_k|+\sum_{j=1}^{i}|\mathcal{R}_1^j|\label{eq:rn}\\
 \tau_k+\sum_{j=1}^{i-1}\tau_1^j& \leq n \leq \tau_k+\sum_{j=1}^{i}\tau_1^j\label{eq:n},
\end{align}
where $(\mathcal{R}_1^j)_j$ and $(\tau_1^j)_j$ are i.i.d. random variables distributed like $\mathcal{R}_{1}=R_{\tau_1}$ and $\tau_1$, respectively. From \eqref{eq:tau_k_increments} we have $\tau_1^j=2|\mathcal{R}_1^j|$, and  from 
the previous two equations we obtain the following upper bound on $\frac{|R_n|}{n}$:
\begin{align*}
\frac{|R_n|}{n} & \leq \frac{1}{2}\left[\dfrac{\tau_k-\tau_{k-1}+\sum_{j=1}^i \tau_1^j}{\tau_k+\sum_{j=1}^{i-1} \tau_1^j} \right]
=\frac{1}{2}\left[1+\dfrac{\tau_1^i}{\tau_k+\sum_{j=1}^{i-1} \tau_1^j}-\dfrac{\tau_{k-1}}{\tau_k+\sum_{j=1}^{i-1} \tau_1^j}\right].
\end{align*}
Since for every $i$, $\tau_1^i$ is almost surely finite and $\tau_k\to\infty$ as $k\to\infty$, the term $\dfrac{\tau_1^i}{\tau_k+\sum_{j=1}^{i-1} \tau_1^j}$ converges almost surely to $0$ as $k\to\infty$. Moreover, since in the null recurrent case, from the proof of Theorem \ref{thm:range-null-rec}, $\frac{\tau_{k-1}}{\tau_k}\to 0$ almost surely, as $k\to\infty$, we also get the almost sure convergence to $0$ of $\dfrac{\tau_{k-1}}{\tau_k+\sum_{j=1}^{i-1} \tau_1^j}$. Taking limits on both sides in the equation above we obtain that $$\limsup_{n\to\infty}\frac{|R_n|}{n}\leq \frac{1}{2},\quad \text{almost surely}.$$
For the lower bound, equations \eqref{eq:rn} and \eqref{eq:n} yield
\begin{align*}
\frac{|R_n|}{n} & \geq \frac{1}{2}\left[\dfrac{\tau_k-\tau_{k-1}+\sum_{j=1}^{i-1} \tau_1^j}{\tau_k+\sum_{j=1}^{i} \tau_1^j} \right]=
\frac{1}{2}\left[1-\dfrac{\tau_1^i}{\tau_k+\sum_{j=1}^{i} \tau_1^j}-\dfrac{\tau_{k-1}}{\tau_k+\sum_{j=1}^{i} \tau_1^j}\right],
\end{align*}
and by the same reasoning as above, we obtain
$$\liminf_{n\to\infty}\frac{|R_n|}{n}\geq \frac{1}{2},\quad \text{almost surely},$$
which together with the upper bound on limsup proves that $\lim_{n\to\infty}\frac{|R_n|}{n}=\frac{1}{2}$ almost surely.
\end{proof}
\subsection{Transient rotor walks}
\label{sec:trans-range}

We consider here the transient case on regular trees, when $\E[\rho(v)]<d-1$.
Then each vertex is visited only finitely many times, and the walk escapes to infinity along a live path. The tree $\T^{\mathsf{good}}_d$ of good children for  $(X_n)$ is a supercritical Galton-Watson tree, with mean offspring number $m=d-\mathbb{E}[\rho(v)]>1$. Thus, $\T^{\mathsf{good}}_d$ survives with positive probability $(1-q)$, where $q\in (0,\infty)$ is the extinction probability, and is the smallest nonnegative root of the equation $f(s)=s$, where $f$ is the generating function of $\T^{\mathsf{good}}_d$.

\begin{notation}
\emph{For the rest of this section, we will always condition on the
event of non-extinction, so that $\T^{\mathsf{good}}_d$ is an infinite random tree.} We denote by 
$\mathbb{P}_{\mathsf{non}}$ and by $\mathbb{E}_{\mathsf{non}}$ the associated probability and expectation conditioned on non-extinction, respectively. 
\end{notation}
That is, if $\mathbb{P}$ is the probability for the rotor walk in the original tree $\T^{\mathsf{good}}_d$, then for some event $A$, we have
\begin{equation*}
\mathbb{P}_{\mathsf{non}}[A]=\frac{\mathbb{P}[A\cap \T^{\mathsf{good}}_d \text{ is infinite}]}{1-q}.
\end{equation*}
In order to understand how the rotor walk $(X_n)$ escapes to infinity, we will decompose the tree 
$\T^{\mathsf{good}}_d$ with generating function $f$ conditioned on non-extinction.
Consider the generating functions
\begin{equation}\label{eq:g-h-genfunc}
g(s)=\dfrac{f\left((1-q)s+q\right)-q}{1-q} \quad \text{and} \quad h(s)=\dfrac{f(qs)}{q}.
\end{equation}
Then the $f$-Galton-Watson tree $\T^{\mathsf{good}}_d$ can be generated by:
\begin{enumerate}[(i)]
\item growing a Galton-Watson tree $\T_g$ with generating function $g$, which has the survival probability $1$.
\item attaching to each vertex $v$ of $\T_g$ a random number $n_v$ of h-Galton–Watson
trees, acting as traps in the environment $\T^{\mathsf{good}}_d$.

\end{enumerate}
The tree $\T_g$ is equivalent (in the sense of finite dimensional distributions) with a tree in which all vertices have an infinite line of descent. $\T_g$ is called the \emph{backbone}
of $\T^{\mathsf{good}}_d$. 
The random variable $n_v$ has a distribution depending only on $d(v)$ in  $\T^{\mathsf{good}}_d$ and given $\T_g$ and $n_v$ the traps are i.i.d. The supercritical Galton-Watson tree $\T^{\mathsf{good}}_d$ conditioned to die out is equivalent to the $h$-Galton-Watson tree with generating function $h$, which is subcritical. For more details on this decomposition and the equivalence of the processes involved above, see \cite{lyons1992-percolation}
and \cite[Chapter I, Part D]{athreya-ney}.
 
The exposition in this paragraph  is a nontrivial consequence of the key Lemma \ref{lem:key-lemma}(b).
Denote by $t_0<\infty$ the number of times the origin was visited. After the $t_0$-th  visit to the origin, there are no returns to the origin, almost surely, and there has to be a leaf $\gamma_0$ belonging to the range of the rotor walk up to the $t_0$-th return, along which the rotor walker escapes to infinity, that is, there is a live path starting at $\gamma_0$, almost surely. 
Denote by $n_0$ the first time the rotor walk arrives at $\gamma_0$. The tree rooted at $\gamma_0$ was not visited previously by the walker, and the rotors are in their random initial configuration.  The tree of good children $\T^{\mathsf{good}}_d$ rooted at $\gamma_0$ in the initial rotor configuration has the same distribution as the supercritical Galton-Watson tree conditioned on non-extinction.
At time $n_0$, we have already a finite visited subtree and its cardinality $|R_{n_0}|$, which is negligible for the limit behavior of the range. When computing the limit for the size of the range, we have to consider also this irrelevant finite part. 

On the event of non-extinction, let $\gamma=(\gamma_0,\gamma_1,\ldots)$ be the rightmost infinite ray in $\T^{\mathsf{good}}_d$ rooted at $\gamma_0$, that is, the rightmost infinite live path in $\TT_d$, which starts at $\gamma_0$. This is the rightmost ray in the tree $\T_g$. Since all vertices in $\T_g$ have an infinite line of descent, such a ray exists. The ray $\gamma$ is then a live path, along which the rotor walk $(X_n)$ escapes to infinity, without visiting the vertices to the left of $\gamma$; see again Lemma \ref{lem:key-lemma}(b). In order to understand the behavior of the range of $(X_n)$ and to prove a law of large numbers, we introduce the sequence of \emph{regeneration times} $(\tau_k)$ for the ray $\gamma$. Let $\tau_0=n_0$ and for $k\geq 1$:
\begin{equation}\label{eq:tau-trans}
\tau_k=\inf\{n\geq n_0: X_n=\gamma_k\}.
\end{equation}
Note that, for each $k$ the random times $\tau_k$ and $\tau_{k+1}-1$ are the first and the last hitting time of $\gamma_k$, respectively. Indeed, once we are at vertex $\gamma_k$, since $\gamma_{k+1}$ is the rightmost child of $\gamma_k$ with infinite line of descent, the rotor walk visits all good children to the right of $\gamma_k$, and makes finite excursions in the trees rooted at those good children, and then returns to $\gamma_k$ at time $\tau_{k+1}-1$. Then, at time $\tau_{k+1}$, the walk moves to $\gamma_{k+1}$ and never returns to $\gamma_k$.
We first prove the following.
\begin{theorem}\label{thm:transience-reg-tree}
Let $(X_n)$ be a transient rotor walk on $\TT_d$. 
If the tree $\T^{\mathsf{good}}_d$ of good children for the rotor walk $(X_n)$ has extinction probability $q$, then conditioned on non-extinction, there exists a constant $\alpha>0$, which depends only on $q$ and $d$, such that
\begin{equation*}
\lim_{k\to\infty}\frac{|R_{\tau_k}|}{\tau_k}=\alpha,\quad \text{almost surely},
\end{equation*}
and $\alpha$ is given by
\begin{equation*}
\alpha=\dfrac{q-f'(q)(q^2-q+1)}{q^2+q-f'(q)(2q^2-q+1)}.
\end{equation*}
\end{theorem}
\begin{proof}
We write again $\mathcal{R}_k:=R_{\tau_k}$, and $\R_0$ for the range of the rotor walk up to time $n_0$, which is finite almost surely. 
Since $\gamma$ is the rightmost infinite live path on which $(X_n)$ escapes to infinity, to the right of each vertex $\gamma_k$ in $\T^{\mathsf{good}}_d$ we have a random number of
vertices, and in the tree rooted at those vertices (which are $h$-Galton-Watson trees), the rotor walk makes only finite excursions. Then, at time $\tau_{k+1}$, the walk reaches $\gamma_{k+1}$ and never returns to $\gamma_k$.
 
For each $k$, $\gamma_{k+1}$ is a good child of $\gamma_k$ and the rotor at $\gamma_k$ points to the right of $\gamma_{k+1}$.\\ For $k=0,1,\ldots$ let $\left\{\gamma_k^{(1)},\ldots,\gamma_k^{(N_k)}\right\}$ be the set of vertices which are good children of $\gamma_k$, and are situated to the right of $\gamma_{k+1}$; denote by $N_k$ the cardinality of this set.
Additionally, denote by $\widetilde{T}_k(j)$ the tree rooted at $\gamma_k^{(j)}$, $j=1,\ldots, N_k$
and by $T_k(j)=\widetilde{T}_k(j)\cup (\gamma_k,\gamma_k^{(j)})$. That is, the trees $T_k(j)$, for $j=1,\ldots, N_k$ have all common root $\gamma_k$, and $|T_k(j)|=|\widetilde{T}_k(j)|+1$.

\textbf{Claim 1.} Conditionally on the event of non-extinction,  $(N_k)_{k\geq 0}$ are i.i.d.\\
\emph{Proof of Claim 1. }
For each $k$, the distribution of $N_k$ depends only on the offspring distribution (which is the number of good children for the rotor walk) of $\T_d^{\mathsf{good}}$, and the last one depends only on the initial rotor configuration $\rho$. Since the random variables $(\rho(v))_v$ are i.i.d, the claim follows.

\textbf{Claim 2.} Conditionally on the event of non-extinction, $(|T_k(j)|)_{1\leq j\leq N_k}$ are i.i.d. \\
\emph{Proof of Claim 2. }
This follows immediately from the definition of the Galton-Watson tree, since each vertex in $\T_d^{\mathsf{good}}$ has $k$ children with probability $r_{d-k}$, independently of all other vertices, and all these children have independent Galton-Watson descendant subtrees. All the subtrees $\widetilde{T}_k(j)$ rooted at $\gamma_k^{(j)}$ are then independent h-Galton-Watson subtrees (subcritical), therefore $(|\widetilde{T}_k(j)|)_{1\leq j\leq N_k}$ are i.i.d. Since $|T_k(j)|=|\widetilde{T}_k(j)|+1$, the claim follows.

\textbf{Claim 3.} Given non-extinction, the increments $(\tau_{k+1}-\tau_k)_{k\geq 0}$  are i.i.d..\\
\emph{Proof of Claim 3. }
Given non-extinction, the time $(\tau_{k+1}-\tau_k)_{k\geq 0}$ depends only on $N_k$ (the number of good children to the right of $\gamma_k$) and on $|T_k(j)|$, which by Claim 1 and Claim 2 above, are all i.i.d. Thus the independence of $(\tau_{k+1}-\tau_k)_{k\geq 0}$ follows as well.

Clearly, each $\gamma_k$ is visited exactly $N_{k}+1$ times, and all vertices to the left of $\gamma$ are never visited. We write $\R_{(k-1,k]}$ for the range of the rotor walk in the time interval $(\tau_{k-1},\tau_k]$. For $r\neq s$, the path of the rotor walk  in the time interval $(\tau_{r-1},\tau_r]$ has empty intersection with the path in the time interval $(\tau_{s-1},\tau_s]$, and we have 
\begin{equation*}
|\R_k|=|\R_0|+\sum_{i=1}^k|\R_{(i-1,i]}|,\quad \text{almost surely}.
\end{equation*}
Moreover, for $i=1,\ldots k$,
\begin{equation*}
|\R_{(i-1,i]}|=1+\sum_{j=1}^{N_{i-1}}|\widetilde{T}_{i-1}(j)|,\quad \text{almost surely}
\end{equation*}
and if we denote by $\alpha_k=\sum_{i=0}^{k-1}N_i$ and if $(\tilde{t}_j)$ is an i.i.d. sequence of random variables with the same distribution as $|\widetilde{T}_{i-1}(j)|$, then, using Claim 1 and Claim 2 we obtain
\begin{equation}\label{eq:rk}
|\R_k|=|\R_0|+k+\sum_{j=1}^{\alpha_k}\tilde{t}_j,\quad \text{ almost surely}.
\end{equation}
Similarly, if we write $\tau_k=n_0+\sum_{i=2}^k(\tau_i-\tau_{i-1})$ and use the fact that, for $i\geq 2$ as a consequence of \eqref{eq:tau_k_increments}
\begin{equation*}
\tau_i-\tau_{i-1}=1+2\sum_{j=1}^{N_{i-1}}|\widetilde{T}_{i-1}(j)|, \quad \text{almost surely},
\end{equation*}
then again by Claim 1 and Claim 2 we get
\begin{equation}\label{eq:tau-k}
\tau_k=|\R_0|+k+2\sum_{j=1}^{\alpha_k}\tilde{t}_j.
\end{equation}
Putting equations \eqref{eq:rk} and \eqref{eq:tau-k} together, we finally get
\begin{equation*}
\frac{|\R_k|}{\tau_k}=\dfrac{|\R_0|+k+\sum_{j=1}^{\alpha_k}\tilde{t}_j}{|\R_0|+k+2\sum_{j=1}^{\alpha_k}\tilde{t}_j},\quad \text{almost surely}.
\end{equation*}
By the strong law of large numbers, we have 
\begin{equation*}
\frac{|\R_0|+k}{\alpha_k}\to \frac{1}{\mathbb{E}_{\mathsf{non}}[N_0]},\quad \frac{\sum_{j=1}^{\alpha_k}\tilde{t}_j}{\alpha_k}\to \mathbb{E}_{\mathsf{non}}[|\widetilde{T}_0(1)|]\quad \text{almost surely}
\end{equation*}
which implies
\begin{equation*}
\frac{|\R_k|}{\tau_k}\to \dfrac{1+\mathbb{E}_{\mathsf{non}}[N_0]\mathbb{E}_{\mathsf{non}}[|\widetilde{T}_0(1)|]}{1+2\mathbb{E}_{\mathsf{non}}[N_0]\mathbb{E}_{\mathsf{non}}[|\widetilde{T}_0(1)|]},\quad \text{almost surely}. 
\end{equation*}
We have to compute now the two expectations $\mathbb{E}_{\mathsf{non}}[N_0]$ and $\mathbb{E}_{\mathsf{non}}[|\widetilde{T}_0(1)|]$ involved in the equation above, in order to get the formula from the statement of the theorem.

Conditioned on non-extinction, for all $k=0,1,\ldots$ and $j=1,\ldots N_k$ the trees $\widetilde{T}_k(j)$ are i.i.d. subcritical Galton-Watson trees with generating function $h(s)=\frac{f(qs)}{q}$ and mean offspring number $h'(1)$. That is, they all die out with probability one, and the expected number of vertices is $\mathbb{E}_{\mathsf{non}}[|\widetilde{T}_k(j)|]=\frac{1}{1-h'(1)}$. On the other hand $h(s)=\frac{1}{q}\sum_{j=0}^d r_{d-j}(qs)^j$, which implies that 
$h'(1)=\frac{1}{q^2}\sum_{j=0}^d jr_{d-j}  q^j$, and 
\begin{equation*}
\mathbb{E}_{\mathsf{non}}[\widetilde{T}_0(1)]=\dfrac{q^2}{q^2-\sum_{j=0}^d jr_{d-j}  q^j},
\end{equation*}
which in terms of the generating function $f(s)$ can be written as
\begin{equation}\label{eq:exp-t0}
\mathbb{E}_{\mathsf{non}}[\widetilde{T}_0(1)]=\frac{q}{q-f'(q)}.
\end{equation}
In computing $\mathbb{E}_{\mathsf{non}}[N_0]$:
\begin{align*}
&\mathbb{E}_{\mathsf{non}}[N_0]=\sum_{i=0}^{d-1}i\mathbb{P}_{\mathsf{non}}[N_0=i]=\frac{1}{1-q}\sum_{i=0}^{d-1}i\mathbb{P}[N_0=i, \T_d^{\mathsf{good}} \text{ is infinite}]\\
&=\frac{1}{1-q}\sum_{j=0}^d\sum_{i=0}^{j-1}i \mathbb{P}[N_0=i,\  \T_d^{\mathsf{good}} \text{ is infinite}|\gamma_0 \text{ has j good children}]
 \mathbb{P}[\gamma_0 \text{ has j good children}]\\
 &=\frac{1}{1-q}\sum_{j=0}^{d}r_{d-j}\sum_{i=0}^{j-1}iq^i(1-q)=\sum_{j=0}^{d}r_{d-j}\sum_{i=0}^{j-1}iq^i=\sum_{j=0}^{d}r_{d-j}\frac{(j-1)q^{j+1}-jq^j+q}{(1-q)^2}\\
 &=\frac{q}{(1-q)^2}\left(1-\sum_{j=0}^dr_{d-j}q^j-\frac{1-q}{q}\sum_{j=0}^d jr_{d-j}q^j\right)=\frac{q}{1-q}\left(1-\sum_{j=1}^d jr_{d-j}q^{j-1}\right)
\end{align*}
In terms of the generating function $f(s)$, using $f(q)=q$, $\mathbb{E}_{\mathsf{non}}[N_0]$ can be written as
\begin{equation}\label{eq:exp-n0}
\mathbb{E}_{\mathsf{non}}[N_0]=\frac{q}{1-q}\left(1-f'(q)\right).
\end{equation}
Putting the two expectations together, we obtain the
constant $\alpha$ in terms of the generating function $f(s)$ given by
\begin{equation*}
\alpha=\dfrac{q-f'(q)(q^2-q+1)}{q^2+q-f'(q)(2q^2-q+1)}=\dfrac{q-(q^2-q+1)\sum_{j=1}^d jr_{d-j}q^{j-1}}{q^2+q-(2q^2-q+1)\sum_{j=1}^d jr_{d-j}q^{j-1}}.
\end{equation*}
\end{proof}

In Theorem \ref{thm:transience-reg-tree}, we have proved a law of large numbers for the range of the  rotor walk along a subsequence $(\tau_k)$. With very little effort, we can show that we have indeed a law of large numbers at all times.

\begin{proof}[Proof of Theorem \ref{main-thm-range}(iii)]
For $n\in\mathbb{N}$, the infinite ray $\gamma$ and the regeneration times $(\tau_k)$ as defined in \eqref{eq:tau-trans} let 
\begin{equation*}
k=\max\{j: \tau_j< n\}.
\end{equation*}
Then $\tau_k< n\leq \tau_{k+1}$ a.s. and $|R_{\tau_k}|\leq |R_n| \leq |R_{\tau_{k+1}} |$ a.s. which in turn gives
\begin{equation}\label{eq:rn-stop-times}
\frac{\tau_{k+1}}{\tau_k} \cdot \frac{|R_{\tau_k}|}{\tau_k}\leq \frac{|R_n|}{n}\leq \frac{|R_{\tau_{k+1}}|}{\tau_{k+1}}\cdot \frac{\tau_k}{\tau_{k+1}},\quad \text{almost surely}.
\end{equation}
By Claim 3 from the proof of Theorem \ref{thm:transience-reg-tree}, the increments $(\tau_{k+1}-\tau_k)$ are i.i.d, and finite almost surely, therefore $\frac{\tau_{k+1}-\tau_k}{\tau_k}\to 0$ almost surely as $k\to \infty$. Then, since 
\begin{equation*}
\frac{\tau_{k+1}}{\tau_k}=\frac{\tau_{k+1}-\tau_k}{\tau_k}+1,
\end{equation*}
we have that $\frac{\tau_{k+1}}{\tau_k}\to 1$ almost surely, as $k\to \infty$.
This, together with Theorem \ref{thm:transience-reg-tree} implies that
 the left hand side of equation \eqref{eq:rn-stop-times} converges to $\alpha$ almost surely. By the same argument we obtain that also the right hand side of \eqref{eq:rn-stop-times} converges to the same constant $\alpha$ almost surely, and this completes the proof.
\end{proof}

\subsection{Uniform rotor walks}

We discuss here the behavior of rotor walks on regular trees $\TT_d$, with uniform initial rotor configuration $\rho$, that is, for all $v\in\TT_d$, the random variables $(\rho(v))$
are i.i.d with uniform distribution on the set $\{0,1,\ldots,d\}$, i.e. $\mathbb{P}[\rho(v)=j]=\frac{1}{d+1}$, for $j\in\{0,1,\ldots,d\}$. Such walks are null recurrent 
on $\TT_2$ and
transient on all $\TT_d$, $d\geq 3$. As a special case of Theorem \ref{main-thm-range}(iii),
 we have the following.
\begin{corollary}
If $(X_n)$ is a uniform rotor walk on $\TT_d$, $d\geq 3$, then the constant $\alpha$ is given by
\begin{equation*}
\alpha=\dfrac{q(1-d)(1-q^{d+1})}{(d+1)(q-1)^3}+\dfrac{q^2(1-q^{d-1})}{(q-1)^3}.
\end{equation*}
\end{corollary}
\begin{proof}
Using that $\mathbb{P}[\rho(v)=j]=\frac{1}{d+1}=r_{d-j}$ and putting $f'(q)=\frac{q^d}{q-1}-\frac{q^{d+1}-1}{(d+1)(q-1)^2}$ in  Theorem \ref{thm:transience-reg-tree} we get the result.
\end{proof}

The following table shows values for the constants $\alpha$ in comparison with the limit $(d-1)/d$ for the simple random walk on trees.

\begin{center}
\begin{tabular}{r|cc}
$d$ & $\alpha$ & $(d-1)/d$ \\
\hline
  2 & 0.500 & 0.500\\
  3 & 0.707 & 0.666\\
  4 & 0.784 & 0.750\\
  5 & 0.825 & 0.800\\
  6 & 0.853 & 0.833\\
  7 & 0.872 & 0.857\\
  8 & 0.888 & 0.875\\
  9 & 0.899 & 0.888\\
 10 & 0.909 & 0.900
\end{tabular}
\end{center}
Note that only in the case of the binary tree $\TT_2$, the limit values for the range of the uniform rotor walk and of the simple random walk are equal, even though the uniform rotor walk is null recurrent and the simple random walk is transient. In the transient case, that is, for all $d\geq 3$ we always have $\alpha > (d-1)/d$.

\section{Speed on regular trees}\label{sec:rate-escape}

In this section, we prove the existence of the almost sure limit $\frac{|X_n|}{n}$, as $n\to \infty$, where $|X_n|$ represents the distance from the root to the position $X_n$ at time $n$ of the walker. 

\subsection{Recurrent rotor walks}

\begin{proof}[Proof of Theorem \ref{main-thm-escape}(i)]
We show that in this case, $l=\lim_{n\to\infty}\frac{|X_n|}{n}=0$, almost surely.  For arbitrary $n$,  let
 \begin{equation*}
k=\max\{ i: \tau_i< n\},
\end{equation*}
which implies that $\tau_k< n\leq \tau_{k+1}$, and up to time $n$ we have $k$ returns to the root.
Let $D_k$ be the maximum distance from the root reached after $k$ returns to the root.

\textit{Positive recurrent rotor walks.} We have 
\begin{equation*}
0\leq \frac{|X_n|}{n}\leq \frac{D_{k+1}}{\tau_k}=\frac{D_{k+1}}{k+1}\cdot \frac{k+1}{\tau_k}.
\end{equation*}
In view of  \cite[Theorem 7(ii)]{angel_holroyd}, the maximal depths grows linearly with the number of returns to the root, that is, $\frac{D_{k+1}}{k+1}$ is almost surely bounded. On the other hand, $\tau_k$ grows exponentially in $k$, therefore $\frac{k+1}{\tau_k}$ converges almost surely to $0$ as $k\to\infty$, that is $l=\lim_{n\to\infty}\frac{|X_n|}{n}=0$ almost surely.

\textit{Null recurrent rotor walks.} In the null recurrent case, the situation is a bit different, since even though all particles return to the root, they reach very great depths; see again \cite[Theorem 7(i)]{angel_holroyd}. Write again $ \mathcal{R}_k:=R_{\tau_k}$. For the position of the rotor walk $X_n$, we distinguish the following three cases:
\begin{enumerate}[(i)]
\item If $X_n\in \mathcal{R}_{k-1}$, then 
$$0\leq \frac{|X_n|}{n}\leq \frac{| \mathcal{R}_{k-1}|}{\tau_k}=\frac{| \mathcal{R}_{k-1}|}{\tau_{k-1}}\cdot\frac{\tau_{k-1}}{\tau_k},$$
and the right hand side above converges to $0$ almost surely in view of Theorem \ref{thm:range-null-rec} together with the fact $\frac{\tau_{k-1}}{\tau_k}\to 0$ almost surely, proven again in Theorem \ref{thm:range-null-rec}. Therefore $\lim_{n\to\infty}\frac{|X_n|}{n}=0$ almost surely.
\item If  $X_n\in \mathcal{R}_{k}\setminus \mathcal{R}_{k-1}$, then there exists $i\in \{1,2,\ldots,L_{k-1}\}$, where $L_{k-1}=|\partial_0\mathcal{R}_{k-1}|$, such that $X_n$ is in the tree rooted at $x_i$, and
$$0\leq \frac{|X_n|}{n}\leq \frac{| \mathcal{R}_{k-1}|}{\tau_k}+\frac{\tau_1^i}{\tau_k},$$
where $\tau_1^i$ is a random variable, independent and identically distributed to $\tau_1$ which is finite almost surely.
The quantity $\frac{| \mathcal{R}_{k-1}|}{\tau_k}$ converges to $0$
by the same argument as in case (i), whereas $\frac{\tau_1^i}{\tau_k}$
converges also to $0$  almost surely, as $k=k(n)\to\infty.$  
\item Finally, if $X_n\in \mathcal{R}_{k+1}\setminus \mathcal{R}_{k}$, then there exist $i\in \{1,2,\ldots,L_{k-1}\}$ and $j\in \{1,2,\ldots,L_{k}\}$
such that 
$$0\leq \frac{|X_n|}{n}\leq \frac{| \mathcal{R}_{k-1}|}{\tau_k}+\frac{\tau_1^i}{\tau_k}+\frac{\tau_1^j}{\tau_k},$$
where both $\tau_1^i,\tau_1^j$ are i.i.d random variables,  identically distributed as $\tau_1$ which is finite almost surely, and the right hand side converges again to $0$ almost surely, as $k=k(n)\to \infty$, and this proves the claim.
\end{enumerate}
\vspace{-1cm}
\end{proof}

\subsection{Transient rotor walks}

Since in the transient case there is a positive probability of extinction of $\T_{d}^{\mathsf{good}}$, we condition on the event of non-extinction. The notation remains the same as in Section \ref{sec:trans-range}. 

\begin{proof}[Proof of Theorem \ref{main-thm-escape}(ii)]
Recall the definition of the infinite ray $\gamma$ along which the rotor walk $(X_n)$ escapes to infinity, and the regeneration times $\tau_k$ as defined in \eqref{eq:tau-trans}. We first prove the existence of the speed $l$ along the sequence $(\tau_k)$. This is rather easy, since $d(r,\gamma_k)=|X_{n_0}|+k$, a.s. where $n_0<\infty$ is the first time the walk reaches $\gamma_0$, from where it escapes without returning to the root; see again Lemma \ref{lem:key-lemma}(b).

 Recall now from the proof of Theorem \ref{thm:transience-reg-tree}, that $\tau_k=n_0+k+2\sum_{j=1}^{\alpha_k}\tilde{t}_j$ a.s. and $\alpha_k=\sum_{i=0}^{k-1}N_i$ a.s., with the involved quantities again as computed in  Theorem \ref{thm:transience-reg-tree}.
Then
\begin{equation*}
\frac{|X_{\tau_k}|}{\tau_k}=\frac{|X_{n_0}|+k}{\tau_k}=\frac{1+\frac{|X_{n_0}|}{k}}{1+\frac{n_0}{k}+\frac{2\sum_{j=1}^{\alpha_k}\tilde{t}_j}{\alpha_k}\cdot\frac{\alpha_k}{k}}.
\end{equation*}
By the strong law of large numbers for sums of i.i.d random variables, we have 
$\frac{2\sum_{j=1}^{\alpha_k}\tilde{t}_j}{\alpha_k}\to \mathbb{E}_{\mathsf{non}}[|\widetilde{T}_0(1)|]$ and $\frac{\alpha_k}{k}\to\mathbb{E}_{\mathsf{non}}[N_0]$ almost surely, as $k\to\infty$, Also, since $|X_{n_0}|$ and $n_0$ are both finite, the almost sure limits $\frac{|X_{n_0}|}{k}$
and $\frac{n_0}{k}$ are $0$, as $k\to\infty$. This implies 
\begin{equation*}
\frac{|X_{\tau_k}|}{\tau_k}\to \frac{1}{1+2\mathbb{E}_{\mathsf{non}}[|\widetilde{T}_0(1)|]\mathbb{E}_{\mathsf{non}}[N_0]},\quad \text{almost surely, as } k\to\infty.
\end{equation*}
By equation \eqref{eq:exp-t0} and \eqref{eq:exp-n0}, which give $\mathbb{E}_{\mathsf{non}}[|\widetilde{T}_0(1)|]$ and $\mathbb{E}_{\mathsf{non}}[N_0]$ in terms of the generating function $f$ of the tree $\T^{\mathsf{good}}_d$, we obtain
\begin{equation}\label{eq:escape-tauk}
\frac{|X_{\tau_k}|}{\tau_k}\to \frac{(q-f'(q))(1-q)}{q+q^2-f'(q)(2q^2-q+1)}=l,\quad \text{ almost surely as }k\to\infty.  
\end{equation}
In order to prove the almost sure convergence of $\frac{|X_n|}{n}$, for all $n$, we take 
\begin{equation*}
k=\max\{ i: \tau_i< n\}.
\end{equation*}
We know that $\tau_k< n\leq \tau_{k+1}$ a.s. $|X_{\tau_k}|=|X_{n_0}|+k$, $|X_{\tau_{k+1}}|=|X_{n_0}|+k+1$ a.s.
and $|X_n|\geq |X_{n_0}|+k$ a.s. Moreover, between times $\tau_k$ and $\tau_{k+1}$, the distance can increase with no more than $\tau_{k+1}-\tau_k$, and we have
\begin{equation*}
\frac{|X_{n_0}|+k}{\tau_{k+1}}\leq \frac{|X_n|}{n}\leq \frac{|X_{n_0}|+k+(\tau_{k+1}-\tau_k)}{\tau_k},\quad \text{almost surely}.
\end{equation*}
Since $\frac{|X_{\tau_k}|}{\tau_k}\cdot \frac{\tau_k}{\tau_{k+1}}=\frac{k}{\tau_{k+1}}$, together with equation \eqref{eq:escape-tauk} and the facts that $\frac{\tau_{k}}{\tau_{k+1}}\to 1$ and $\frac{|X_{n_0}|}{\tau_{k+1}}\to 0$ almost surely, we obtain that the left hand side of the equation above converges to $l$ almost surely, as $k\to \infty$. For the right hand side, we use again equation \eqref{eq:escape-tauk}, together with the fact that $\frac{\tau_{k+1}-\tau_k}{\tau_k}\to 0$ almost surely, since the increments $(\tau_{k+1}-\tau_k)$ are i.i.d and almost surely finite. This yields the almost sure convergence of the right hand side of the equation above to the constant $l$, which implies that $\frac{|X_n|}{n}\to l$ almost surely as $n\to \infty$.
\end{proof}

\section{Rotor walks on Galton-Watson trees}
\label{sec:rrw-galton}

The methods we have used in proving the law of lange numbers and the existence of the rate of escape for rotor walks $(X_n)$ with random initial configuration on regular trees can be, with minor modifications, adapted to the case when the rotor walk $(X_n)$ moves initially on a Galton-Watson tree $\T$. We get very similar results to the ones on regular trees, which we will state below. We will not write down the proofs again, but only mention the differences which appear on Galton-Watson trees.

Let $\T$ be a Galton-Watson tree 
with offspring distribution $\xi$ given by $p_k=\mathbb{P}[\xi=k]$, for $k\geq 0$ and we assume that $p_0=0$, that is  $\T$ is supercritical and survives with probability $1$. Moreover
the mean offspring number $\mu=\E[\xi]$ is also greater than $1$.
We recall the notation and the main result from \cite{huss_muller_sava_gw}.
For each $k\geq 0$ we choose a probability
distribution $\mathcal{Q}_k$ supported on $\{0,\ldots,k\}$. That is, we have the sequence of distributions $(\mathcal{Q}_k)_{k\in\N_0}$, where
\begin{equation*}
\mathcal{Q}_k = \big(q_{k,j}\big)_{0\leq j\leq k}
\end{equation*}
with $q_{k,j}\geq 0$ and $\sum_{j=0}^k q_{k,j} = 1$. Let $\mathcal{Q}$ be the infinite lower
triangular matrix having $\mathcal{Q}_k$ as row vectors, i.e.:
\begin{equation*}
\mathcal{Q} = 
\begin{pmatrix}
q_{00} & 0      & 0      & 0      & \hdots \\
q_{10} & q_{11} & 0      & 0      & \hdots \\
q_{20} & q_{21} & q_{22} & 0      & \hdots \\
q_{30} & q_{31} & q_{32} & q_{33} & \hdots \\
\vdots & \vdots & \vdots & \vdots & \ddots \\
\end{pmatrix}.
\end{equation*}
Below, we write $\d_x$ for the (random) degree of vertex $x$ in $\T$.
\begin{definition}
A random rotor configuration $\rho$ on  $\T$ is $\mathcal{Q}$-distributed,
if  for each $x\in \T$, the rotor $\rho(x)$ is a random variable with the following properties:
\begin{enumerate}
 \setlength\itemsep{0em}
\item $\rho(x)$ is $\mathcal{Q}_{\mathsf{d}_x}$ distributed, i.e.,
$\Pb[\rho(x) = \d_x - l \,|\, \d_x = k] = q_{k,l}$, with $l=0,\ldots \d_x$,
\item $\rho(x)$ and $\rho(y)$ are independent if $x\not=y$, with $x,y\in \T$.
\end{enumerate}
We write $\mathsf{R}_{\T}$ for the corresponding probability measure.
\end{definition}
Then $\rgw=\rc_{\T}\times \gw$  represents 
the probability measure given by choosing a tree $\T$ according to the $\gw$ measure, and then independently
choosing a rotor configuration $\rho$ on $\T$ according to $\rc_{\T}$. 
Recall that to the root $r\in\T$, we have added an additional sink vertex $s$.
If we start with $n$ rotor particles, one after another, at the root $r$ of $\T$, with random initial configuration $\rho$, and we denote by $E_n(\T,\rho)$ the number of particles out of $n$  that escape to infinity, then the main result of \cite{huss_muller_sava_gw} is the following.

\begin{theorem}
\label{main_thm}\cite[Theorem 3.2]{huss_muller_sava_gw}
Let $\rho$ be a random $\mathcal{Q}$-distributed rotor configuration on a Galton-Watson tree $\T$ with offspring distribution $\xi$, and let
$\nu = \xi\cdot\mathcal{Q}$. Then we have  for $\rgw$-almost all~$\T$ and $\rho$:
\begin{enumerate}[(a)]
 \setlength\itemsep{0em}
\item \label{main_thm_a}$\displaystyle E_n(\T,\rho) = 0$ for all $n\geq 1$, if $\E[\nu] \leq 1$,
\item \label{main_thm_b}$\displaystyle\lim_{n\to\infty} \frac{E_n(\T,\rho)}{n} = \gamma(\mathcal{T})$, if $\E[\nu] > 1$,
\end{enumerate}
where $\gamma(\T)$ represents the probability that simple random walk started at the root  of $\T$  never returns  to $s$. 
\end{theorem}
Let us denote $m:=\E[\nu]$. That is, if $m\leq 1$, then $(X_n)$ is recurrent and if $m>1$, then $(X_n)$ is transient.

\subsection{Range and speed on Galton-Watson trees}

\paragraph{Tree of good children.}
If $\rho$ is $\mathcal{Q}$-distributed,
\begin{equation*}
\Pb[x \text{ has $l$ good children} \,|\, \d_x = k] = q_{k,l},
\end{equation*}
then the distribution of the number of good children of a vertex $x$ in $\T$ is given by
\begin{equation}\label{eq:offspring-gw}
\Pb[x \text{ has $l$ good children}] = \sum_{k = l}^{\infty} p_k q_{k,l}=:r_l,
\end{equation}
which is the $l^{\text{th}}$ component of the vector $\nu = \xi\cdot\mathcal{Q}$.
The tree of good children $\T^{\mathsf{good}}$ is in this case a Galton-Watson tree  with offspring distribution
$\nu= \xi\cdot\mathcal{Q}$ whose mean was denoted by $m$. Denote by $f_{\T}$ the generating function of $\T^{\mathsf{good}}$ and by $q$ its extinction probability.
For the range $R^{\T}_n$ of rotor walks $(X_n)$ on Galton-Watson trees, we get the following result. 

\begin{theorem}\label{main-thm-range-GW}
Let $\T$ be a Galton-Watson tree with offspring distribution $\xi$ and mean offspring number $\E[\xi]=\mu>1$.
If $(X_n)$ is a rotor walk with random $\mathcal{Q}$-distributed  initial configuration on $\T$, and $\nu=\xi\cdot\mathcal{Q}$, then there is a constant $\alpha_{\T}>0$ such that
\begin{equation*}
\lim_{n\to\infty}\frac{|R^{\T}_n|}{n}=\alpha_{\T},\quad \rgw \text{ -almost surely}.
\end{equation*}
If we write $m=\E[\nu]$, then the constant $\alpha_{\T}$ is given by:
\begin{enumerate}[(i)]
 \setlength\itemsep{-1em}
\item If $(X_n)_{n\in\N}$ is positive recurrent, then
\begin{equation*}
\alpha_{\T}=\frac{\mu-1}{2(\mu-m)}.
\end{equation*}
\item If $(X_n)_{n\in\N}$ is null recurrent, then 
\begin{equation*}
\alpha_{\T}=\frac{1}{2}.
\end{equation*}
\item If $(X_n)_{n\in\N}$ is transient, then conditioned on the non-extinction of $\T^{\mathsf{good}}$,
\begin{equation*}
\alpha_{\T}=\frac{q-f_{\T}'(q)(q^2-q+1)}{q^2+q-f_{\T}'(q)(2q^2-q+1)} 
\end{equation*}
where $q>0$ is the extinction probability of $\T^{\mathsf{good}}_d$.
\end{enumerate}
\end{theorem}
The proof follows the lines of the proof of Theorem \ref{main-thm-range}, with minor changes which we state below. The offspring distribution of the tree of good children $\T^{\mathsf{good}}$ will be here replaced with
\eqref{eq:offspring-gw}, and the corresponding mean offspring number is $m$.  If $L_k$ and $\R_{\tau_k}$ are the same as in the proof of Theorem  \ref{thm:range_lln_homtree}, then in case when $(X_n)$ moves on a Galton-Watson tree $\T$, \begin{align}
\label{eq:Lk_Rk_relation}
L_k = 1 + \sum_{v\in \mathcal{R}_k} \xi_v - |\mathcal{R}_k|,\quad \text{almost surely},
\end{align}
where $\xi_v$ denotes the (random) number of children of the vertex $v\in \T$. Since
$\T$ is the initial Galton-Watson tree with mean offspring number $\mu$, by Wald's identity we get
\begin{align*}
\E[L_1] =1+ \frac{1}{1-m} \mu - \frac{1}{1-m} = \frac{\mu - m}{1-m}:=\lambda>1.
\end{align*}
Then $\nu$ in Theorem \ref{thm:range_lln_homtree} will  be replaced with $\lambda$ and the rest works through.
In the proof of Theorem \ref{main-thm-range}(i), $\alpha$ will be replaced with $\frac{\mu-1}{2(\mu-m)}$.
Theorem \ref{thm:range-null-rec} works as well here, with a minor change in equation \eqref{ineq:tauk},
where the last term will be $$\frac{1 + \sum_{v\in \mathcal{R}_{k-1}} \xi_v - |\mathcal{R}_{k}|}{\tau_{k-1}},$$
while the following relations stay the same.
In the transient case, in the proof of Theorem \ref{thm:transience-reg-tree}, when computing the expectation
$\E_{\mathsf{non}}[N_0]$, the degree $d-1$ has to be replaced by the offspring distribution $\xi$, which in terms of the generating function $f_{\T}$ of $\T^{\mathsf{good}}$, and its extinction probability $q$ produces the same result. 

If $(X_n)$ is an uniform rotor walk on the Galton-Watson tree $\T$, we have a particularly simple limit.
In \cite{huss_muller_sava_gw}, it was shown that $(X_n)$ is recurrent if and only if $\mu \leq 2$.
Moreover, the tree of good children $\T^{\mathsf{good}}$ is a subcritical Galton-Watson tree with
offspring distribution given by 
$\nu = \left(\sum_{k\geq l} \frac{1}{k+1}p_k\right)_{l\geq 0}$, and mean offspring number $m = \E[\nu] = \frac{\mu}{2} \leq 1$.

\begin{corollary}
\label{range_law_of_large_numbers_uniform}
For the range of uniform rotor walks 
on Galton-Watson trees $\T$, we have 
\begin{equation*}
\lim_{n\to\infty}\frac{|R_n^{\T}|}{n} = \frac{\mu - 1}{\mu},\quad \rgw- \text{ almost surely}. 
\end{equation*}
\end{corollary}

Finally, we also have the existence of the rate of escape.
\begin{theorem}
\label{main-thm-escape-GW}
Let $\T$ be a Galton-Watson tree with offspring distribution $\xi$ and mean offspring number $\E[\xi]=\mu>1$.
If $(X_n)$ is a rotor walk with random $\mathcal{Q}$-distributed  initial configuration on $\T$, and $\nu=\xi\cdot\mathcal{Q}$, then there exists a constant $l_{\T}\geq 0$, such that
\begin{equation*}
\lim_{n\to \infty}\frac{|X_n|}{n} =l_{\T},\quad \rgw-\text{almost surely}.
\end{equation*}
\begin{enumerate}[(i)]
\item If $(X_n)_{n\in\N}$ is  recurrent, then $l_{\T}=0$.
\item If $(X_n)_{n\in\N}$ is transient, then conditioned on non-extinction of $\T^{\mathsf{good}}$, 
\begin{equation*}
l_{\T}=\frac{(q-f'(q))(1-q)}{q+q^2-f'(q)(2q^2-q+1)},
\end{equation*}
where $q>0$ is the extinction probability of $\T^{\mathsf{good}}_d$.
\end{enumerate}
\end{theorem}
The constant $l_{\T}$ is in the following relation with the constant $\alpha_{\T}$ from Theorem \ref{main-thm-range-GW} in the null recurrent and in the transient case:
\begin{equation*}
2\alpha_{\T}-l_{\T}=1.
\end{equation*}

\subsection{Simulations on Galton-Watson trees}

We present here some simulation data about the growth of the range for a few one
parameter families of Galton-Watson trees, for which the rotor walk is either recurrent
or transient, depending on the parameter. In the table below, the Galton-Watson trees we use in the simulation are presented.

\begin{figure}[h]
\centering
\includegraphics[width=0.75\linewidth]{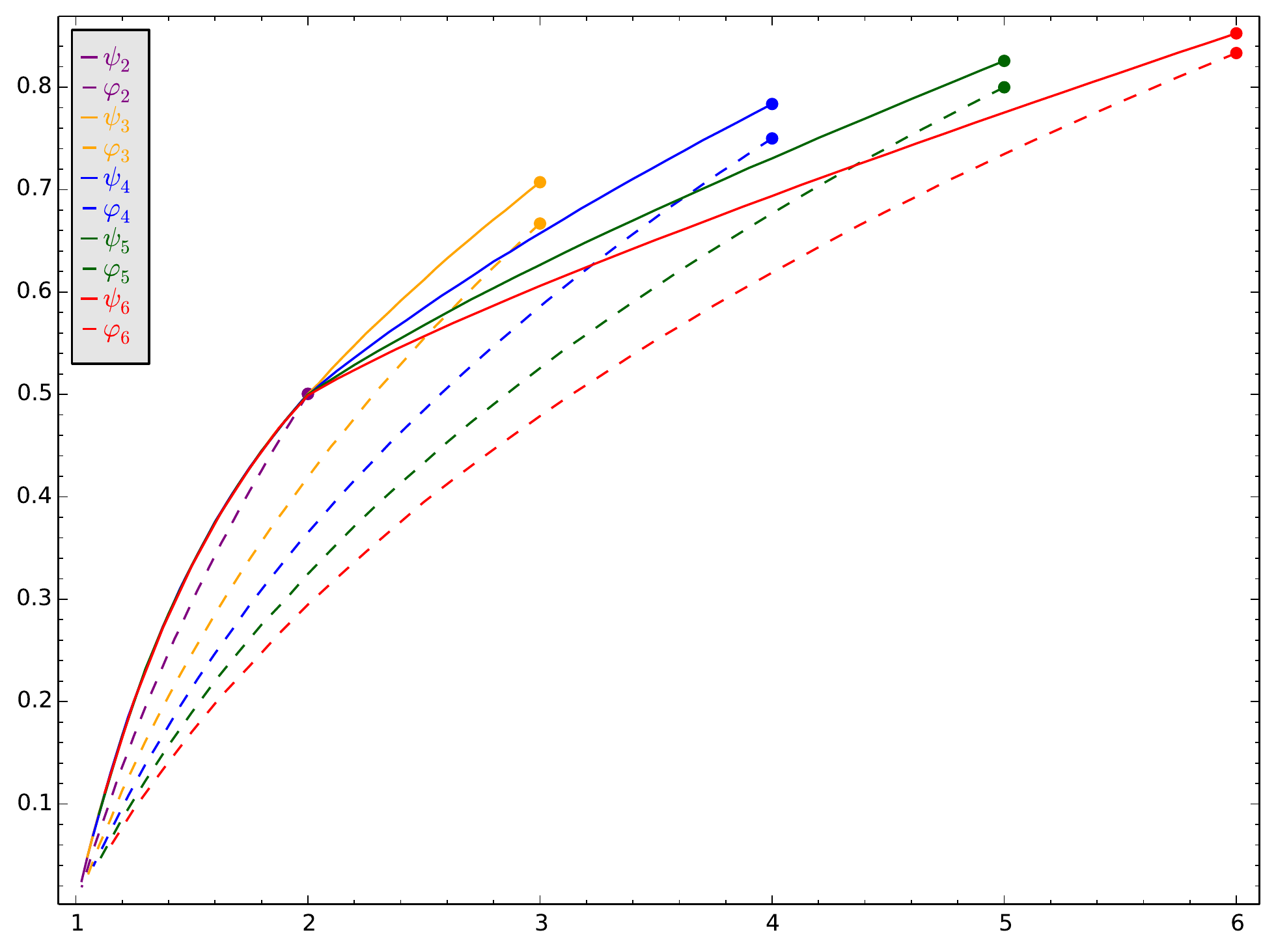}
\caption{\label{fig:range_growth_gw} Plots of the linear growth coefficients of the size of the range for rotor walk (solid lines) and simple random walk (dashed lines) on the Galton-Watson trees $\T_2,\ldots,\T_6$. The $x$-axis depicts the branching number of the tree. The dots show the corresponding values for the regular trees.}
\end{figure} 

\begin{center}
\begin{tabular}{c|c|c}
$\T_i$ & $(p_1,\ldots, p_6)$ & $\mu_i$ \\
\hline
$\T_{2}$ & $ p_1 = p, p_2 = 1-p $ & $ 2-p $  \\
$\T_{3}$ & $ p_1 = p, p_3 = 1-p $ & $ 3-2 p $  \\
$\T_{4}$ & $ p_1 = p, p_4 = 1-p $ & $ 4-3 p $  \\
$\T_{5}$ & $ p_1 = p, p_5 = 1-p $ & $ 5-4 p $  \\
$\T_{6}$ & $ p_1 = p, p_6 = 1-p $ & $ 6-5 p $  
\end{tabular}
\end{center}

For each $i=2,\ldots,6$ and $p\in[0,1)$ denote by $\widetilde{\psi}_i(p)$ and $\widetilde{\varphi}_i(p)$ the simulated values of the limits
\begin{equation}
\lim_{n\to\infty} \frac{| R^{i}_n|}{n} = \widetilde{\psi}_i(p) \qquad \lim_{n\to\infty} \frac{| S^{i}_n|}{n} = \widetilde{\varphi}_i(p),
\end{equation}
where $R_n^{i}$ and $S_n^{i}$ represent the range of the rotor walk and of the simple random walk up to time $n$ on $\T_i$, respectively.
To be able to compare the values we plot the constants $\widetilde{\psi}_i(p)$ and $\widetilde{\varphi}_i(p)$ against the mean offspring number $\mu_i(p)$ of the offspring distribution on $\T_i$.
That is, we look at the functions $\psi_i = \tilde{\psi}_i\circ \mu^{-1}_i$ and $\varphi_i = \tilde{\varphi}_i\circ \mu^{-1}_i$; see Figure \ref{fig:range_growth_gw}.

\begin{appendices}
\section{The contour of a subtree}\label{sec:contour}
We discuss here some further ways one can look at the range of rotor walks on regular trees $\TT_d$, and at their contour functions, which according to simulations seem to have interesting fractal properties.

For $d\geq 2$, let $\Sigma_d = \{0,\ldots,d-1\}$ and denote by $\Sigma_d^\star$ the set of
finite words over the alphabet $\Sigma_d$. We use $\epsilon$ to denote the
empty word.
For $w\in\Sigma_d^\star$ we write $|w|$ for the number of letters in $w$. If $w,v\in\Sigma_d^\star$ we write $wv$ for the concatenation of
the words $w$ and $v$.
We  identify the tree $\widetilde{\TT}_d$ with the set $\Sigma_d^\star$, since every vertex in $\widetilde{\TT}_d$ can be uniquely represented by a word in $\Sigma_d^\star$.
For $w = w_1\dots w_{n-1}w_n \in \Sigma_d^\star\setminus\{\epsilon\}$, the word $w_1\dots w_{n-1}$ is the predecessor of $w$ in the tree, and for all $w\in\Sigma_d^\star$ the children of $w$ are given by the words $w k$ with $k \in \Sigma_d$.
Using the previous notation, for $w=w_1\dots w_{n-1}w_n$ we have $w^{(0)} = w_1\dots w_{n-1}$ and for $k = 0,\ldots,d-1$, $w^{(d-k)} = wk$.
Let $A\subset \widetilde{\TT}_d$ be a finite connected subset (a subtree) of $\widetilde{\TT}_d$ containing the root $\epsilon$. We
identify $A$ with a piecewise constant function $f_A:[0,1]\to {\mathbb{N}}_{\geq 0}$ as following. For each
$x\in[0,1]$, we identify $x$ with the infinite word $x_1x_2\dots$ where the $x_i$ are the digits
expansion of $x$ in base $d$, that is $ x = \sum_{i=1}^\infty x_i d^{-i}$.
We then define the function $f_A$ pointwise as following
\begin{equation}
f_A(x) = \min \{n \geq 1: x_1\dots x_n \not\in A\},
\end{equation}
and we call $f_A$ the \emph{contour} of the set $A$.

\begin{figure}
\includegraphics[width = 0.24\linewidth]{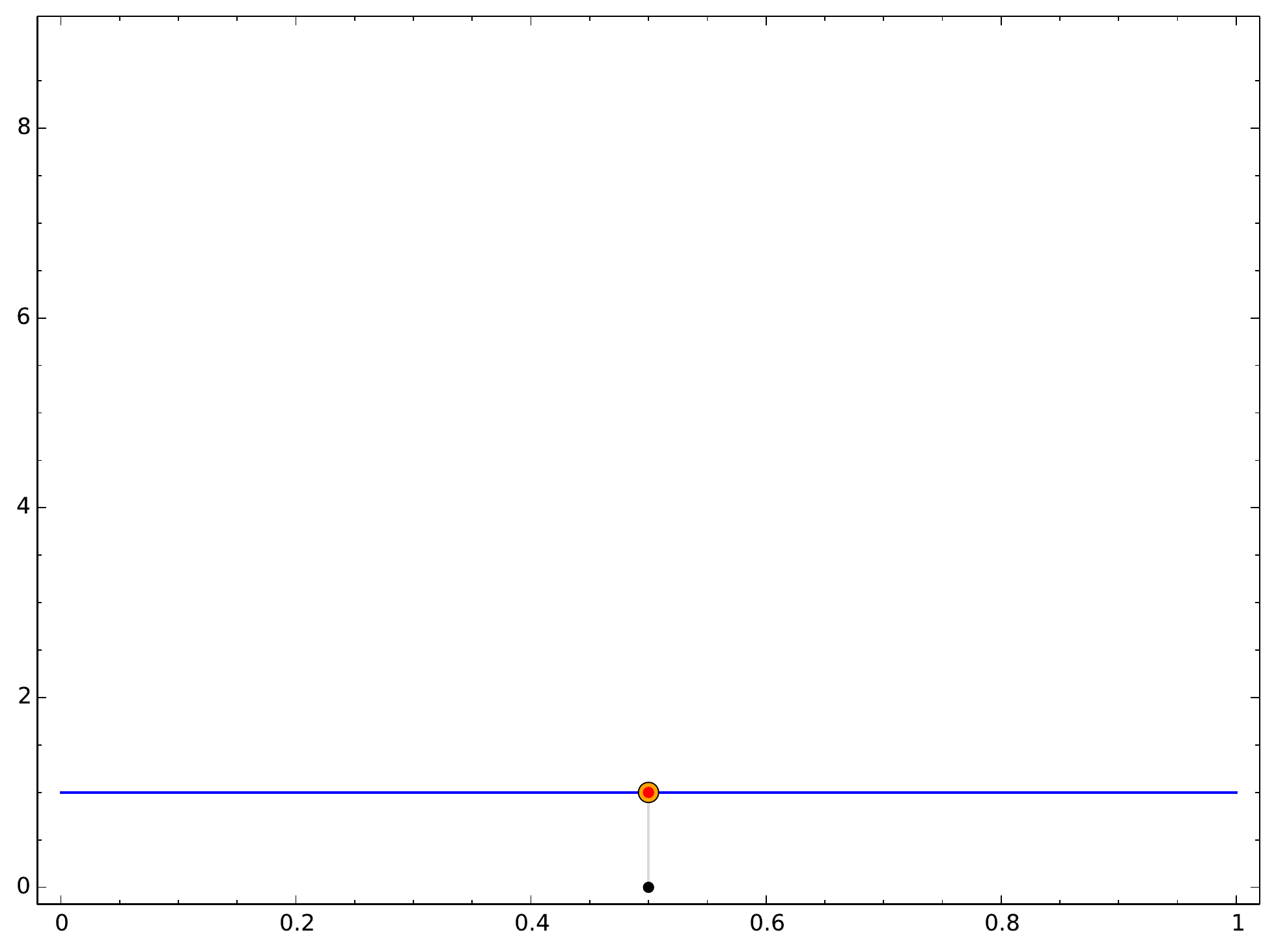}\hfill
\includegraphics[width = 0.24\linewidth]{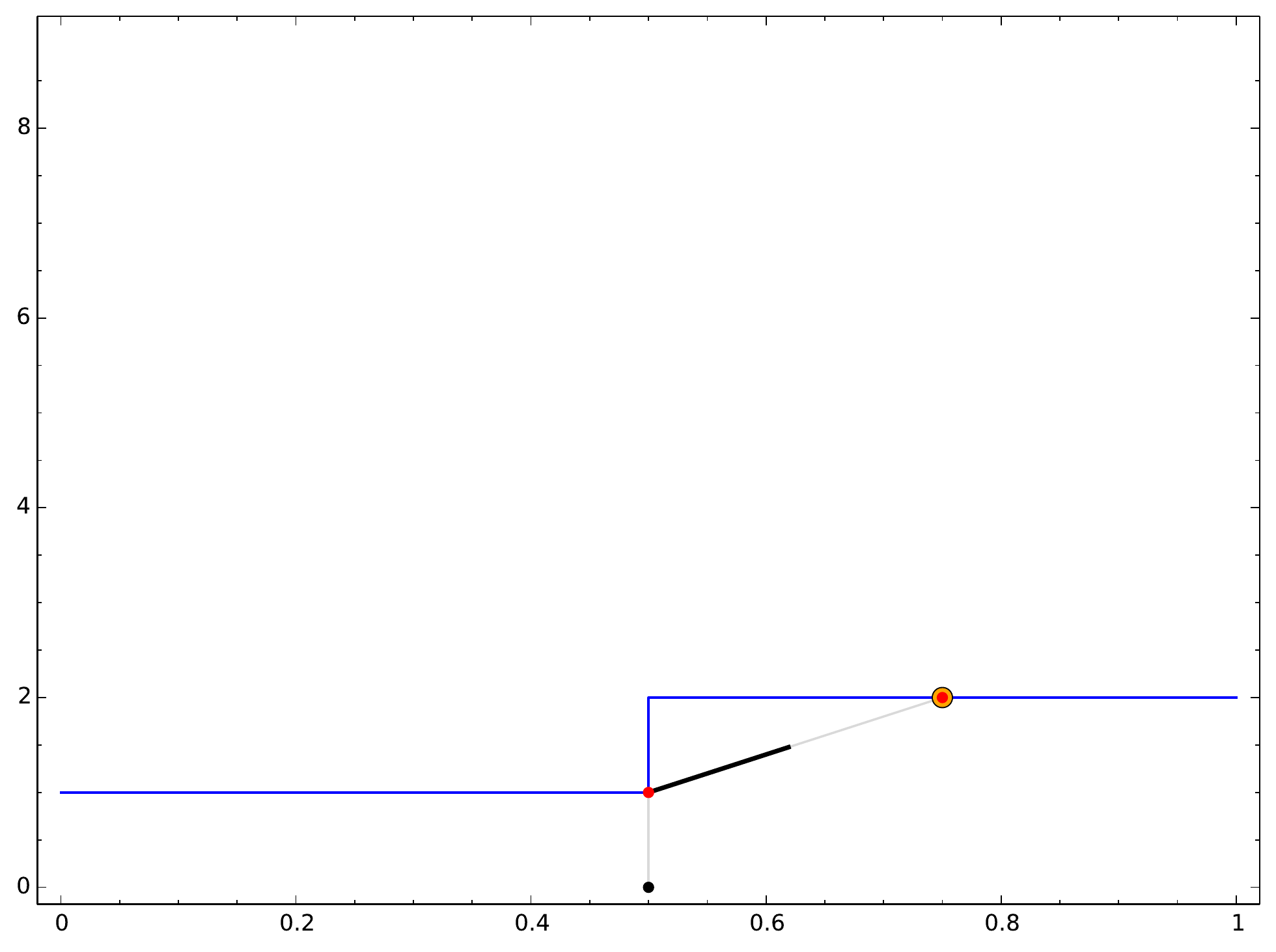}\hfill
\includegraphics[width = 0.24\linewidth]{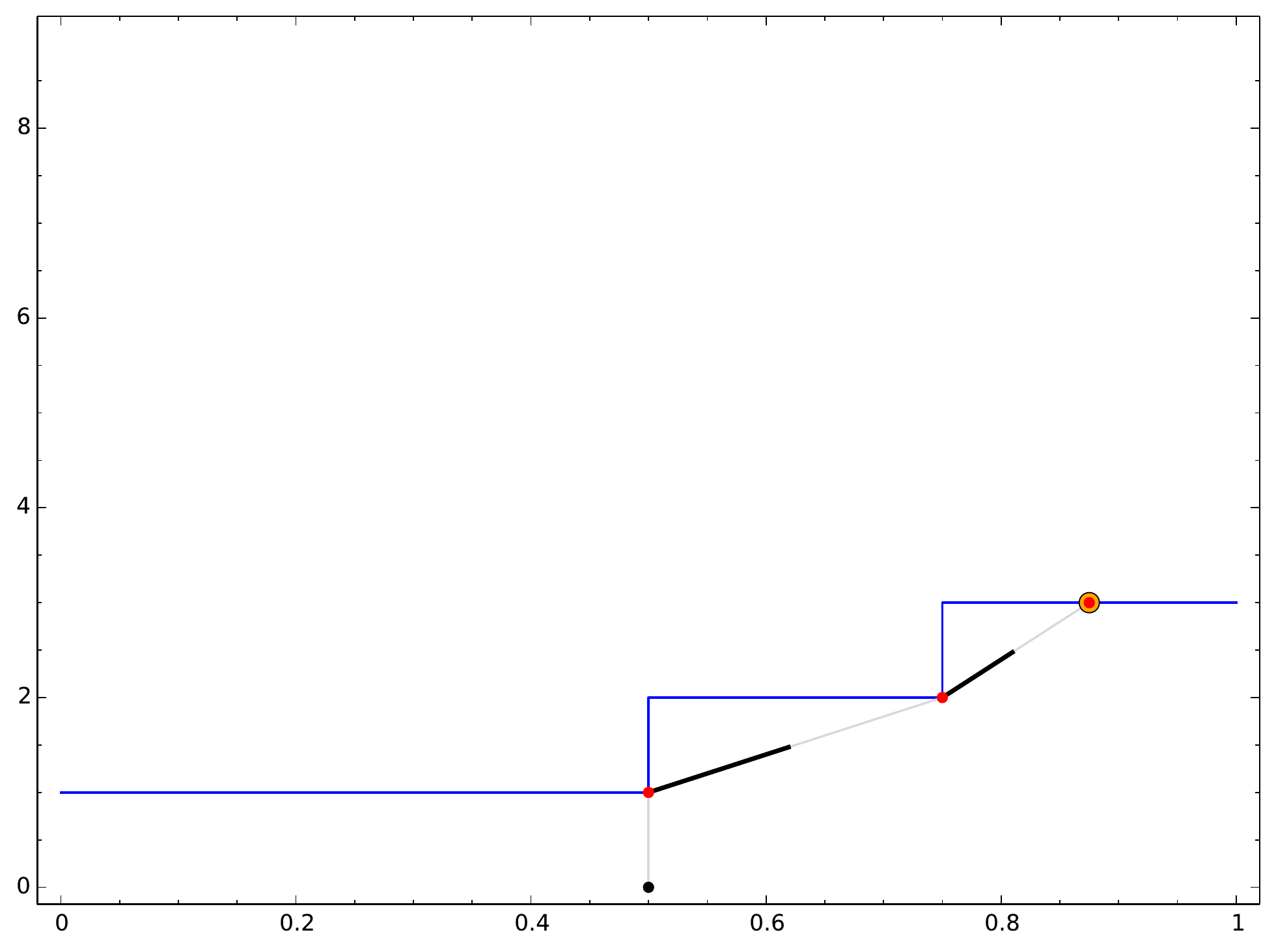}\hfill
\includegraphics[width = 0.24\linewidth]{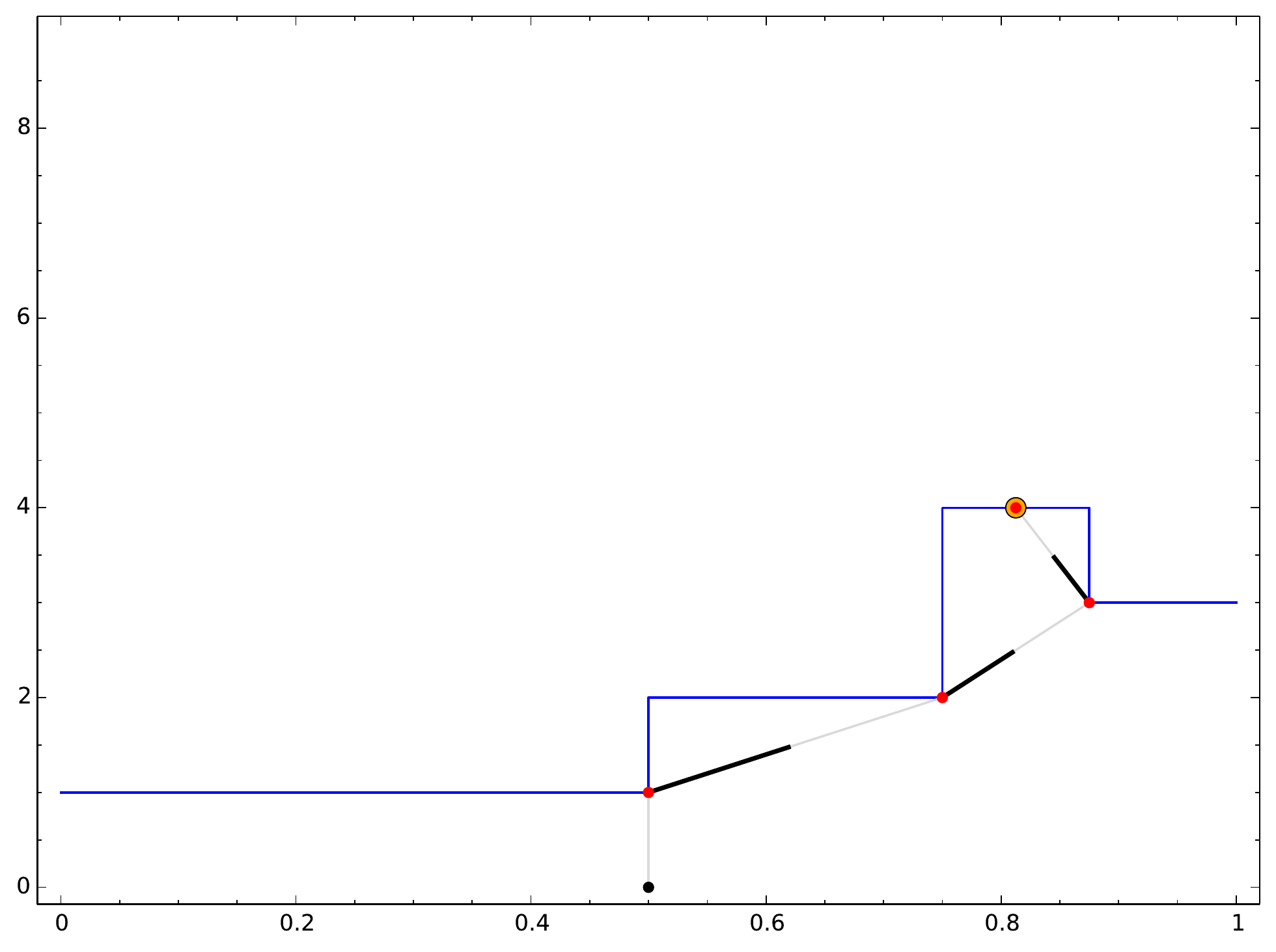}\\[1ex]
\includegraphics[width = 0.24\linewidth]{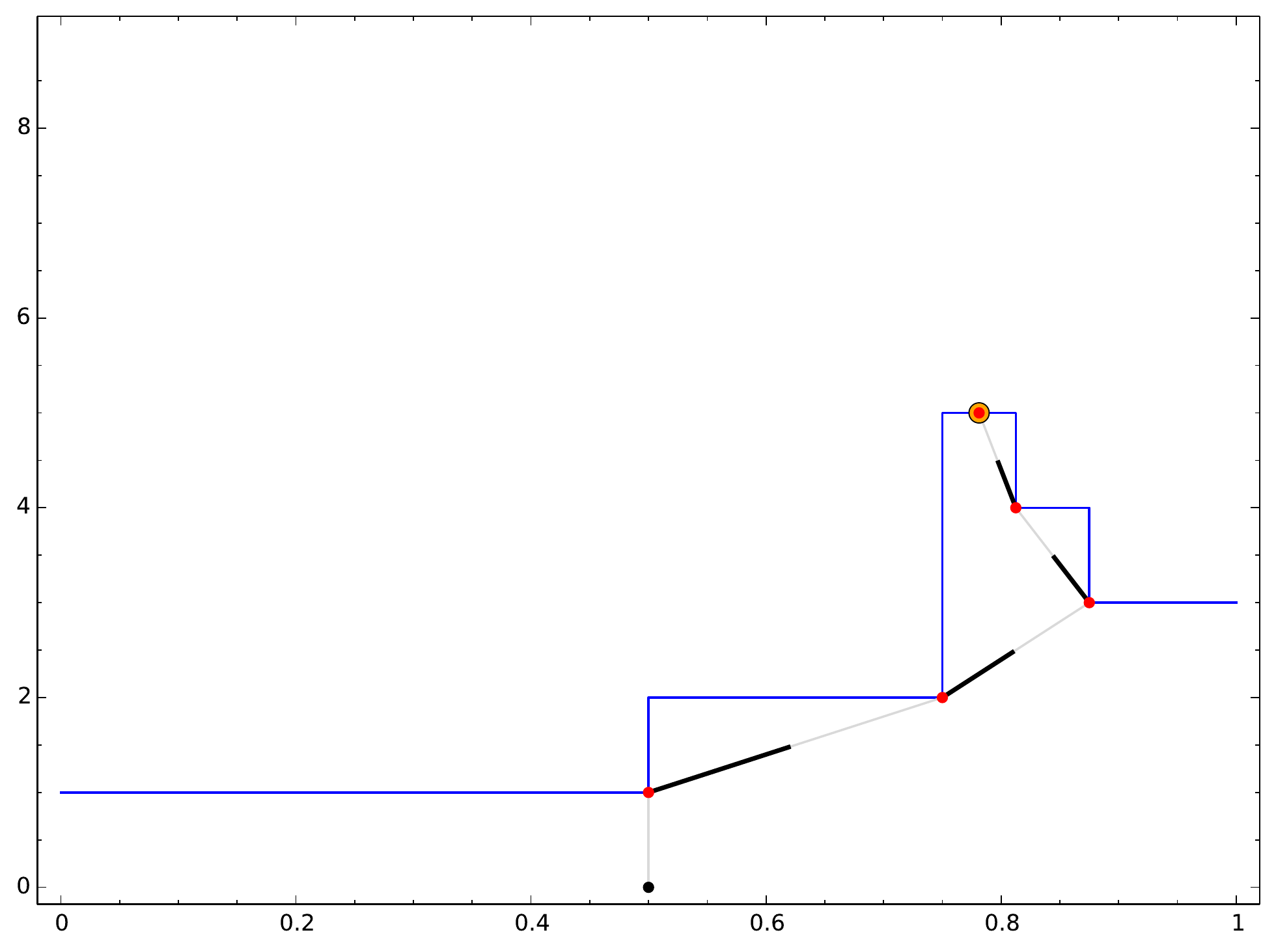}\hfill
\includegraphics[width = 0.24\linewidth]{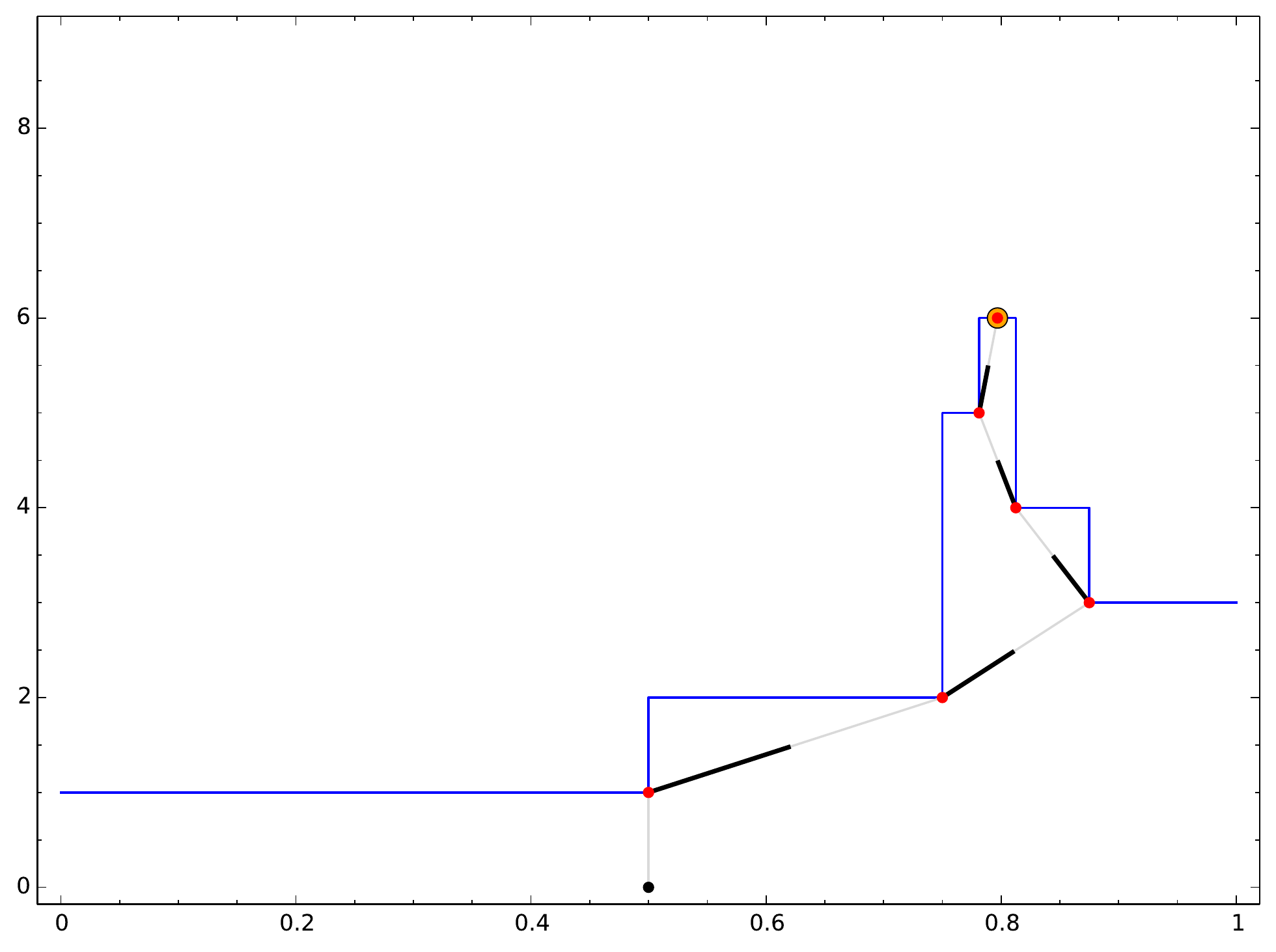}\hfill
\includegraphics[width = 0.24\linewidth]{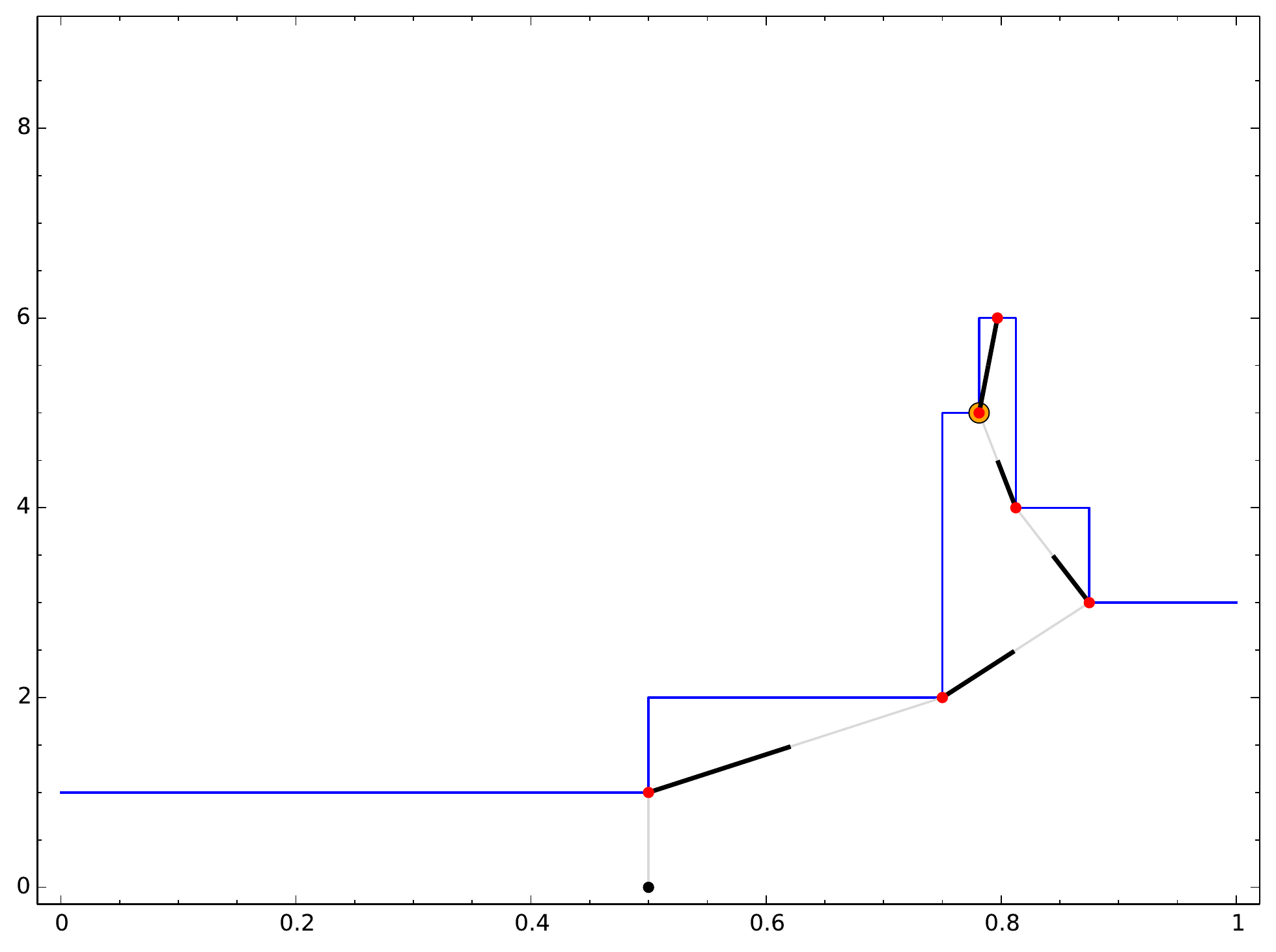}\hfill
\includegraphics[width = 0.24\linewidth]{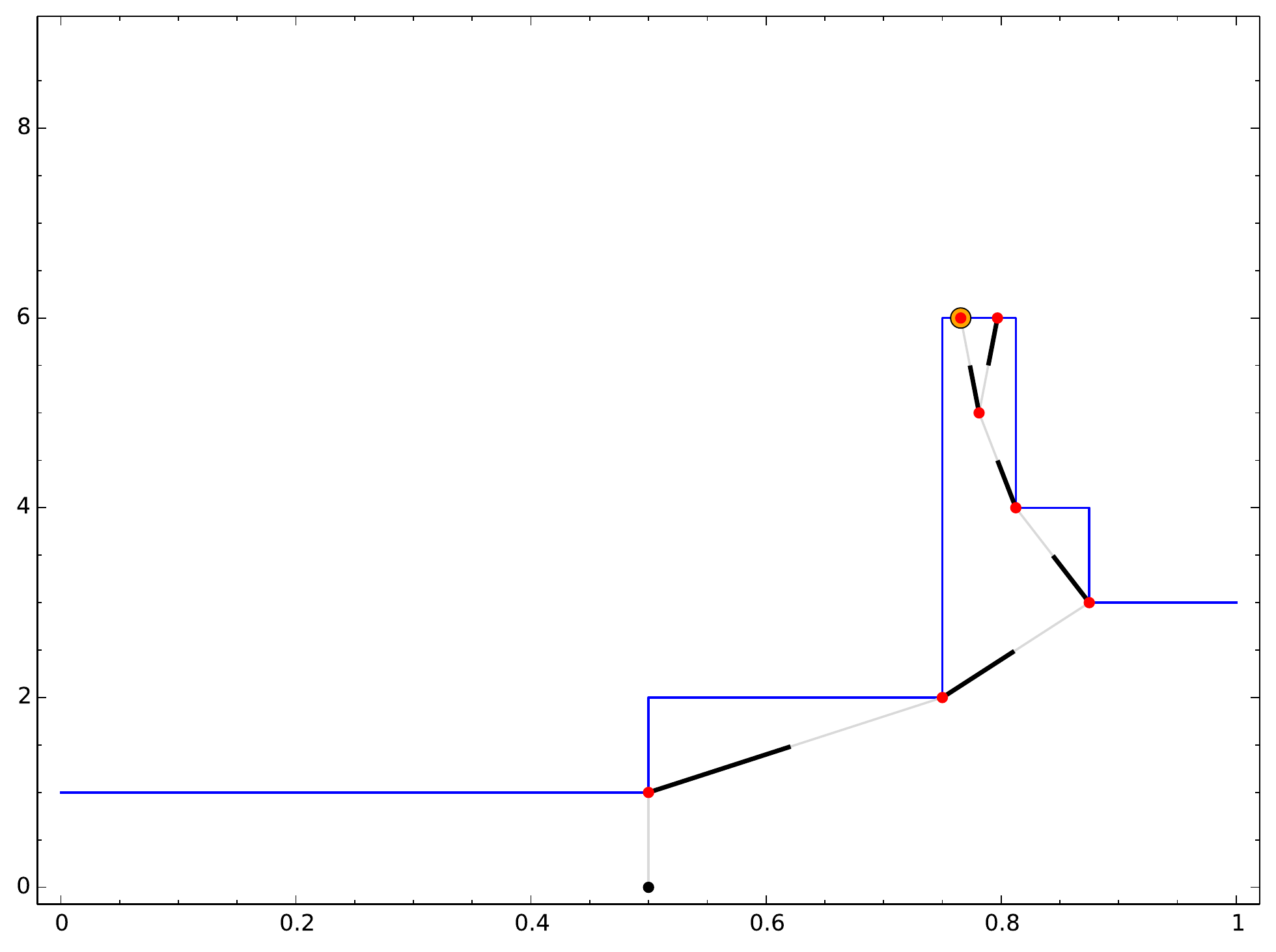}\\[1ex]
\includegraphics[width = 0.24\linewidth]{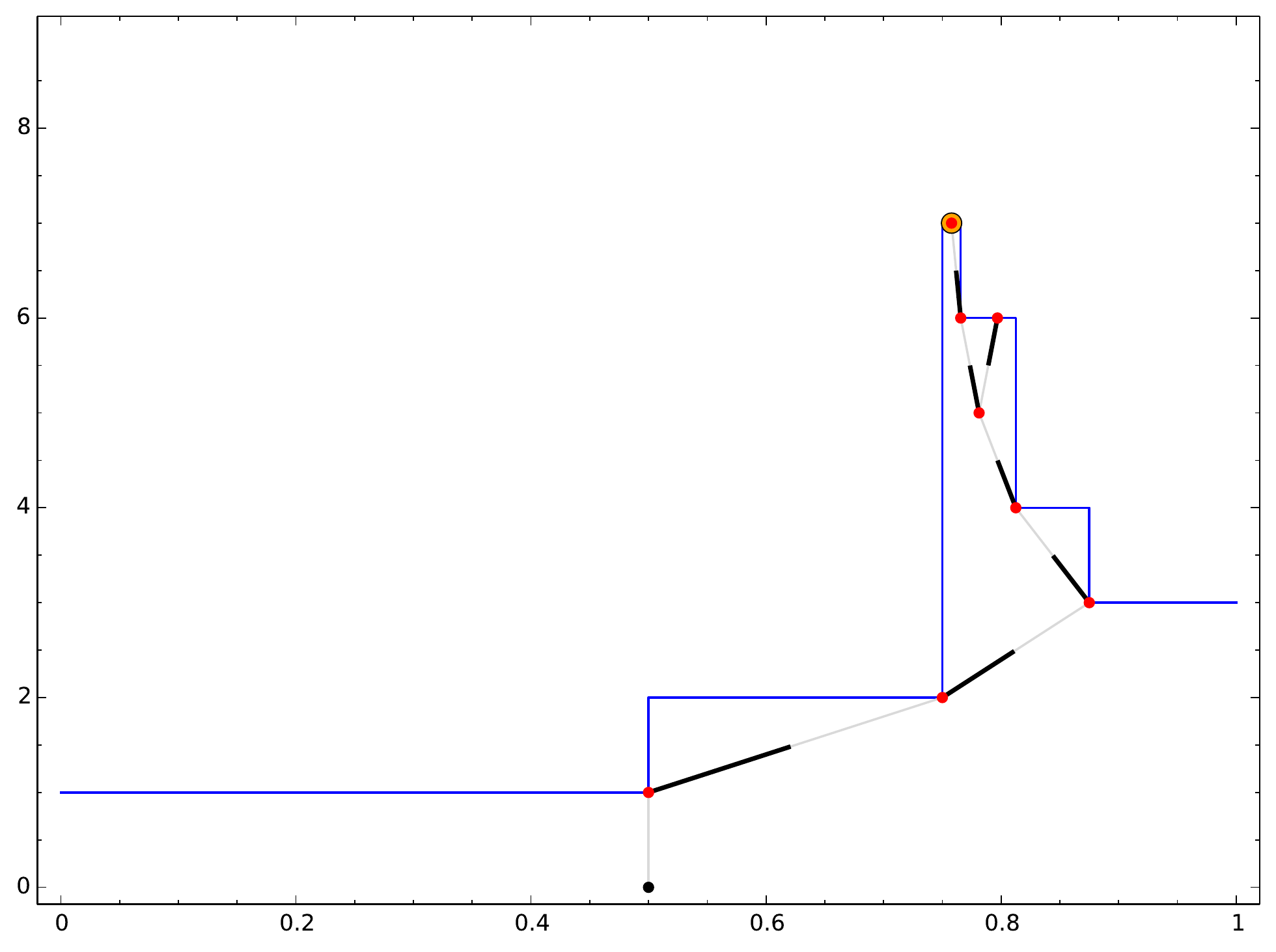}\hfill
\includegraphics[width = 0.24\linewidth]{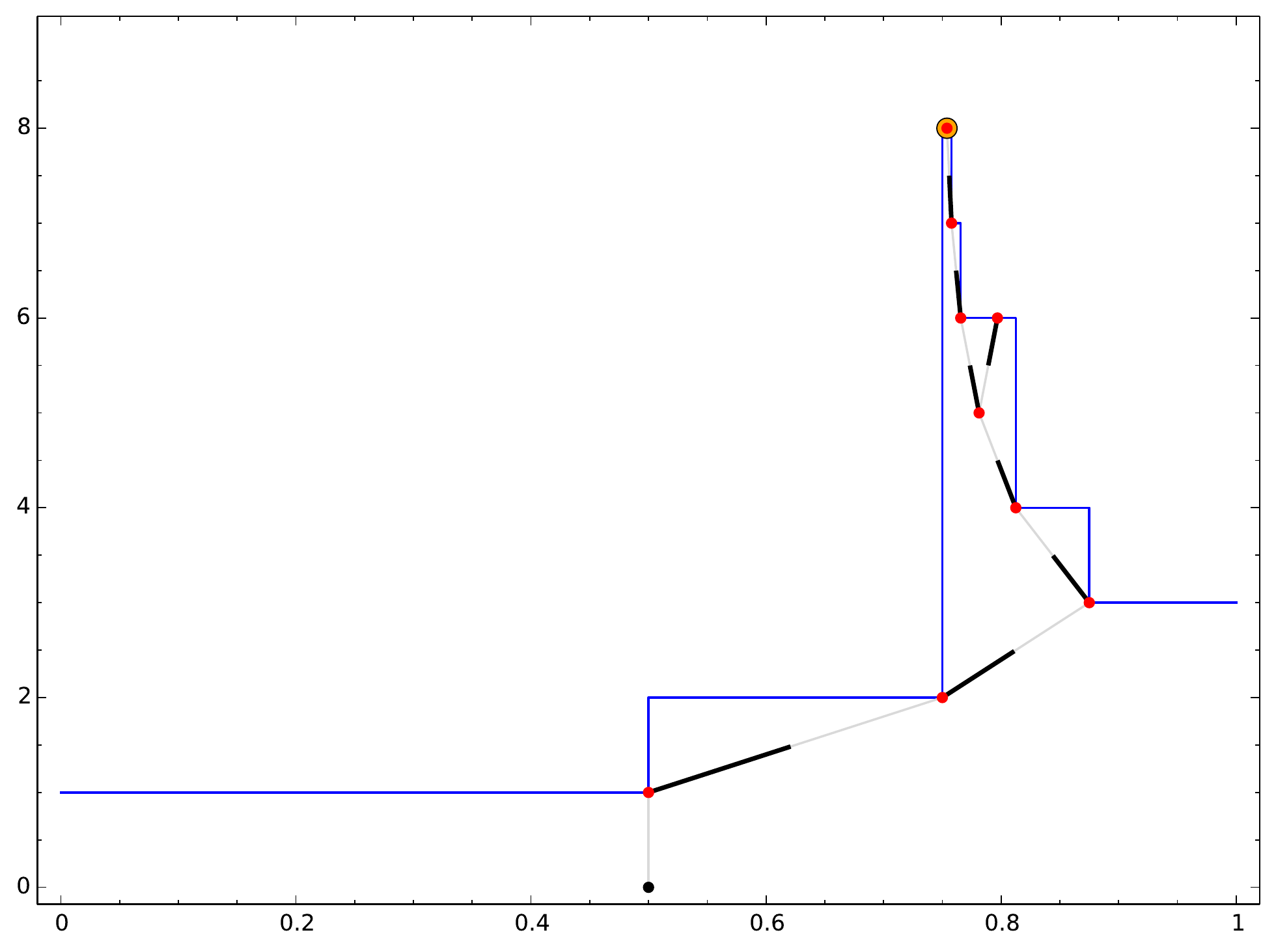}\hfill
\includegraphics[width = 0.24\linewidth]{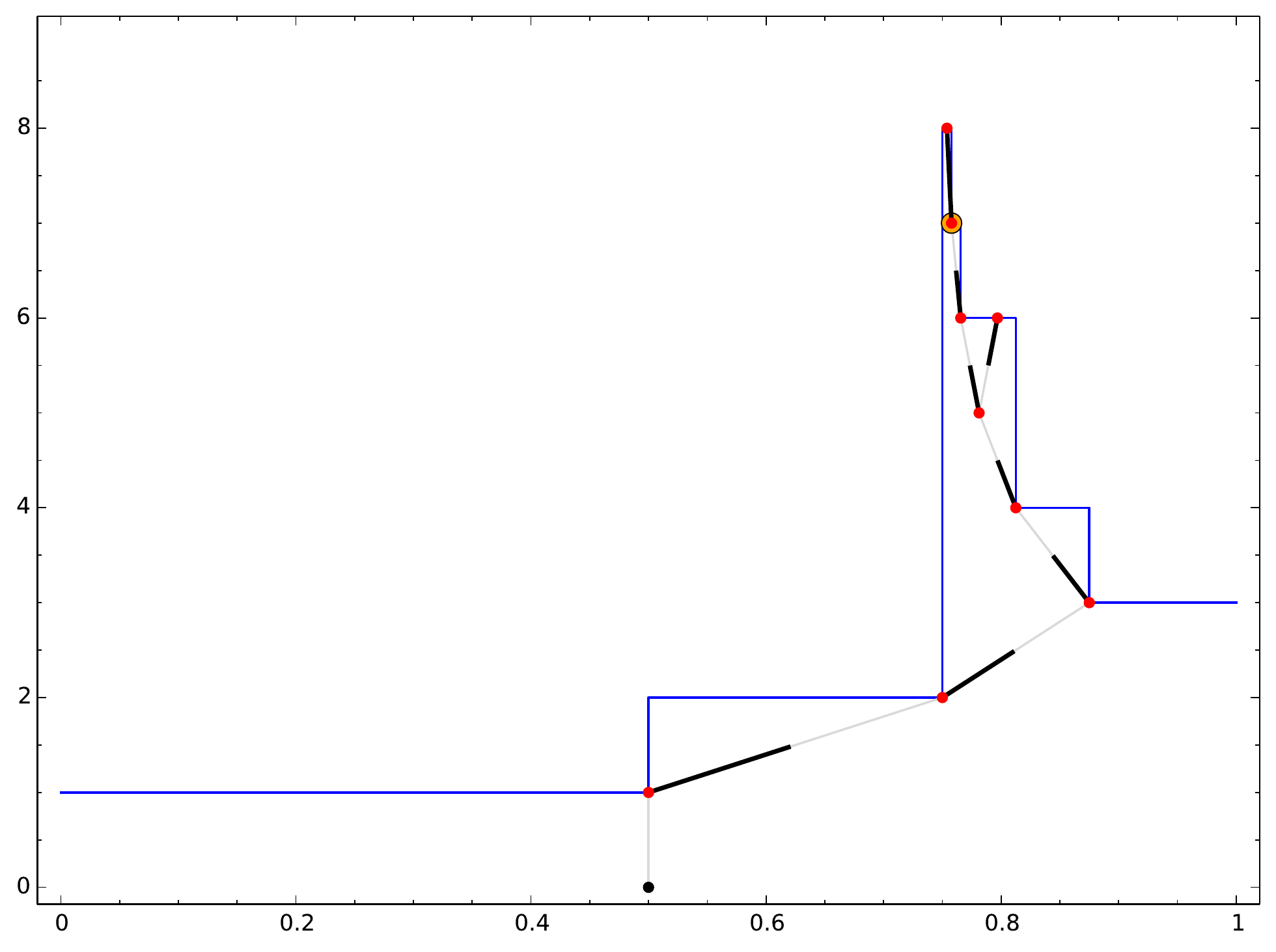}\hfill
\includegraphics[width = 0.24\linewidth]{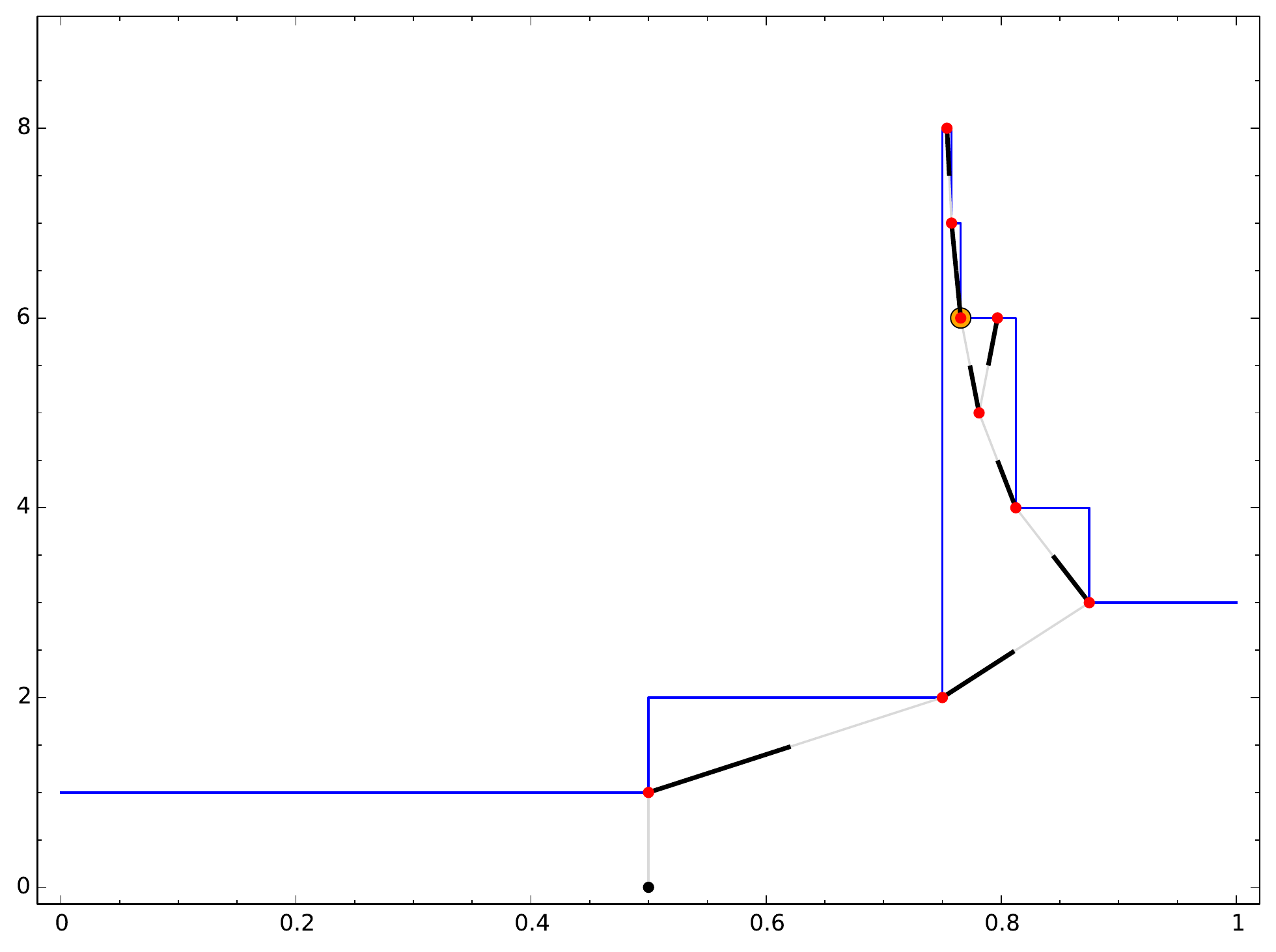}\\[1ex]
\includegraphics[width = 0.24\linewidth]{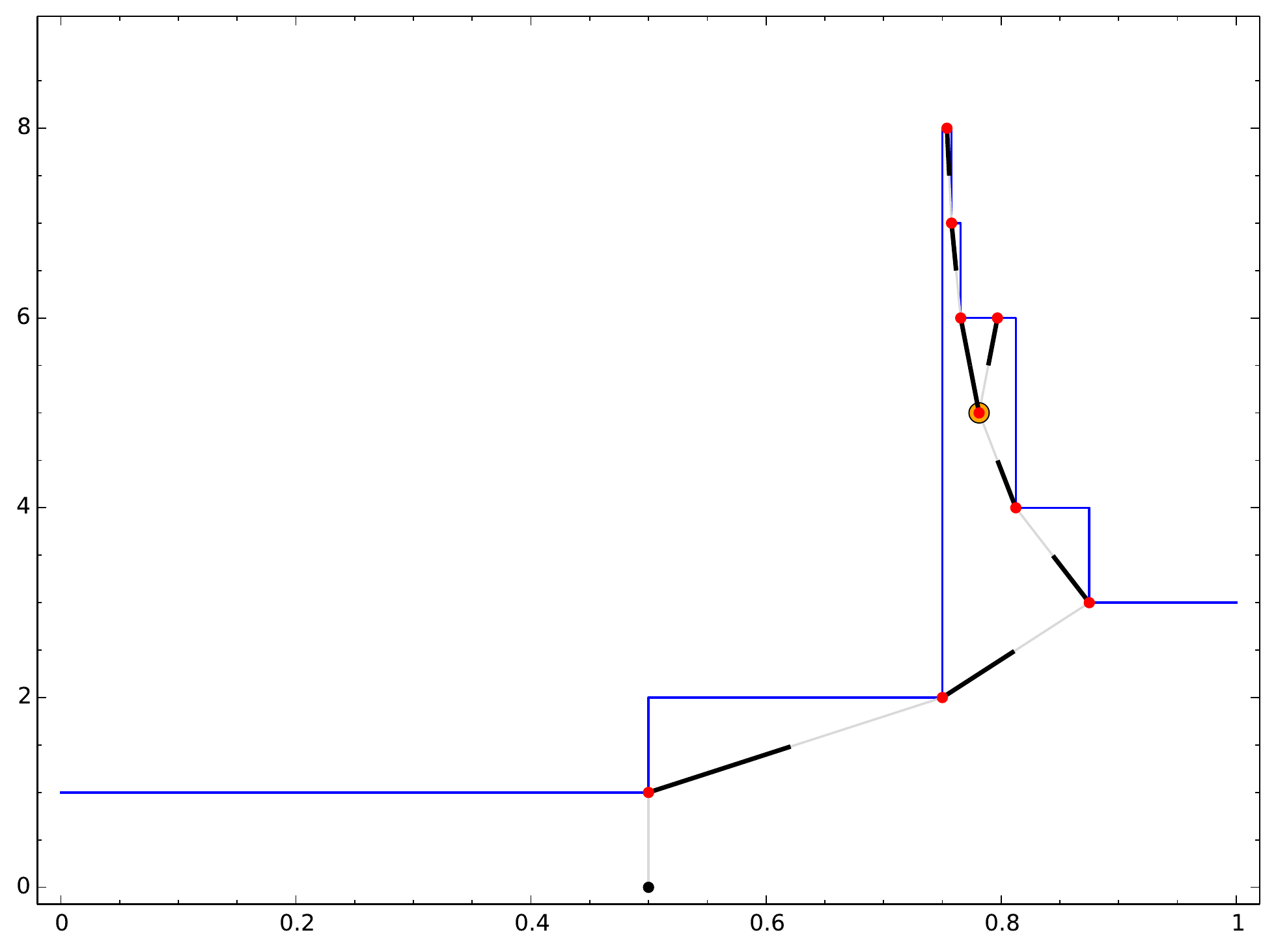}\hfill
\includegraphics[width = 0.24\linewidth]{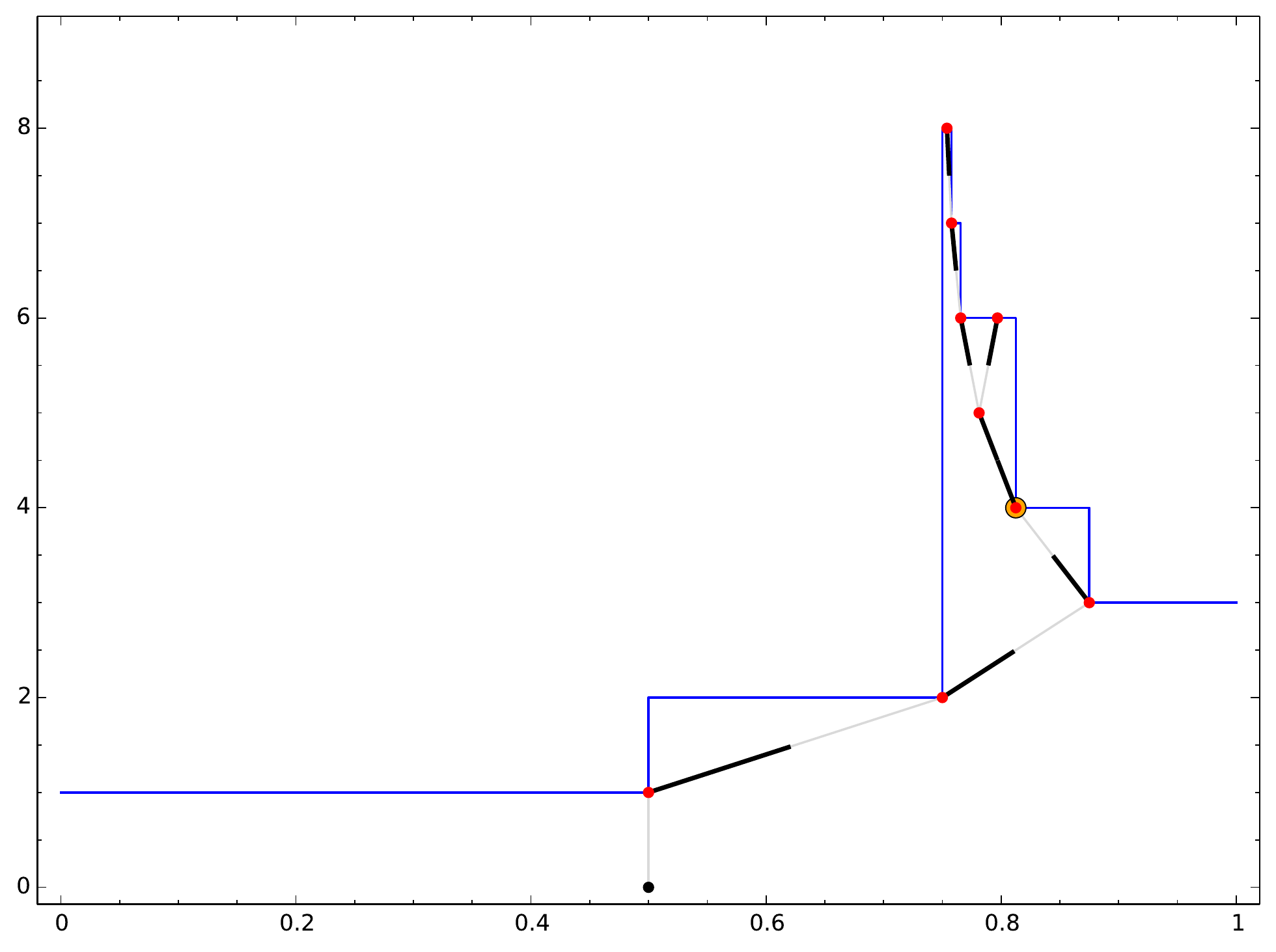}\hfill
\includegraphics[width = 0.24\linewidth]{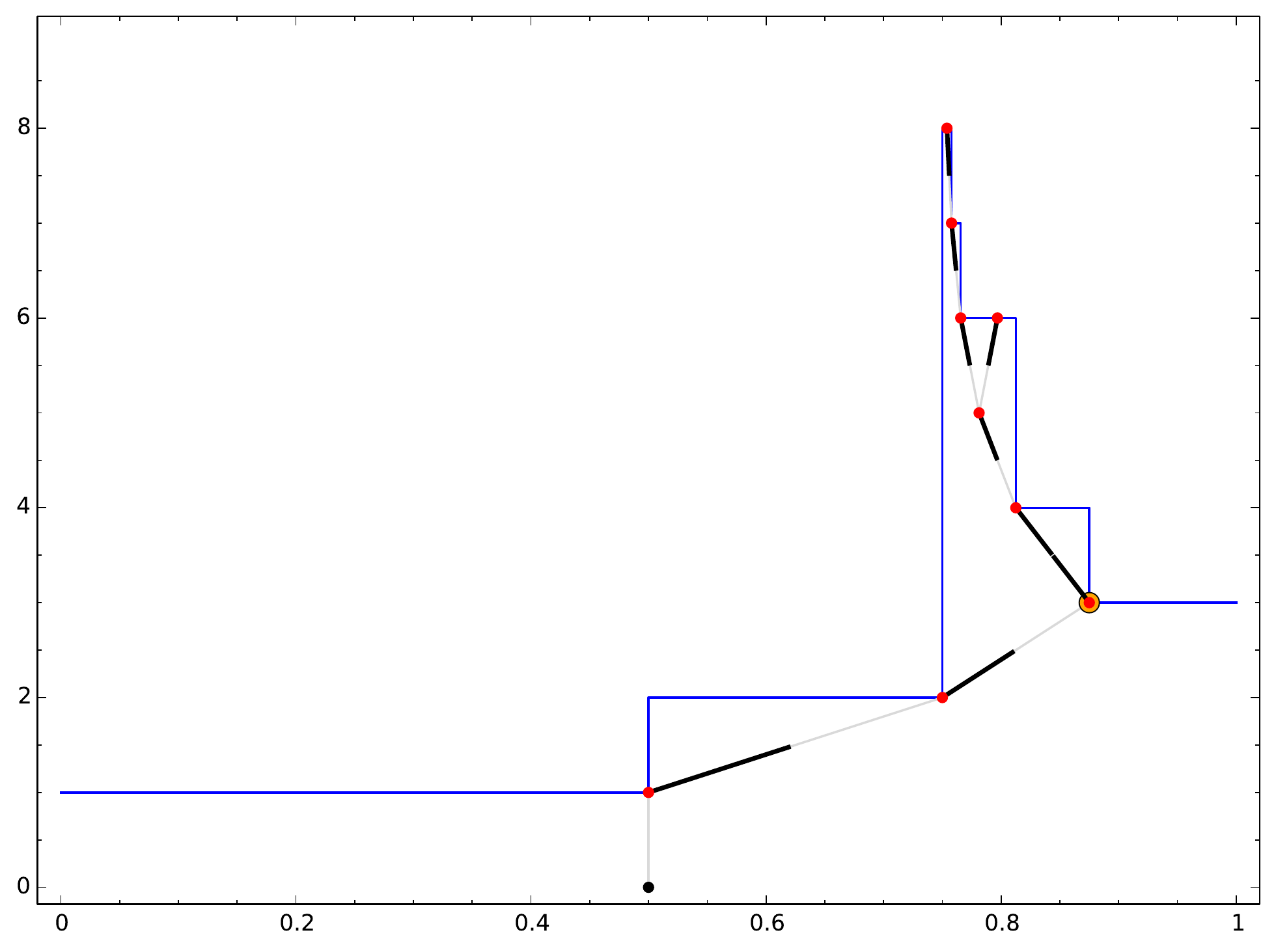}\hfill
\includegraphics[width = 0.24\linewidth]{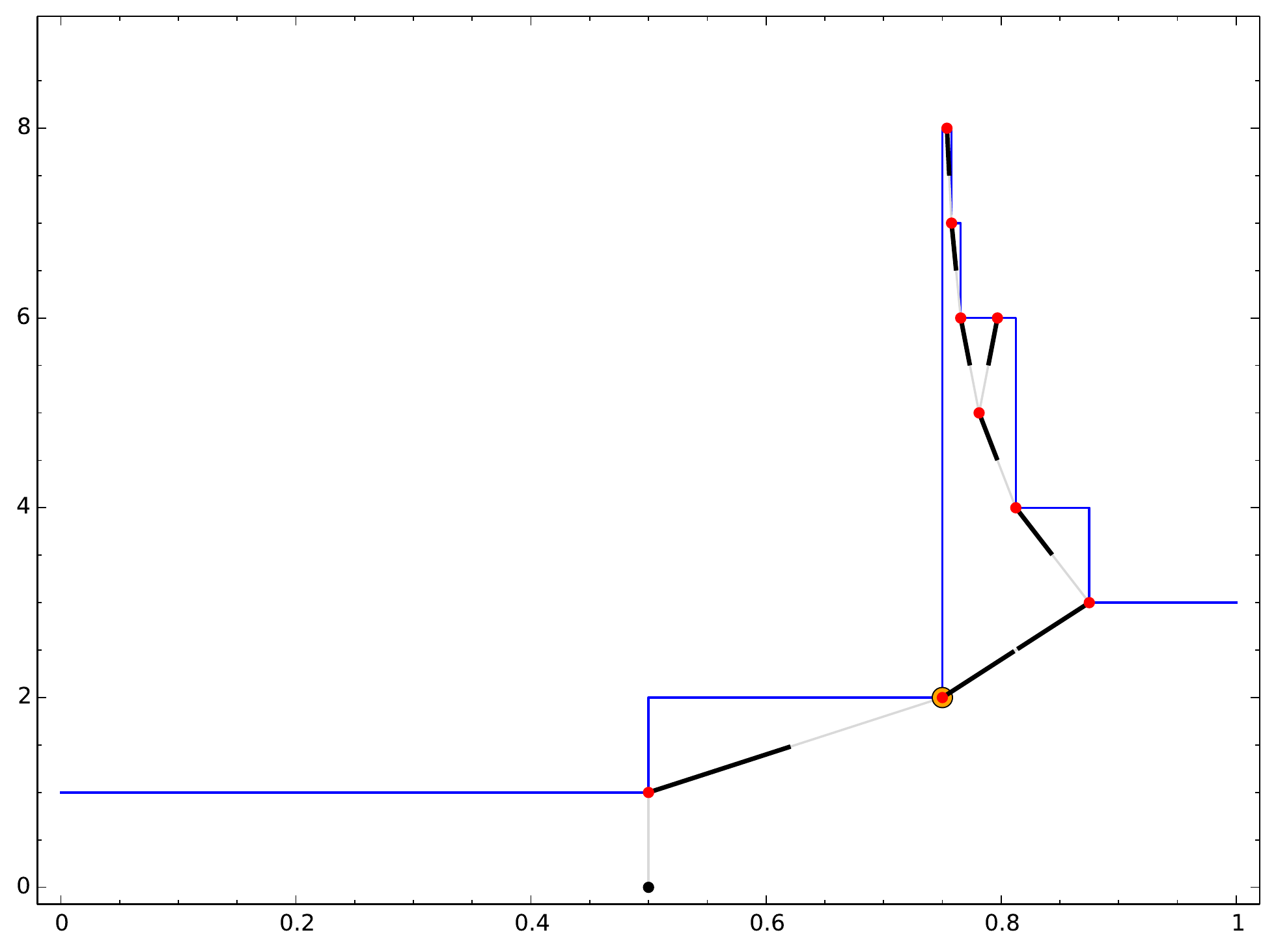}\\[1ex]
\includegraphics[width = 0.24\linewidth]{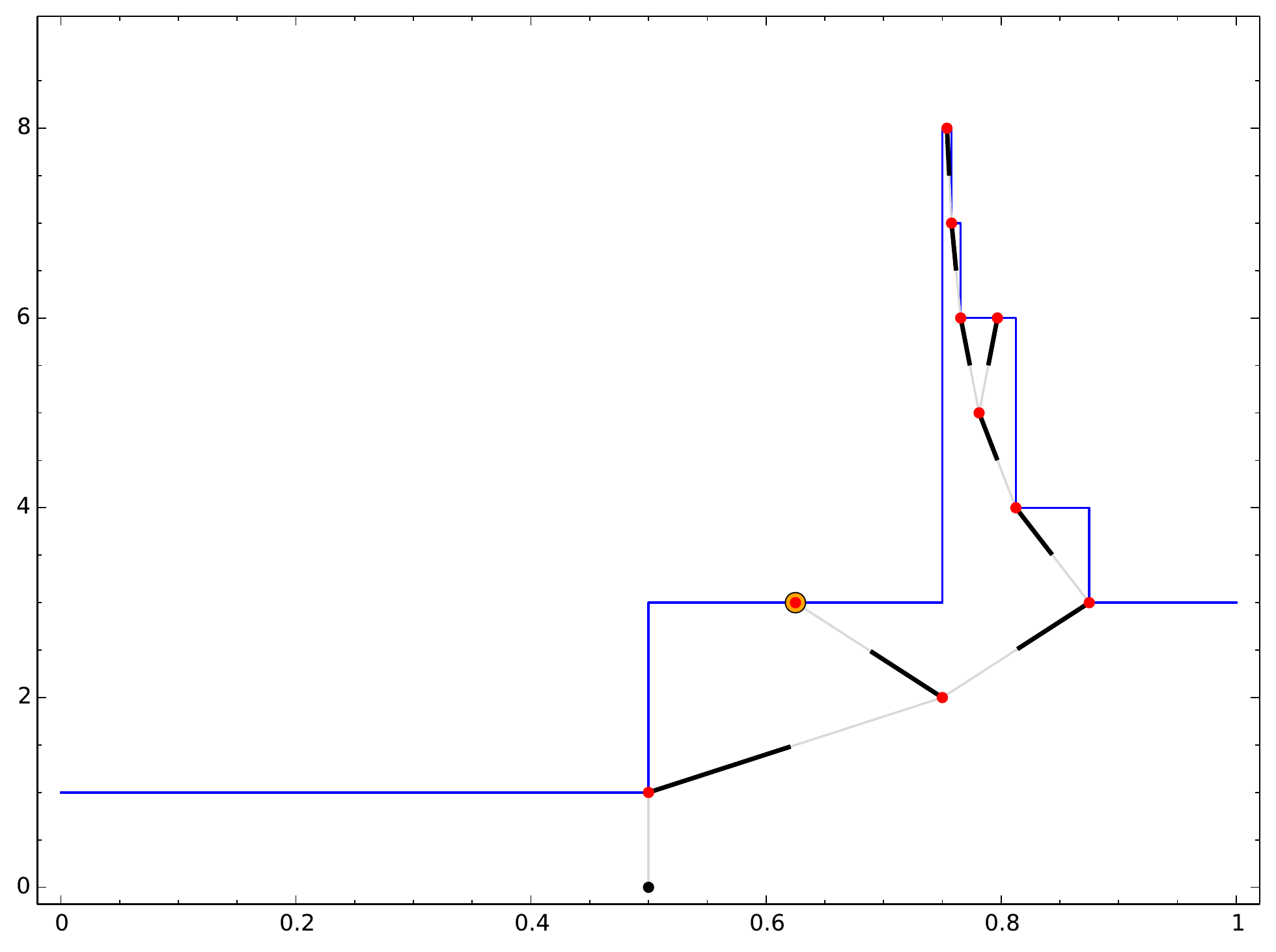}\hfill
\includegraphics[width = 0.24\linewidth]{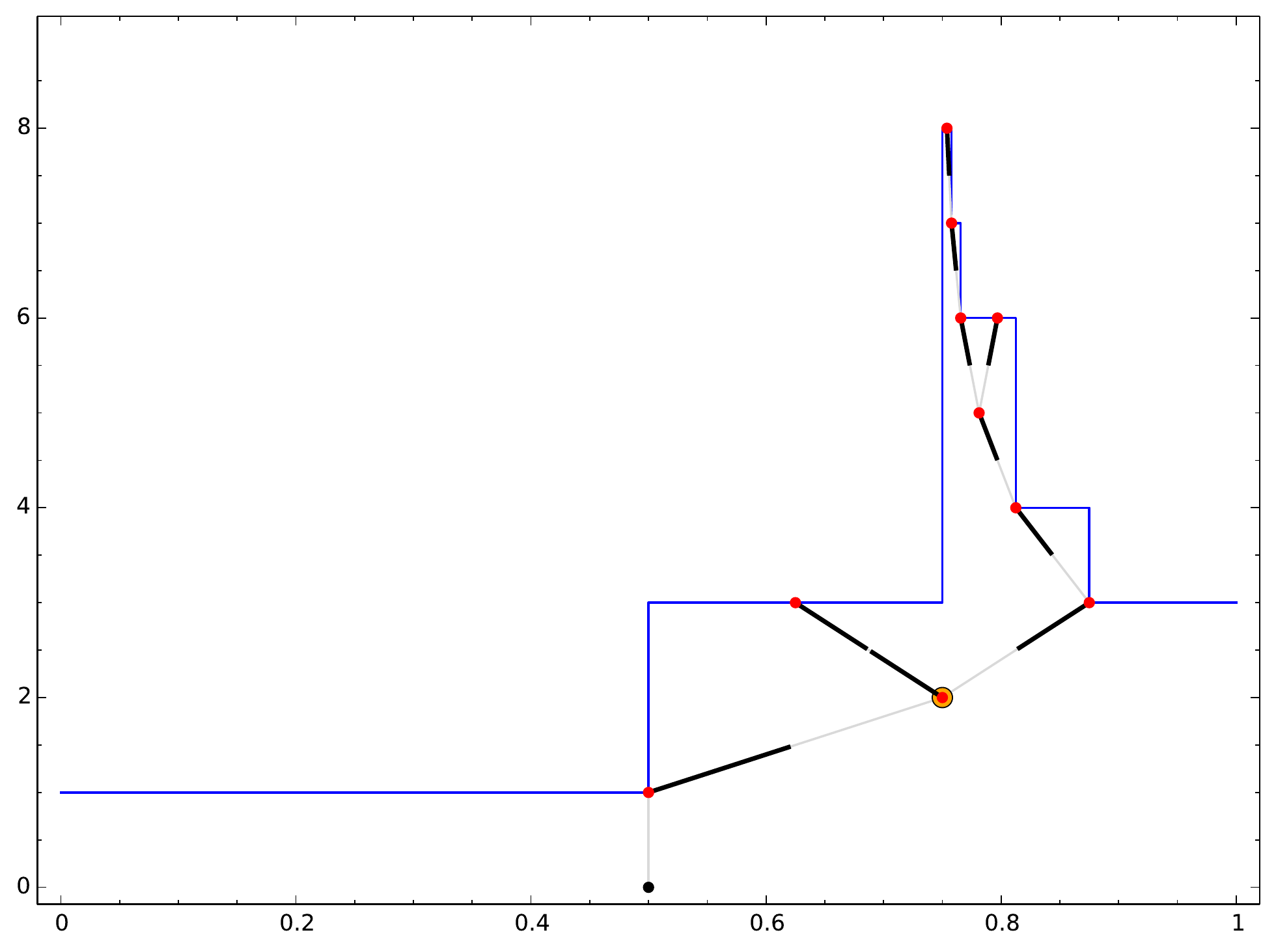}\hfill
\includegraphics[width = 0.24\linewidth]{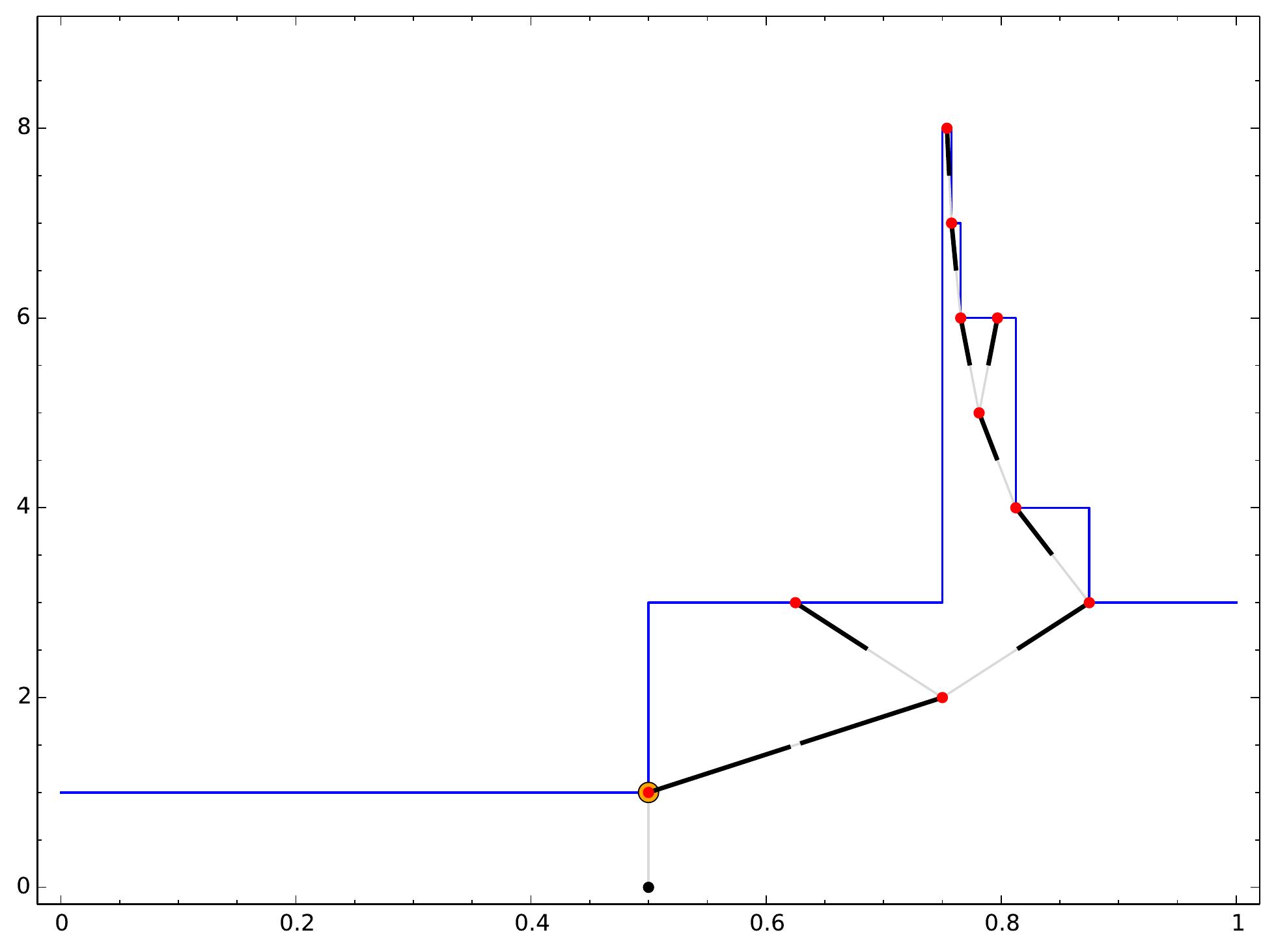}\hfill
\includegraphics[width = 0.24\linewidth]{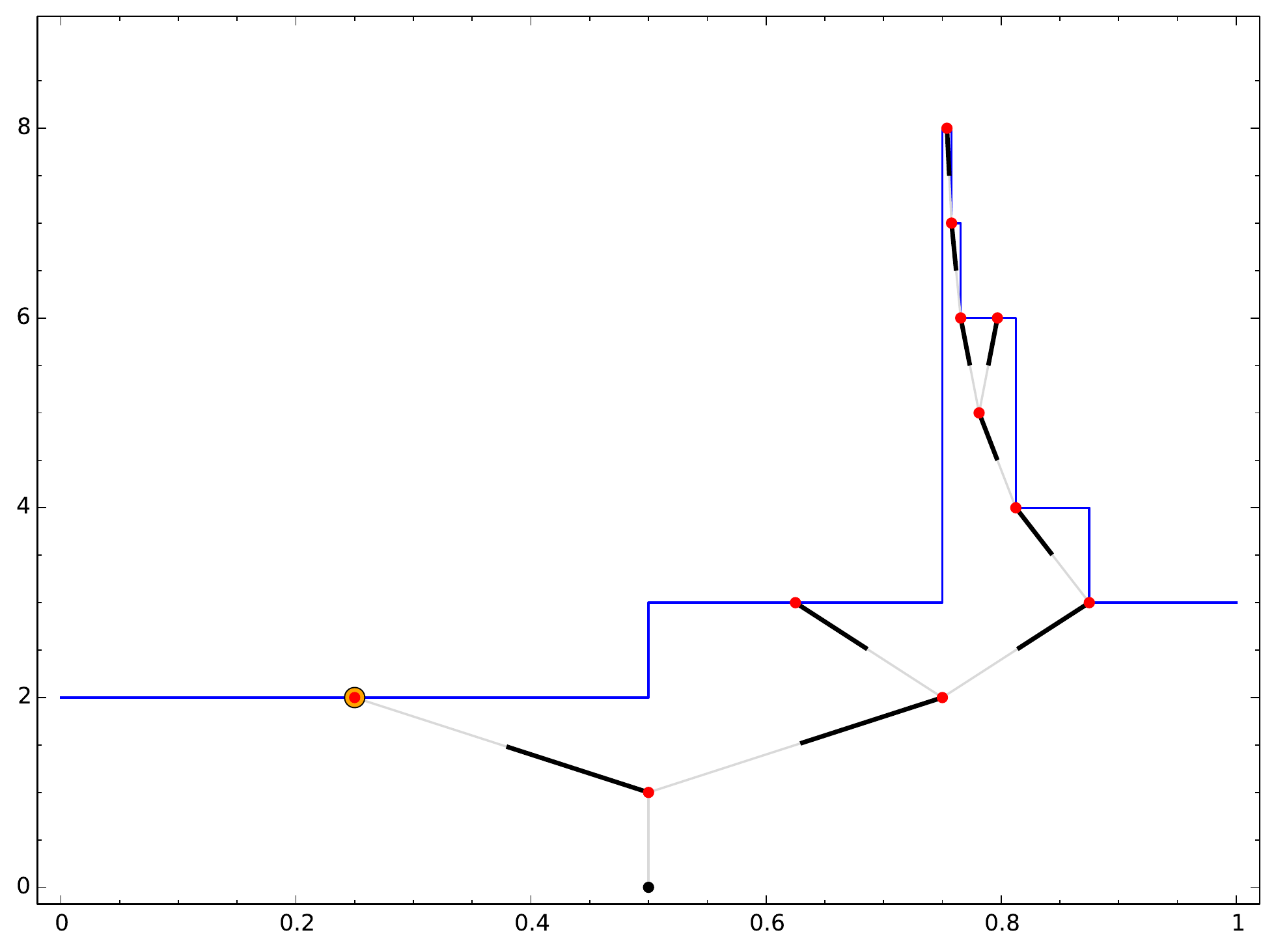}
\caption{\label{fig:contour_first_steps}The first few steps in the computation of the range of rotor walk on $\TT_2$.
The blue functions are the contour functions of the current range $f_{R_n}$. The short black
lines represent the current rotor configuration.}
\end{figure}

\subsection{Contour of the range: recurrent case}

The range $R_n$ of the rotor walk $(X_n)$ on $\TT_{d}$ is a subtree of $\TT_{d}$. In what follows, we write $f_n = f_{R_n}$ for the contour of the range of the rotor walk up to time $n$. Since
we start with a random initial rotor configuration, $f_n$ is a random càdlàg-function. See
Figure \ref{fig:contour_first_steps} for the contours of the rotor walk range for
the first few steps of the process on the binary tree $\TT_2$. Figure \ref{fig:sample_binary_range_10000} shows a typical contour of the rotor range on the
binary tree for $n = 10000$ steps. Figure \ref{fig:expected_binary_range_10000} shows a
numerical approximation of the expectation $\E[f_{10000}]$. 

As in the proofs of the law of large numbers for $|R_n|$, we look first at the times when the rotor walk returns to the root. Recall the definition of the return times $(\tau_k)$, as defined in \eqref{eq:tauk-rec}, for recurrent rotor walks. Write again
$\R_k = R_{\tau_k}$ for the range up to time $\tau_k$ of the $k$-th return of the rotor walk $(X_n)$ to the sink, and denote by $g_k(x) = \E\big[f_{\R_k}(x)\big]$ the expected contour after the $k$-th excursion, that is 
$g_k(x) = \sum_{m=1}^\infty m \Pb\big[f_{\R_k}(x) = m\big]$. Recall that we are in a case of a random initial configuration $\rho$ on $\TT_d$, with $\E[\rho(v)]\geq 1$. The distribution of $\rho$ is $\Pb\big[\rho(v) = i\big] = r_i\geq 0$.
Some additional notation will be needed. For $i = 0,\ldots,d$ let
\begin{equation*}
p_i = \sum_{j=0}^{i-1} r_j, \qquad q_i = \sum_{j = i}^d r_j.
\end{equation*}

For each $x = (x_1,x_2,\ldots) \in [0, 1]$ we can now compute the probability $\Pb\big[f_{\R_1}(x) = m+1\big]$ that the
rotor walk visits the first $m$ vertices of the ray represented by $(x_1,x_2,\ldots)$
before taking a step back towards the sink vertex. Once the rotor walk makes a step towards the sink, it cannot further explore the ray $x$ without first returning to the sink
vertex. Furthermore, the depth the walk can explore the ray before returning depends only on the
initial rotor state along the vertices of the ray. Below, $f_{\R_1}(x) = 1$ means that $x_1$ is not in the range of the walk after the first full excursion, $x_1$ being a vertex at level $1$. We get the following

\begin{align*}
\Pb\big[f_{\R_1}(x) = 1\big] &= \left.\begin{cases}
r_1+r_2+\dots +r_d&\quad\text{if } x_1 = d-1, \\
r_2+\dots +r_d&\quad\text{if } x_1 = d-2, \\
&\quad\vdots\\
r_d&\quad\text{if } x_1 = 0
\end{cases}\right\} = q_{(d-x_1)}.\\[2ex]
\Pb\big[f_{\R_1}(x) = 2\big] &= \left.\begin{cases}
r_0 & \quad \text{if } x_1 = d-1, \\
r_0+r_1 & \quad \text{if } x_1 = d-2, \\
        & \quad \vdots \\
r_0+r_1+\dots+r_{d-1} & \quad \text{if } x_1 = 0
\end{cases}\right\}\\
&\;\;\cdot
\left.\begin{cases}
r_1+r_2+\dots +r_d&\quad\text{if } x_2 = d-1, \\
r_2+\dots +r_d&\quad\text{if } x_2 = d-2, \\
&\quad\vdots\\
r_d&\quad\text{if } x_2 = 0 
\end{cases}\right\} \\
& = p_{(d-x_1)} q_{(d-x_2)}
\end{align*}

\begin{figure}[h]
\begin{subfigure}[t]{0.49\linewidth}
\includegraphics[width = \linewidth]{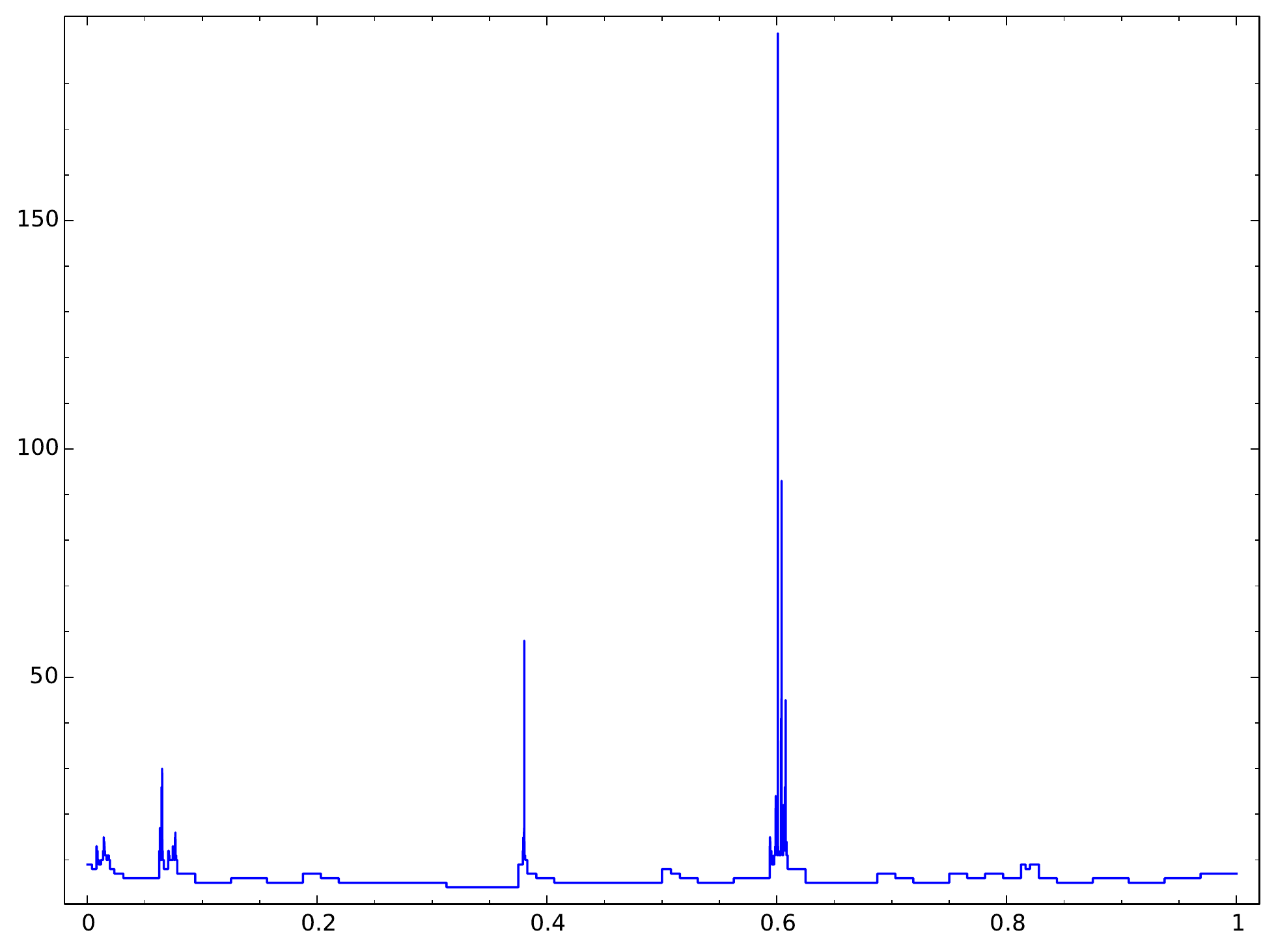}
\caption{\label{fig:sample_binary_range_10000} }
\end{subfigure}
\hfill
\begin{subfigure}[t]{0.49\linewidth}
\includegraphics[width = \linewidth]{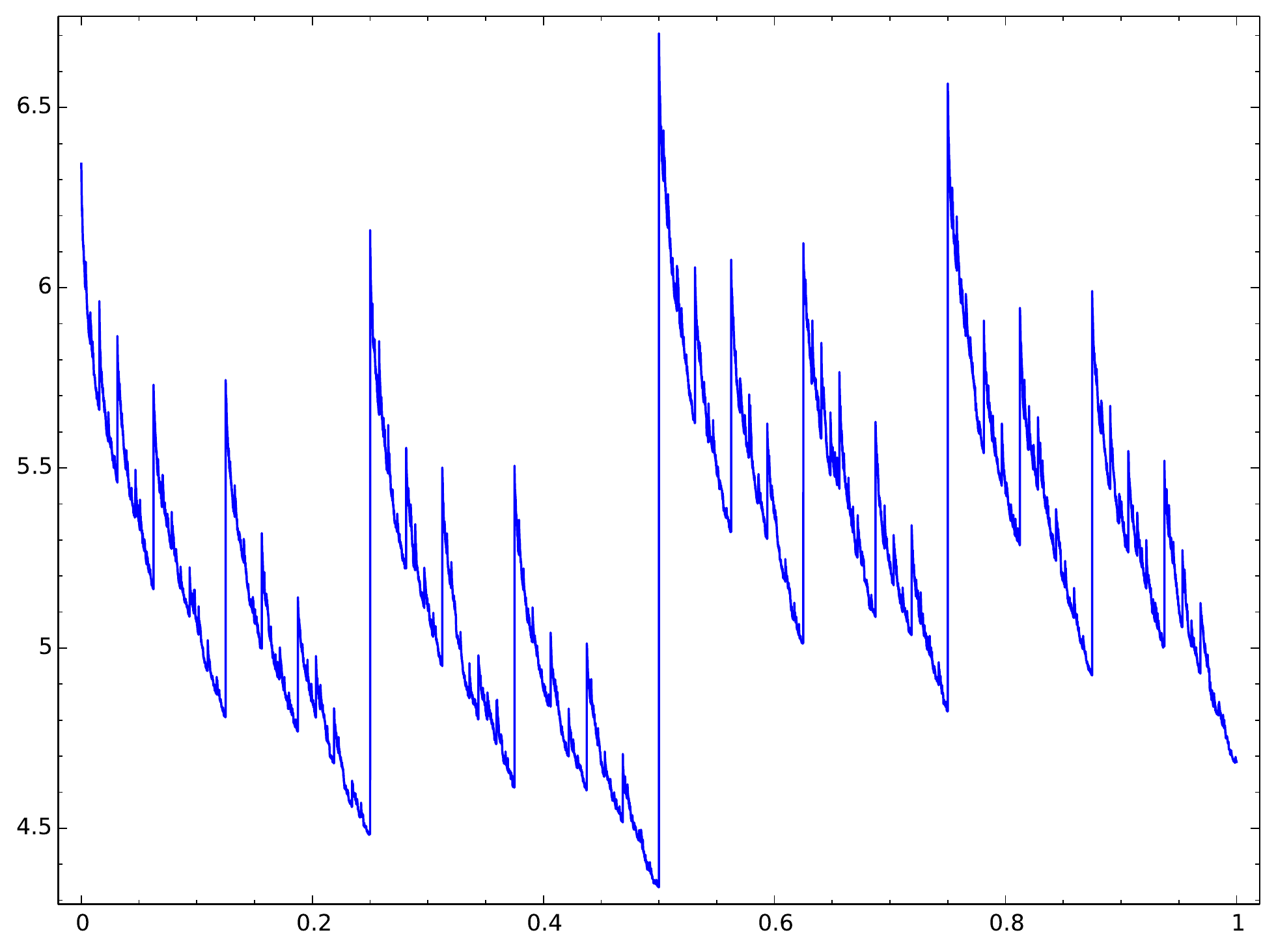}
\caption{\label{fig:expected_binary_range_10000}}
\end{subfigure}
\caption{\label{fig:binary_range_10000}}(a) Typical range contour of the uniform
rotor walk on the binary tree $\TT_2$. The rotor walk performed $10000$ steps.
(b) Average over $1000$ samples of the contour $f_{10000}$ on the binary tree.
\end{figure}

For the general case $m\geq 1$ we get
\begin{equation}\label{eq:fR1}
\Pb\big[f_{\R_1}(x) = m\big] = \left(\prod_{i = 1}^{m-1} p_{(d-x_i)}\right)\cdot q_{(d-x_m)}.
\end{equation}
After completing the $(m-1)$-st excursion all rotors in the visited set $\R_{m-1}$ point
towards the sink $o$. Hence before completing the $m$-th excursion the walk will visit all
boundary points of $\R_{m-1}$. At each boundary point $w\in\partial \R_{m-1}$, the rotor walk makes a full excursion into the subtree rooted at $w$ before exploring the rest of the
boundary and finally returning to $o$ to complete the excursion. Thus the $m$-th excursion 
depends only on the set $\R_{m-1}$ and on the random initial states of the rotors on  $\TT_d \setminus \R_{m-1}$.

For $x = (x_1,x_2,\dots) \in [0,1]$ denote by $\overleftarrow{x}^l = \{d^l x\} = (x_{l+1},x_{l+2},\dots)$, where $\{\bullet\}$ is the fractional part of a positive real number.
For $k\geq 2$ we have $\Pb\big[f_{\R_k}(x) = 1\big]=0$ since
between each return the tree is explored at least for one additional level.
When $k\geq 2$ and $m\geq 2$:
\begin{align}
\label{eq:f_convolution}
\Pb\big[f_{\R_k}(x) = m\big] = \sum_{l=1}^{m-1} \Pb\big[f_{\R_1}(x) = l \big]\cdot\Pb\big[f_{\R_{k-1}}(\overleftarrow{x}^l) = m-l\big].
\end{align}
For $k\geq 2$ we have
\begin{align*}
g_k(x) &= \E\big[f_{\R_k}(x)\big] =  
          \sum_{m=1}^\infty m \Pb\big[f_{\R_k}(x) = m\big]\\
       &= \sum_{m = 1}^\infty m \sum_{l=1}^{m-1}
          \Pb\big[f_{\R_1}(\overleftarrow{x}^l) = m-l\big] \cdot
          \Pb\big[f_{\R_{k-1}}(x) = l\big] \\
       &= g_{k-1}(x) + \sum_{l=1}^\infty g_1(\overleftarrow{x}^l) \Pb\big[f_{\R_{k-1}}(x) = l\big].
\end{align*}

\begin{figure}
\begin{subfigure}[b]{0.49\linewidth}
\includegraphics[width=\linewidth]{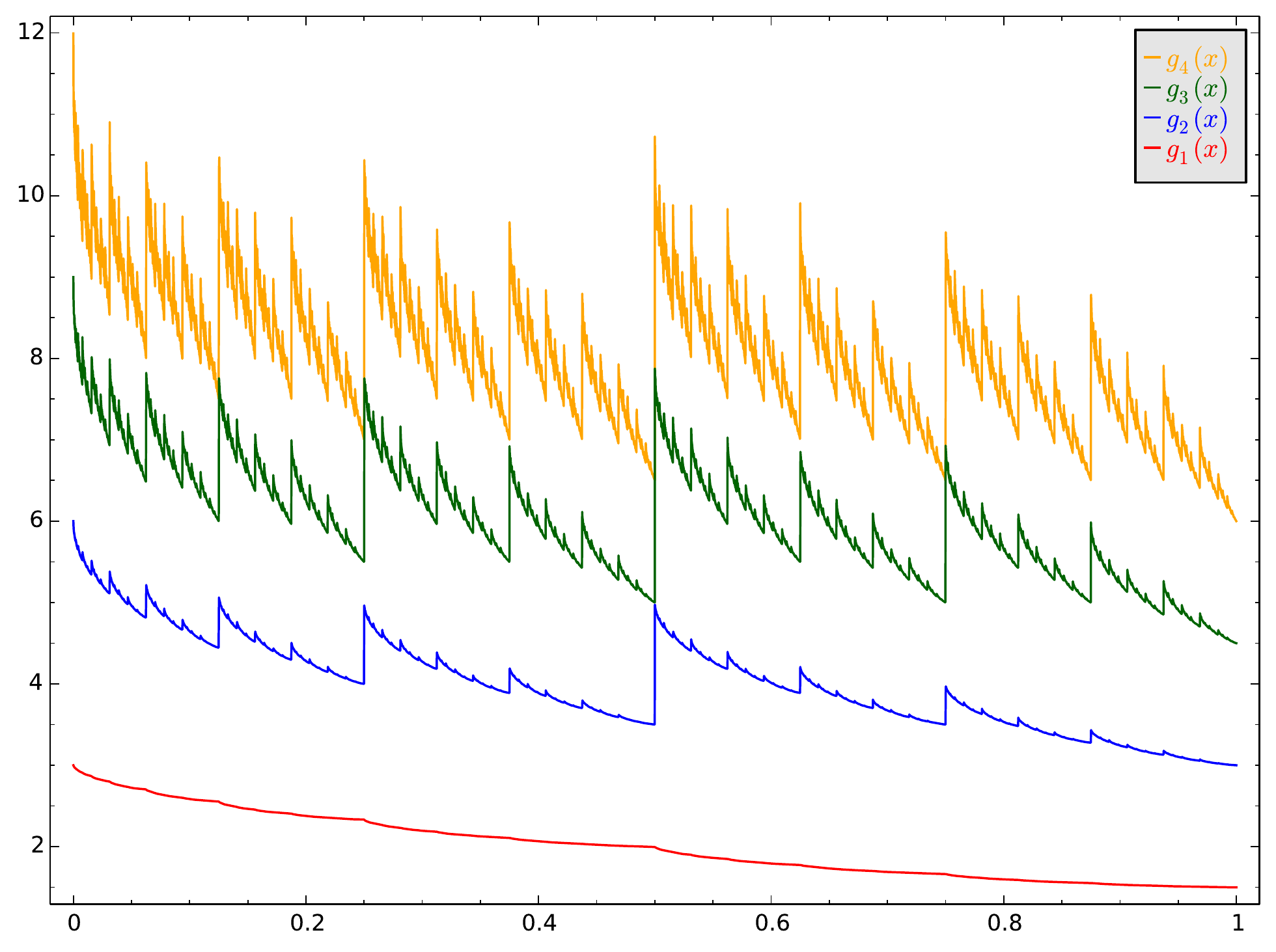}
\caption{\label{fig:g_k_binary_critical}}
\end{subfigure}
\hfill
\begin{subfigure}[b]{0.49\linewidth}
\includegraphics[width=\linewidth]{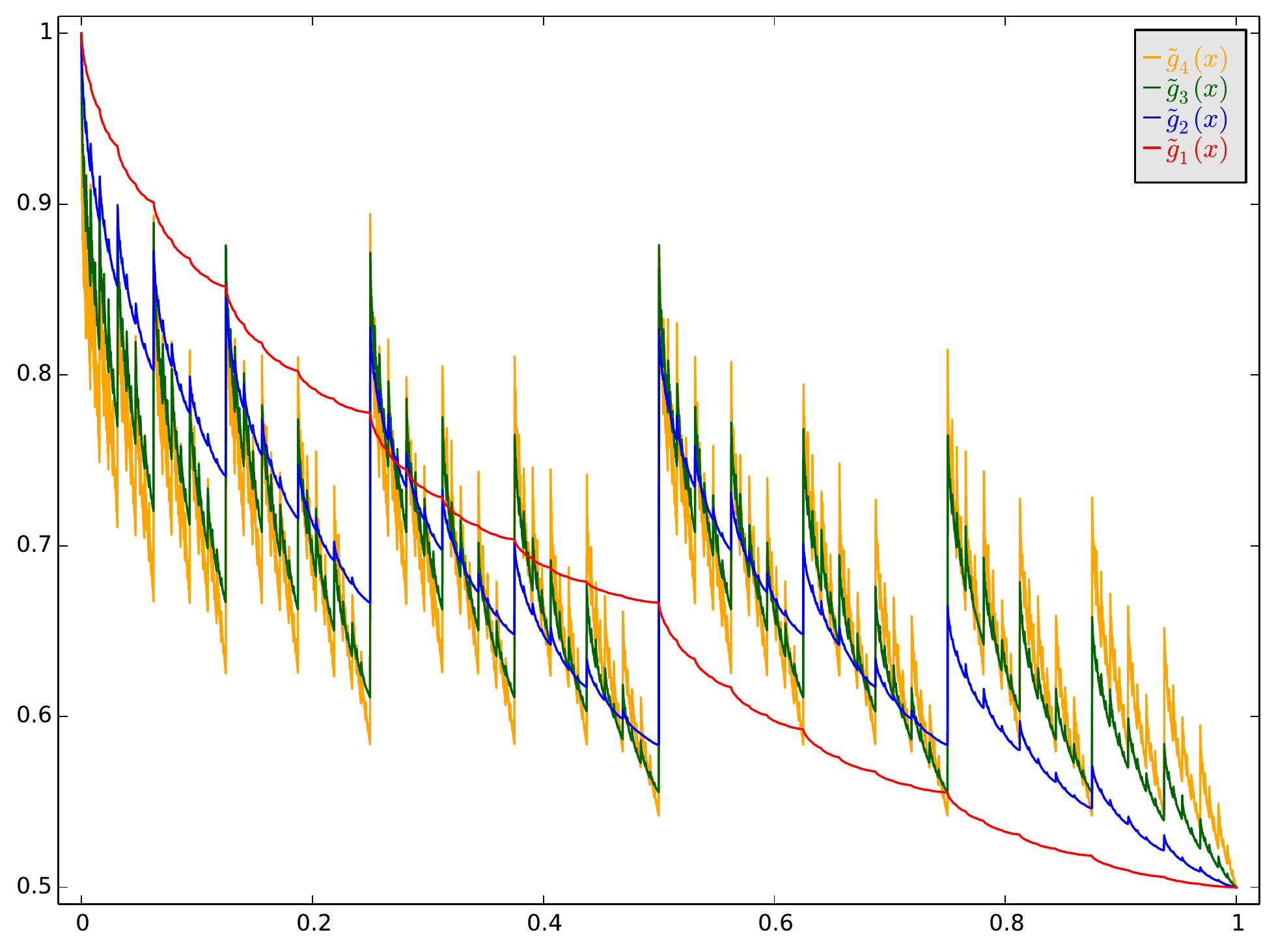}
\caption{\label{fig:g_k_binary_critical_normalized}}
\end{subfigure} \\
\begin{subfigure}[b]{0.49\linewidth}
\includegraphics[width=\linewidth]{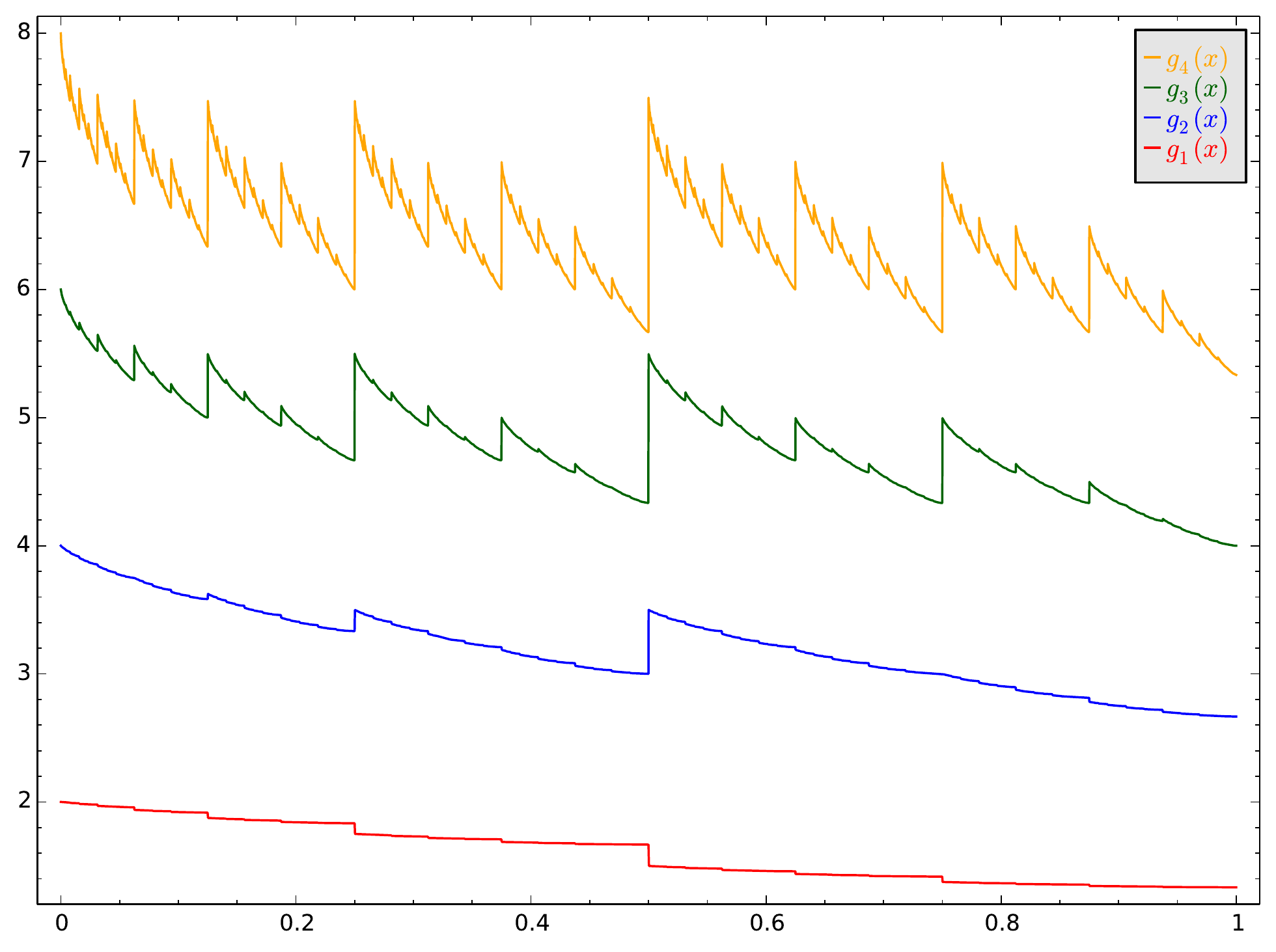}
\caption{\label{fig:g_k_binary_subcritical}}
\end{subfigure}
\hfill
\begin{subfigure}[b]{0.49\linewidth}
\includegraphics[width=\linewidth]{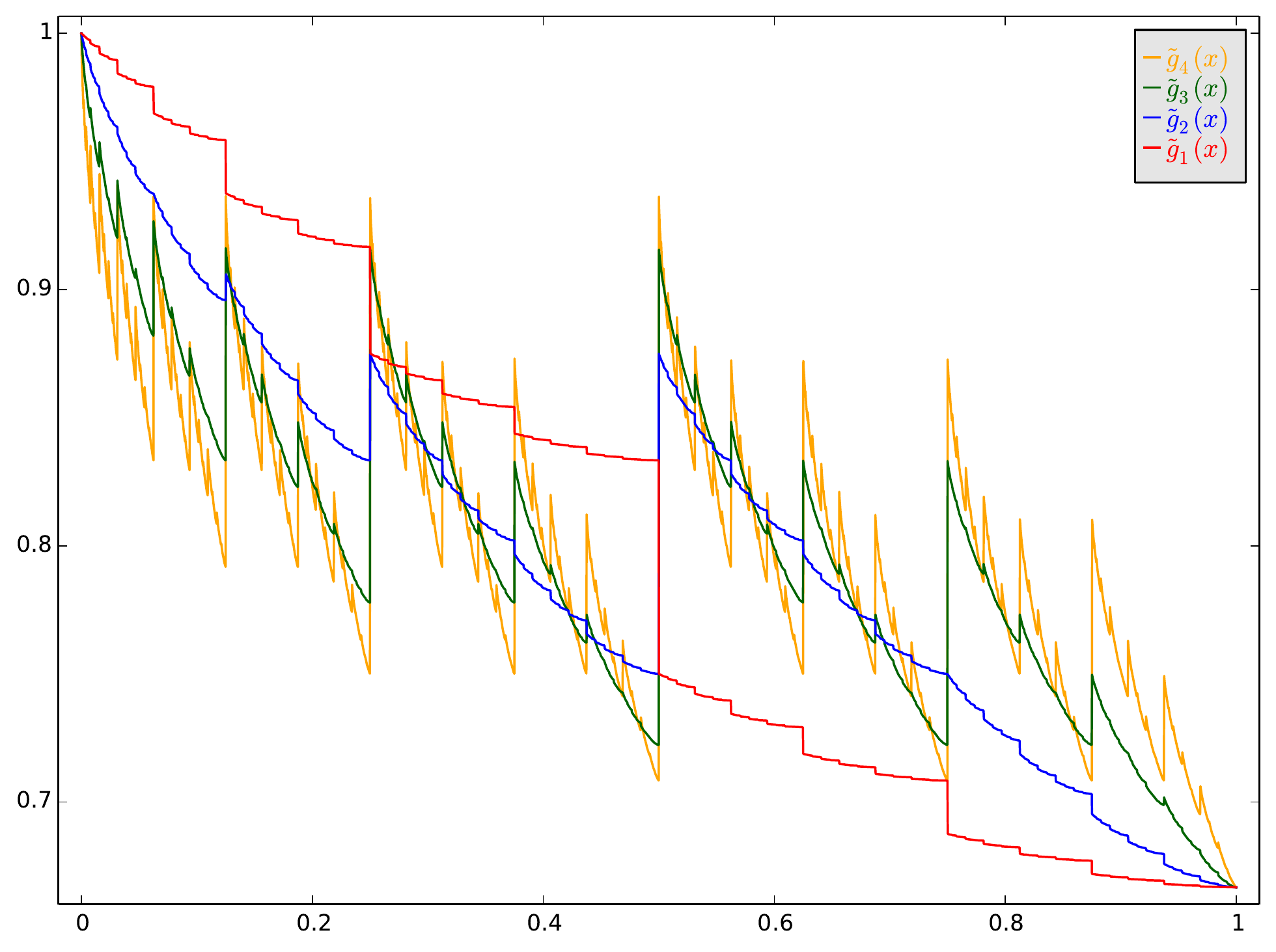}
\caption{\label{fig:g_k_binary_subcritical_normalized}}
\end{subfigure}
\caption{(a) Expected contours of the range of uniform rotor walk on $\TT_2$ up to the completion of the first $4$ excursions. (b) Normalized versions of the functions in (a).
(c) Expected contours $g_k$ for a positive recurrent rotor walk on $\TT_2$ with $r_0 = 1/4$, $r_1 = 1/4$, $r_2 = 1/2$. (d) Scaled versions of the functions in (c).
}
\end{figure}

Figure \ref{fig:g_k_binary_critical} shows the plots of the functions $g_1, \ldots, g_4$ for
the uniform rotor walk on $\TT_2$, which clearly suggests the fractal nature of these
functions.
Note that for $x\in\{0,1\}$ the shifted ray $\overleftarrow{x}^l$ always equals $x$. Hence the number
of steps the rotor walk descends into the left- and rightmost rays in the $k$-th excursion is i.i.d.
In view of the two relations $g_k(0) = k g_1(0)$ and $g_k(1) = k g_1(1)$, it makes sense to
look at the normalized versions $\tilde{g}_k(x) = \frac{g_k}{k}(x)$ (see Figure \ref{fig:g_k_binary_critical_normalized}). Figures \ref{fig:g_k_binary_subcritical} and \ref{fig:g_k_binary_subcritical_normalized} show corresponding plots for a positive recurrent rotor walk on $\TT_2$ with initial distribution given by
$r_0 = 1/4$, $r_1 = 1/4$ and $r_2 = 1/2$.

\begin{theorem}
Set $g_0\equiv 0$. Then for all $x\in[0,1]$ and all $k\geq 1$ we have the self similar
equations:
\begin{equation*}
g_k\left(\frac{x+i}{d}\right) = 1 + \big(1-p_{(d-i)}\big) g_{k-1}(x) + p_{(d-i)}g_k(x),
\end{equation*}
for $i\in\{0,\ldots,d-1\}$.
\end{theorem}
\begin{proof}
For $x\in[0,1]$ we identify $x$ with its $d$-ary expansion $x=(x_1,x_2,x_3,\ldots)$, where $x = \sum_{i=1}^\infty x_i\cdot d^{-i}$ with $x_i \in \{0,\ldots,d-1\}$.
For $i\in\{0,\ldots,d-1\}$ and $x=(x_1,x_2,x_3,\ldots)$ we have the correspondence
\begin{equation*}
\frac{x+i}{d} = (i,x_1,x_2,x_3,\ldots)
\end{equation*}
and thus $\overleftarrow{\left(\frac{x+i}{d}\right)}^l = \overleftarrow{x}^{l-1}$ for all $l\geq 1$.
It follows that $\Pb\big[f_{\R_1}\left(\frac{x+i}{d}\right) = 1\big] = q_{(d-i)}$ and for all $m\geq 2$, as a consequence of \eqref{eq:fR1}
\begin{align*}
\Pb\left[f_{\R_1}\left(\frac{x+i}{d}\right) = m\right] &= p_{(d-i)}\left( \prod_{i=2}^{m-1}p_{(d-x_i)}\right)q_{(d-x_m)} \\
&= p_{(d-i)} \Pb\big[f_{\R_1}(x)=m-1\big].
\end{align*}
We first look at the case $k=1$
\begin{align*}
g_1\left(\frac{x+i}{d}\right) &= \Pb\left[f_{\R_1}\left(\frac{x+i}{d}\right) = 1\right] + \sum_{m\geq 2} m \Pb\left[f_{\R_1}\left(\frac{x+i}{d}\right) = m\right] \\
&=q_{(d-i)} + \sum_{m\geq 2} m p_{(d-i)}\Pb\big[f_{\R_1}(x)=m-1\big] 
\end{align*}
which in turn equals 
\begin{align*}
&=q_{(d-i)} +p_{(d-i)}\sum_{m\geq 1} (m+1)\Pb\big[f_{\R_1}(x)=m\big] \\
&=q_{(d-i)} +p_{(d-i)}\sum_{m\geq 1} m\Pb\big[f_{\R_1}(x)=m\big] + p_{(d-i)}\sum_{m\geq 1} \Pb\big[f_{\R_1}(x)=m\big] \\
&= q_{(d-i)} +p_{(d-i)} + p_{(d-i)}g_1(x)=1 + p_{(d-i)}g_1(x).
\end{align*}
Since $g_0(x) = 0$ by definition the case $k=1$ follows.
We now look at the case $k\geq 2$. By the convolution formula \eqref{eq:f_convolution} we have
\begin{align*}
\Pb\left[f_{\R_k}\left(\frac{x+i}{d}\right) = m\right] & = \sum_{l=1}^{m-1} \Pb\left[f_{\R_1}\left(\frac{x+i}{d}\right) = l \right]\cdot\Pb\left[f_{\R_{k-1}}\left(\overleftarrow{\left(\frac{x+i}{d}\right)}^l\right) = m-l\right] \\
&=\Pb\left[f_{\R_1}\left(\frac{x+i}{d}\right) = 1 \right]\cdot\Pb\left[f_{\R_{k-1}}\left(\overleftarrow{\left(\frac{x+i}{d}\right)}^1\right) = m-1\right]\\
&\phantom{=}+\sum_{l=2}^{m-1} \Pb\left[f_{\R_1}\left(\frac{x+i}{d}\right) = l \right]\cdot\Pb\left[f_{\R_{k-1}}\big(\overleftarrow{x}^{l-1}\big) = m-l\right] 
\end{align*}
which equals 
\begin{align*}
&=q_{(d-i)}\cdot\Pb\left[f_{\R_{k-1}}(x) = m-1\right] 
+p_{(d-i)}\sum_{l=2}^{m-1} \Pb\left[f_{\R_1}(x) = l-1 \right]\cdot\Pb\left[f_{\R_{k-1}}\big(\overleftarrow{x}^{l-1}\big) = m-l\right]\\
&=q_{(d-i)}\cdot\Pb\left[f_{\R_{k-1}}(x) = m-1\right] 
+p_{(d-i)}\sum_{l=1}^{m-2} \Pb\left[f_{\R_1}(x) = l \right]\cdot\Pb\left[f_{\R_{k-1}}\big(\overleftarrow{x}^l\big) = m-1-l\right] \\
&=q_{(d-i)}\cdot\Pb\left[f_{\R_{k-1}}(x) = m-1\right] +p_{(d-i)}
\Pb\left[f_{\R_k}(x) = m-1\right],
\end{align*}
where in the last step we use the convolution formula \eqref{eq:f_convolution} again
in reverse.
Thus,
\begin{align*}
g_k\left(\frac{x+i}{d}\right)  &=  \sum_{m=2}^\infty m \Pb\left[f_{\R_k}\left(\frac{x+i}{d}\right) = m\right] \\
&=q_{(d-i)} \sum_{m=2}^\infty m \Pb\left[f_{\R_{k-1}}(x) = m-1\right] +
p_{(d-i)}\sum_{m=2}^\infty m
\Pb\left[f_{\R_k}(x) = m-1\right] \\
&=q_{(d-i)} \sum_{m=1}^\infty (m+1) \Pb\left[f_{\R_{k-1}}(x) = m\right] +
p_{(d-i)}\sum_{m=1}^\infty (m+1)
\Pb\left[f_{\R_k}(x) = m\right] \\
&= q_{(d-i)} \big(g_{k-1}(x) + 1\big) + p_{(d-i)}\big(g_k(x) + 1\big) \\
&= 1 + (1-p_{(d-i)}) g_{k-1}(x) + p_{(d-i)}g_k(x),
\end{align*}
which proves the theorem in the general case.
\end{proof}
\end{appendices}

\paragraph{Some comments.}

It may be interesting to understand on which graphs, other than regular and Galton-Watson trees, does the Einstein relation \eqref{eq:einstein} hold. Other classes of trees such as periodic trees are definitely a good candidate. 

One can also look at finer estimates such as law of iterated logarithm, or central limit theorems for the range and the speed for rotor walks on regular trees.

\paragraph*{Acknowledgements}
The question of investigating the range of rotor walks on trees comes from Lionel Levine, whom we thank for inspiring conversations. We  are very grateful to the anonymous referee for a very careful reading of the manuscript, for finding a gap in one of our proofs and for suggesting an alternative proof to Theorem 1.1(i).

\bibliography{range}{}
\bibliographystyle{alpha_arxiv}

\textsc{Wilfried Huss, ADB Safegate Austria}
\texttt{husswilfried@gmail.com};

\textsc{Ecaterina Sava-Huss, Institute of Discrete Mathematics, Graz University of Technology, Austria.}
\texttt{sava-huss@tugraz.at};\\ \url{http://www.math.tugraz.at/~sava} 

\end{document}